\documentclass{article}
\usepackage{graphicx}
\usepackage[left=3cm, top=2cm, right=3cm, bottom=3cm]{geometry}
\usepackage{listings}
\usepackage{float}
\usepackage{amssymb,amsmath,mathtools}
\usepackage{enumerate}
\usepackage[numbers,square]{natbib}
\usepackage{wrapfig}
\usepackage{framed}
\usepackage{tabstackengine}
\usepackage{braket}
\usepackage[nottoc]{tocbibind}\settocbibname{References}
\usepackage{bm}
\usepackage{relsize}
\usepackage{mathrsfs}
\usepackage{setspace}
\usepackage{graphicx, overpic}

\usepackage[colorinlistoftodos]{todonotes}

\usepackage{fancyhdr}
\fancypagestyle{firststyle}
{
	
	\fancyfoot[l]{\footnotesize\emph{Date:} \today}
	\fancyfoot[c]{1}

}

\usepackage{tikz-cd} 
\tikzset{
    labl/.style={anchor=south, rotate=90, inner sep=.5mm}
}

\usepackage{enumitem}
\setlist[enumerate,1]{label={(\arabic*)}}

\usepackage[utf8]{inputenc}
\usepackage[english]{babel} 

\usepackage{amsthm}
\newtheorem{thm}{Theorem}[section]
\newtheorem{lem}[thm]{Lemma}
\newtheorem{prop}[thm]{Proposition}
\newtheorem{cor}[thm]{Corollary}
\numberwithin{equation}{section}

\theoremstyle{definition}
\newtheorem{defn}[thm]{Definition} 
\newtheorem{exmp}[thm]{Example} 
\newtheorem{remk}[thm]{Remark}
\newtheorem{notation}[thm]{Notation}

\newcommand{\bbS}{\mathbb{S}}
\newcommand{\bbX}{\mathbb{X}}
\newcommand{\bbA}{\mathbb{A}}

\newcommand{\tto}{\mathtt{o}}

\newcommand{\bfP}{\mathbf{P}}

\newcommand{\bfX}{\mathbf{X}}
\newcommand{\bfY}{\mathbf{Y}}

\newcommand{\calX}{\mathcal{X}}
\newcommand{\calL}{\mathcal{L}}

\newcommand{\Str}{\operatorname{Str}}
\newcommand{\lk}{\operatorname{lk}}

\newcommand{\Cay}{\operatorname{Cay}}

\newcommand{\Aut}{\operatorname{Aut}}
\newcommand{\Isom}{\operatorname{Isom}}

\newcommand{\wt}[1]{\widetilde{#1}}

\newcommand{\cO}{\mathcal{O}}
\newcommand{\Zv}{\mathbb{Z}_v}
\newcommand{\scrG}{\mathscr{G}}

\newenvironment{claim}[1]{\par\noindent\underline{Claim:}\space#1}{}
\newenvironment{claimproof}[1]{\par\noindent\underline{Proof:}\space#1}{\leavevmode\unskip\penalty9999 \hbox{}\nobreak\hfill\quad\hbox{$\blacksquare$}}

\newcommand{\morp}[1]{\pmb{[ #1 ]}}



\makeatletter
\newsavebox{\@brx}
\newcommand{\llangle}[1][]{\savebox{\@brx}{\(\m@th{#1\langle}\)}%
  \mathopen{\copy\@brx\kern-0.5\wd\@brx\usebox{\@brx}}}
\newcommand{\rrangle}[1][]{\savebox{\@brx}{\(\m@th{#1\rangle}\)}%
  \mathclose{\copy\@brx\kern-0.5\wd\@brx\usebox{\@brx}}}
\makeatother

\begin{document}

	\thispagestyle{firststyle}
	\begin{center}
{\LARGE\bf
Quasi-isometric rigidity for graphs of virtually free groups with two-ended edge groups}\\
\bigskip
\bigskip
{\large Sam Shepherd and Daniel J. Woodhouse}
\end{center}
\bigskip

\begin{abstract}
We study the quasi-isometric rigidity of a large family of finitely generated groups that split as graphs of groups with virtually free vertex groups and two-ended edge groups.
Let $G$ be a group that is one-ended, hyperbolic relative to virtually abelian subgroups, and has JSJ decomposition over two-ended subgroups containing only virtually free vertex groups that aren't quadratically hanging. 
Our main result is that any group quasi-isometric to $G$ is abstractly commensurable to $G$. In particular, our result applies to certain ``generic'' HNN extensions of a free group over cyclic subgroups.	
\end{abstract}

\bigskip
\tableofcontents
\bigskip
\section{Introduction}

 Theorem~\ref{thm:main}, the main result of this paper, is necessarily technical in a way that obscures how generic the positive results are. 
 To quickly give a reader a meaningful sense of what is proven in this paper we present the following ``sample theorem'' that illustrates our results in an accessible manner.
 
 \begin{thm} \label{thm:sample}
  Let $G$ be a cyclic HNN extension or amalgamation of a finite rank free groups of either of the following forms:
 \[
    \mathbb{F}_n\ast_{\mathbb{Z}}  =  \langle \mathbb{F}_n  \mid tw_1t^{-1} = w_2 \rangle \; \; \textrm{ or } \; \;
   \mathbb{F}_m \ast_\mathbb{Z} \mathbb{F}_n = \langle \mathbb{F}_m, \mathbb{F}_n \mid w_1 = w_2 \rangle
  \]
   where $n,m \geq 2$ and $w_1,w_2 \in \mathbb{F}_m \cup \mathbb{F}_n$ are suitably random/generic elements that are not proper powers.
   If a finitely generated group $G'$ is quasi-isometric to $G$, then $G'$ is abstractly commensurable to $G$.
 \end{thm}

 For the HNN extension, if $g^{-1}w_1g\in\{w_2,w_2^{-1}\}$ for some $g\in G$, then $G$ is hyperbolic relative to $\langle w_1,gt\rangle$ -- which is isomorphic to either $\mathbb{Z}^2$ or the Klein bottle group (and the latter has an index two $\mathbb{Z}^2$ subgroup). Otherwise $G$ is hyperbolic; and the amalgamation is always hyperbolic.
 When we say that $w_1$ and $w_2$ are suitably random, what we really mean is that the induced line patterns on the vertex groups $\mathbb{F}_n$ and $\mathbb{F}_m$ are rigid, see Section~\ref{Riglinesec} for more about rigid line patterns, and Remark~\ref{rem:random} for a simple sufficient condition (which justifies the use of the word random). Theorem \ref{thm:sample} follows from our more general Theorem \ref{thm:main} and Example \ref{exmp:Cbullet}.
 
In his seminal essay, Gromov introduced a program for understanding finitely generated groups up to quasi-isometry~\cite{Gromov87}. 
If $G$ is a finitely presented group with finite generating set $S$, then the induced path metric on the associated Cayley graph $\Cay(G,S)$ is the \emph{word metric} on $G$.
Different finite generating sets give distinct word metrics, but up to the equivalence relation given by \emph{quasi-isometry} these metrics are all equivalent (see Section \ref{sec:quasi-isometries} for the definition of quasi-isometry). The set of quasi-isometries of $G$ up to Hausdorff equivalence forms the \emph{quasi-isometry group} of $G$, denoted $\mathscr{G} := \mathcal{QI}(G)$ (see \cite[Remark 5.1.12]{Loeh17} for a discussion of why we consider equivalence classes).
Note that a quasi-isometry $\psi: G \rightarrow G'$ induces an isomorphism of the quasi-isometry groups given by $[f]\mapsto[\psi f \psi^{-1}]$ (where $\psi^{-1}$ is any quasi-inverse to $\psi$).

This program is usually formulated in terms of studying the \emph{quasi-isometric rigidity} of groups, although this may mean a number of things.
Recall that we say two groups $G$ and $G'$ are \emph{virtually isomorphic} if there exists finite index subgroups $H \leqslant G$ and $H' \leqslant G'$ and finite, normal subgroups $F \trianglelefteq H$ and $F' \trianglelefteq H'$ such that $H / F$ is isomorphic to $H'/F'$
(we note that if both $G$ and $G'$ are residually finite this is equivalent to being abstractly commensurable).
 A finitely generated group $G$ is said to be \emph{quasi-isometrically rigid} if any other finitely generated group $G'$ that is quasi-isometric to $G$ is virtually isomorphic to $G$.
Rigidity in this form holds for finite groups, two-ended groups, abelian groups \cite{Pansu83,Drutu09}, free groups \cite{Dunwoody85, DrutuKapovich18}, cocompact Fuchsian groups~\cite{Tukia88, Gabai92, CassonJungreis94}, uniform lattices in right angled Fuchsian buildings (or Bourdon buildings) \cite{BourdonPajot00, Haglund06}, word hyperbolic surface-by-free groups~\cite{FarbMosher02}, non-uniform lattices in semisimple lie groups~\cite{Schwartz95, Eskin98}, mapping class groups~\cite{Behrstock06, BehrstockKleinerMinskyMosher12}, and many RAAGS~\cite{Huang18}.

A class of groups $\mathscr{C}$ is \emph{quasi-isometrically rigid} if any finitely generated group $G$ that is quasi-isometric to some  $G' \in \mathscr{C}$ is virtually isomorphic to a group in $\mathscr{C}$.
The classes of nilpotent groups, closed surface groups, closed $3$-manifold groups, finitely presented groups, hyperbolic groups, amenable groups, closed hyperbolic $n$-manifold groups, one-ended groups, and groups with solvable word problem are all known to be quasi-isometrically rigid. See~\cite{Kapovich14} for an introduction and survey.


We note that in general, a hyperbolic group $G$ is not quasi-isometrically rigid.
The most classical examples are given by uniform lattices in rank-1 symmetric spaces.
Other counterexamples include free products of surface groups~\cite{Whyte99, PapasogluWhyte02} and simple surface amalgams~\cite{Malone10, Stark17, DaniStarkThomas18}.
We note that quasi-isometric rigidity of hyperbolic groups is only known to fail in three particular cases: 1) $G$ is quasi-isometric to a rank-1 symmetric space, 2) $G$ has infinitely many ends, 3) $G$ has \emph{quadratically hanging} vertex groups in its JSJ decomposition (see Section~\ref{sec:JSJofFPgroups}).

\subsection{The family of groups $\mathscr{C}$}

Let $\mathscr{C}$ denote the family of finitely presented groups which split as finite graphs of groups with virtually free vertex groups and two-ended edge groups.
The subfamily of torsion-free groups $\mathscr{C}_{tf} \subseteq \mathscr{C}$ has finitely generated free vertex groups and infinite cyclic edge groups. 
These are very wide families of groups containing many surface groups, Baumslag-Solitar groups, and one-relator groups.
Many of these groups will not be one-ended, or Gromov hyperbolic, or relatively hyperbolic, or residually finite, or quasi-isometrically rigid.

In~\cite{omnipotence}, Wise showed that subgroup separability of a group $G \in \mathscr{C}_{tf}$ is equivalent to the non-existence of a non-Euclidean Baumslag-Solitar subgroup, or the non-existence of $1\neq g, t \in G$ such that $tg^pt^{-1} = g^q$, where $|p| > |q|$. 
In Section~\ref{sec:virtuallyTorsionFree}, we will show that this criterion can be extended to all groups in $\mathscr{C}$, by showing that if $G \in \mathscr{C}$ is \emph{balanced}, then it is virtually torsion-free.
In~\cite{HsuWise10} Wise and Hsu showed that all subgroup separable $G \in \mathscr{C}_{tf}$ are cocompactly cubulated, and moreover in the hyperbolic case are virtually special.
We show that all subgroup separable $G \in \mathscr{C}_{tf}$ are virtually special, and generalise further to the case with torsion:

\begin{thm} \label{thm:separableBalancedRelHyp}
 Let $G \in \mathscr{C}$. The following are equivalent.
 \begin{enumerate}
  \item \label{item:separable} $G$ is subgroup separable,
  \item \label{item:balanced} $G$ is balanced (in the sense of Definition~\ref{balanced}, and with respect to any finite graph of groups decomposition of $G$ with virtually free vertex groups and two-ended edge groups),
  \item \label{item:relHyperbolic} $G$ is hyperbolic relative to peripheral subgroups that are virtually $\mathbb{Z}\times\mathbb{F}_n$ ($n\geq0$).
  \item \label{item:special} $G$ is virtually the fundamental group of a finite special cube complex.
 \end{enumerate}
\end{thm}

As a consequence of Theorem~\ref{thm:separableBalancedRelHyp} we can deduce that subgroup separability is an invariant up to quasi-isometry within $\mathscr{C}$ (an application of~\cite[Theorem 1.6]{DrutuSapir05} to \ref{item:relHyperbolic}, and of~\cite{Margolis21} the quasi-isometric rigidity of $ \mathbb{Z}\times\mathbb{F}_n$).

Cashen-Macura describe in \cite{LinePat} how a finite collection of cyclic subgroups of a free group $F$ define a line pattern $\mathcal{L}$ on $F$, and they call such a line pattern \emph{rigid} if the group of quasi-isometries of $F$ that respect $\mathcal{L}$ acts by isometries on some model space for $F$. We give precise definitions in Section \ref{Riglinesec}.
We will see in Section \ref{Riglinesec} that a non-abelian free vertex group in a JSJ decomposition of a one-ended  group is either \emph{quadratically hanging} or the line pattern induced from the incident edge groups and their cosets is rigid.
 The aim of this paper is to start from the rigidity of line patterns in free groups to arrive at the quasi-isometric rigidity of graphs of free groups.

In this paper we will restrict to the subfamily $\mathscr{C}^\bullet \subseteq \mathscr{C}$ consisting of one-ended, subgroup separable groups that have JSJ decompositions over two-ended subgroups containing only virtually free vertex groups and no quadratically hanging (QH) vertex groups.
We refer to Section~\ref{sec:JSJofFPgroups} for background on JSJ decompositions. It follows from Theorem \ref{thm:QIpreserveJSJ}, and the quasi-isometric invariance of subgroup separability mentioned above, that $\mathscr{C}^\bullet$ is a quasi-isometrically rigid class of groups.
If $G \in \mathscr{C}^\bullet$ is hyperbolic relative to $\mathbb{Z}^2$-peripheral subgroups, we prove that $G$ is quasi-isometrically rigid.

\begin{thm} \label{thm:main}
 Let $G$ be a one-ended group, with JSJ decomposition over two-ended subgroups containing only virtually free vertex groups and no QH vertex groups.
 If $G$ is hyperbolic relative to virtually abelian subgroups, then any group quasi-isometric to $G$ is abstractly commensurable to $G$. 
\end{thm}

\begin{remk}
If $T$ is a JSJ tree for $G$, then being hyperbolic relative to virtually abelian subgroups is equivalent to the stabilisers of the cylinders in $T$ being virtually abelian (in fact virtually $\mathbb{Z}$ or $\mathbb{Z}^2$) -- see Proposition \ref{prop:hyprelcylinders} and Remark \ref{rem:peripheralIsCylindrical}.
\end{remk}

Conversely, in Section~\ref{sec:counterExample} we give $G \in \mathscr{C}^\bullet$ that is not quasi-isometrically rigid.
Outside of $\mathscr{C}^{\bullet}$ there are many examples of groups known not to be quasi-isometrically rigid.
Although solvable Baumslag-Solitar groups are quasi-isometrically rigid~\cite{FarbMosher98, FarbMosher99}, ``higher'' Baumslag-Solitar groups are not quasi-isometrically rigid~\cite{Whyte01}, so subgroup separability is a natural assumption to make for our class $\mathscr{C}^\bullet$ in light of Wise's results.
Quasi-isometries between infinite-ended groups are very flexible as shown by the work of Papasoglu and Whyte \cite{PapasogluWhyte02}, so infinite-ended groups are typically not quasi-isometrically rigid; indeed combining this work with \cite{Whyte99} shows that free products of surface groups are not quasi-isometrically rigid. So one-endedness is also a natural assumption for our class $\mathscr{C}^\bullet$.
Finally, simple surface amalgams are shown to not be quasi-isometrically rigid in ~\cite{Malone10, Stark17, DaniStarkThomas18}, so it is necessary to exclude QH vertex groups in Theorem \ref{thm:main}.

Closely related to Theorem~\ref{thm:main} are the recent results of Taam-Touikan~\cite{TaamTouikan19}.
They prove the corresponding result in the hyperbolic setting with rigid vertex groups that are hyperbolic (closed) surface groups instead of rigid free vertex groups.
So the analogue of Theorem~\ref{thm:sample} for their work would replace the free groups with closed hyperbolic surface groups.
They conjecture our main theorem in the hyperbolic case, and have further speculations as to the extent of the rigidity of hyperbolic groups.
The main obstruction to proving their conjecture was the inability to prove Leighton's theorem for graphs with fins, subsequently proven by the second author~\cite{FinLeighton}, and reformulated in this paper for our own purposes.
We note that Taam and Touikan's strategy differs from our own and we do not recover their result, although our list of ingredients is near identical: subgroup separability, rigidity of line patterns, considering stretch ratios (or \emph{clutching ratios} in their own terminology).
We note that there are differences between our strategies in the final stages when commensurability is deduced.
They construct a common model geometry that is a cube complex, then prove the existence of what they call \emph{flat groupings}.
On the other hand, we construct graphs of spaces and explicitly construct a common cover.

\subsection{Summary of the proof}

We now give a summary of the proof of Theorem \ref{thm:main}. Let $G\in\mathscr{C}^\bullet$ be hyperbolic relative to virtually abelian subgroups, and let $G'$ be another group quasi-isometric to $G$. As a consequence of work of \cite{Pap05}, $G$ and $G'$ both act on the same canonical \emph{tree of cylinders} $T_c$, and for each vertex $v\in VT$ the stabilisers $G_v$ and $G'_v$ are quasi-isometric. Moreover, the vertices of $T_c$ come in two types, \emph{rigid} and \emph{cylindrical}, and the rigid vertex stabilisers are exactly the rigid vertex groups from (any) JSJ decompositions of $G$ and $G'$ -- normally there would also be QH vertex stabilisers, but in our case the definition of $\mathscr{C}^\bullet$ excludes them. Trees of cylinders are discussed in Section \ref{sec:Cyl}.

By Theorem \ref{thm:separableBalancedRelHyp} and Section \ref{sec:ProofOfSeparableEquiv}, we may assume that $G$ and $G'$ are torsion-free, that their rigid vertex stabilisers are non-abelian free, and that their cylindrical vertex stabilisers are isomorphic to either $\mathbb{Z}$ or $\mathbb{Z}^2$. Each rigid vertex stabiliser acts freely cocompactly on a tree with incident edge stabilisers inducing what we call a \emph{line pattern}. Work of Cashen \cite{Cash}, Cashen and Macura \cite{LinePat}, and Hagen and Touikan \cite{Panel} shows that we can choose this tree with line pattern to be a \emph{rigid model space}, meaning that any quasi-isometry of the tree that respects the line pattern is at bounded distance from an isometry. In particular, for a rigid vertex $v\in VT_c$ the stabilisers $G_v$ and $G'_v$ both act isometrically on the same rigid model space. Rigid line patterns are discussed in Sections \ref{Riglinesec} and \ref{GammaJSJ}.

A line pattern on a tree can be encoded by attaching a fin to each line. We do this for the rigid model spaces above to obtain a tree of trees with fins $\bfY$, one tree with fins for each rigid vertex in $T_c$, and we have an isometric action of the quasi-isometry group $\mathscr{G}=\mathcal{QI}(G)$ on $\bfY$. This is done in Section \ref{sec:treeoftrees}. Viewing $G,G'\leqslant\mathscr{G}$, we have actions of $G$ and $G'$ on $\bfY$, and the quotients can be used to define graphs of spaces $\calX$ and $\calX'$ for $G$ and $G'$ (adding suitable circles and tori for the cylindrical vertex groups), with the property that any pair of vertex spaces in $\calX$ and $\calX'$ corresponding to the same vertex in $T_c$ have a common universal cover. The rigid vertex spaces will be \emph{graphs with fins}, and the fins correspond to the edge spaces. This is all done in Section \ref{sec:graphsofspacesforGG'}.

The goal is then to construct a common finite cover $\widehat{\calX}$ of $\calX$ and $\calX'$, thus proving that $G$ and $G'$ are commensurable. Leighton's Theorem for graphs with fins was proven by the second author in \cite{FinLeighton}, which allows us to build common finite covers of the graphs with fins that appear as the rigid vertex spaces of $\calX$ and $\calX'$. The incident edge space structure on the cylindrical vertex spaces is controlled by what we call \emph{cylinder numbers}; in Section \ref{sec:CylNumbers} we arrange for the cylinder numbers in $G$ and $G'$ to be equal, an argument combining subgroup separability with some elementary coarse geometry. These common finite covers of vertex spaces in $\calX$ and $\calX'$ will form the vertex spaces of $\widehat{\calX}$ -- details given in Section \ref{sec:vertexedgecovers} -- so it remains to glue them together.

If we want to glue together a pair of rigid vertex spaces in $\widehat{\calX}$ along a pair of fins, then we need these fins to admit covering maps to fins in $\calX$ of the same degree, and likewise for covering maps to fins in $\calX'$. We thus need the ratio of the lengths of the two fins in $\calX$ to equal the ratio of the lengths of the two fins in $\calX'$. 
Fortunately these \emph{stretch ratios} and were considered by Cashen and Martin in \cite{CashMar} and were shown to be equal; this is essentially a consequence of the trees with fins upstairs in $\bfY$ being rigid model spaces -- we give a self contained explanation of this in Section \ref{sec:stretchratio}. The equality of these ratios is still not enough however to be able to glue together the fins in $\widehat{\calX}$, as they might have different lengths; we fix this by passing to further finite covers of the vertex spaces in $\widehat{\calX}$ and applying the omnipotence of free groups due to Wise \cite{omnipotence}.

The arguments so far allow us to glue together pairs of vertex spaces in $\widehat{\calX}$, but this does not guarantee that we can glue all of them together -- indeed an attempt to do so might leave unglued edge spaces that don't match up. What we need to do is take a suitable number of copies of each vertex space in $\widehat{\calX}$ so that everything can be glued up -- this reduces to solving a set of \emph{Gluing Equations} as we describe in Section \ref{sec:LocaltoGlobal}. Solving these Gluing Equations is arguably the crux of the whole proof. The key strategy is to \emph{colour} the fins in the trees with fins that make up $\bfY$ according to their $\mathscr{G}$-orbits (in fact we colour \emph{oriented fins}, but we will ignore this distinction for the purpose of this summary). These colours then descend to the vertex spaces in $\calX$ and $\calX'$, and we require the covering maps from the vertex spaces in $\widehat{\calX}$ to the vertex spaces in $\calX$ and $\calX'$ to respect colours. Furthermore, for a pair of rigid vertex spaces $\calX_u$ and $\calX'_{u'}$ covered by a vertex space in $\widehat{\calX}$, we require each pair of fins in $\calX_u$ and $\calX'_{u'}$ of the same colour to be covered by fins upstairs, and for the total length of these fins upstairs to be in proportion to the product of the lengths of the pair of fins downstairs. This symmetry property of the rigid vertex spaces in $\widehat{\calX}$, combined with notions of \emph{density} that measure the relative abundance of different colours of fins, allows us to solve the Gluing Equations -- this is also done in Section \ref{sec:LocaltoGlobal}. The existence of these symmetrically coloured common covers of graphs with coloured fins is proven in Section \ref{sec:LeightonAgain}. As it turns out, the common cover constructed by the second author in \cite{FinLeighton} already has this symmetry property, essentially because it was built from a canonical collection of pieces (polyhedral pairs) determined by the Haar measure of the appropriate group acting on the universal cover -- the extra work in this paper is converting the symmetry of these pieces into symmetry of the fins.

\subsection{Outline of the paper}

In Section~\ref{sec:preliminaries} we establish the required background on Bass-Serre theory, JSJ theory, and rigid line patterns in free groups.
In Section~\ref{sec:ProofOfSeparableEquiv} we prove Theorem~\ref{thm:separableBalancedRelHyp} by extending previous results in the torsion-free case.
In Section~\ref{sec:LeightonAgain} we will revisit Leighton's Theorem for graphs with oriented coloured fins to establish certain properties of the common covers we will use.
We prove Theorem \ref{thm:main} in Sections \ref{sec:GraphOfSpaces} and \ref{sec:CommonFiniteCover}.
In Section~\ref{sec:GraphOfSpaces} we construct graphs of spaces for $G$ and $G'$, where $G$ is quasi-isometric to $G'$, noting that certain geometric invariants -- the stretch ratios of the edge groups and cylinder numbers for $\mathbb{Z}^2$-peripheral subgroups -- match for both graphs of spaces.
In Section~\ref{sec:CommonFiniteCover} we explicitly construct a common finite cover by taking common finite covers of the vertex spaces given by our solution to Leighton's Theorem for graphs with fins, then solving a set of equations to show that we can glue them all together.
Finally, in Section~\ref{sec:counterExample} we provide a counterexample to quasi-isometric rigidity in the situation where the group is hyperbolic relative to a $\mathbb{Z}\times\mathbb{F}_n$ peripheral subgroup with $n\geq2$.

{\bf Acknowledgements:} 
We thank the referee for helpful comments and corrections, as well as Martin Bridson, Dawid Kielak and Mark Hagen.
The second author is grateful for the support of a Glasstone Research fellowship.

\bigskip
\section{Preliminaries} \label{sec:preliminaries}

\subsection{Quasi-isometries}\label{sec:quasi-isometries}

A $(Q, \epsilon)$-\emph{quasi-isometry} between two metric spaces $f: (X, d_X) \rightarrow (Y, d_Y)$ is a function that satisfies the following two conditions:
\begin{enumerate}
	\item For all $x,x' \in X$ the following inequality is satisfied:
	\[
	\frac{1}{Q} d_X(x,x') - \epsilon \leq d_Y(f(x), f(x')) \leq Q d_X(x, x') + \epsilon.
	\]
	\item For all $y \in Y$, there exists $x \in X$ such that $d_Y(f(x),y) < \epsilon$.
\end{enumerate}
Two quasi-isometries $f, h: X \rightarrow Y$ are said to be \emph{Hausdorff equivalent} if there exists some $B < \infty$ such that $d_Y(f(x), h(x)) < B$ for all $x \in X$ -- we will write this as $f\approx h$, and denote the equivalence class of $f$ by $[f]$.
If $A$ and $B$ are subsets of a metric space, we define the \emph{Hausdorff distance} between $A$ and $B$ to be 
$$d_H(A,B):=\sup \{d(a,B),d(A,b)\mid a\in A,\,b\in B\}.$$
We will also write $A\sim B$ for $d_H(A,B)<\infty$ and $A\sim_D B$ for $d_H(A,B)\leq D$.

\subsection{Bass-Serre theory} \label{sec:BassSerre}

A \emph{graph} $\Gamma$ is a set of \emph{vertices} $V\Gamma$ and \emph{edges} $E\Gamma$.
An edge $e$ is oriented with an \emph{initial vertex} $\iota(v)$ and a \emph{terminal vertex} $\tau(v)$.
Associated to $e$ is the edge $\bar{e}$ with reversed orientation: $\iota(e) = \tau(\bar{e})$ and $\tau(e) = \iota(\bar{e})$. We have $e\neq\bar{e}$ and $e=\bar{\bar{e}}$ for all $e\in E\Gamma$. 
For $v$ a vertex in a graph, the \emph{link} of $v$ is defined to be the following set:
\[
 \lk(v) = \{ e \mid \tau(e) = v\}
\]


We refer to~\cite{SerreTrees, ScottWall79, Bass93} for full background on Bass-Serre theory and graphs of groups.
A \emph{graph of groups} is a finite graph $\Gamma$ with the following data:
\begin{enumerate}
 \item each vertex $v \in V\Gamma$ has an associated \emph{vertex group} $G_v$
 \item each edge $e \in E\Gamma$ has an associated \emph{edge group} $G_e$ such that $G_{\bar{e}} \cong G_e$.
 \item an injective homomorphism $\zeta_e : G_e \rightarrow G_{\tau(e)}$ for each $e \in E\Gamma$.
\end{enumerate}

Associated to a graph of groups we have a \emph{graph of spaces} $X$ which is a topological space constructed from $\Gamma$ with the following data
\begin{enumerate}
 \item  a \emph{vertex space} $X_v$ for each $v \in V\Gamma$ such that $\pi_1(X_v) \cong G_v$ 
 \item an \emph{edge space}  $X_e$ for each $e \in E\Gamma$ such that $\pi_1(X_e) \cong G_e$ and $X_{\bar{e}} \cong X_e$
 \item  an inclusion map $\phi_e : X_e \rightarrow X_{\tau(v)}$ for each $e\in E\Gamma$ that induces the homomorphism $\zeta_e$ on the fundamental groups.
\end{enumerate}

Then we construct $X$ as the following quotient space:
\[
 X = \bigsqcup_{v \in V \Gamma} X_v \bigsqcup_{e \in E\Gamma} X_e \times [0,1] / \sim 
\]
where $\sim$ is the relation that identifies $X_e \times \{0\}$ with $X_{\bar{e}} \times \{0\}$ via the identification of $X_e$ with $X_{\bar{e}}$, and the point $(x, 1)\in X_e \times [0,1]$ with $\phi_e(x)$.

Then $G = \pi_1(X)$ is the \emph{fundamental group} of the graph of groups.
We will refer to the pair $(G, \Gamma)$ as a \emph{graph of groups decomposition} for the group $G$ (here $\Gamma$ implicitly comes with the associated data of vertex and edge groups and edge morphisms).
The associated \emph{Bass-Serre tree} $T$ is the simplicial tree derived from the tree of spaces decomposition of the universal cover $\widetilde{X} \rightarrow X$.
For vertices and edges $v, e$ in $T$ we denote these stabilisers by $G_v$ and $G_e$.
This notation is justified by the fact that the conjugacy classes of vertex and edge stabilisers correspond precisely to the vertex and edge groups in $\Gamma$. 
Given an action of a group $G$ on a tree $T$ without edge inversions we obtain an associated graph of spaces decomposition $(G ,\Gamma)$ with the underlying graph $\Gamma$ given by the quotient $T/G$.

\subsection{JSJ decompositions of finitely presented groups} \label{sec:JSJofFPgroups}

Inspired by the JSJ decomposition of $3$-manifolds there has been extensive work generalizing such results to finitely generated groups.
We refer to~\cite{JSJ} for the most modern and comprehensive overview of this field.
All of the JSJ decompositions in this paper will be decompositions of one-ended groups over two-ended, and frequently just infinite cyclic, groups -- keep in mind that the definitions and theorems that we cite from \cite{JSJ} are far more general.

\begin{defn}(JSJ tree \cite{JSJ})\label{defn:JSJ}\\
	Let $G$ be a one-ended group. Consider trees with a minimal $G$-action, without inversion, and such that all edge stabilisers are two-ended. 
	A tree is \emph{universally elliptic} if its edge stabilisers are elliptic in every tree. 
	A tree $T$ is a \emph{JSJ tree} for $G$ if it is universally elliptic and its vertex stabilisers are elliptic in every other universally elliptic tree.
	A vertex group $G_v$, where $v$ is a vertex in a JSJ tree, is \emph{rigid} if it is elliptic in every splitting over two-ended edge groups, and \emph{flexible} otherwise.
\end{defn}

The idea of a JSJ decomposition over two-ended subgroups is that we wish to find a ``maximal'' splitting over two-ended subgroups and claim that it is essentially canonical.
Certain vertex groups in the decomposition, the \emph{rigid} vertex groups, are unambiguously going to belong to any maximal splitting as they cannot be split any further.
There is potential for ambiguity when there may be many ways to split a vertex group over two-ended groups.
Consider, for example, a vertex group coming from a hyperbolic surface with boundary, such that the incident edge groups correspond to the boundary components. The multitude of pants decompositions of the surface give many ways to split the vertex group relative to its boundary subgroups; but the edge groups in one such splitting will not be elliptic in other such splittings, hence splitting this vertex group would not give a \emph{universally elliptic} tree. This would be an example of a \emph{flexible} vertex group. Definition \ref{defn:JSJ} thus gives us a splitting that is canonical in the sense that the collection of non-two-ended vertex stabilisers is the same for every JSJ tree (easy exercise). The great success of JSJ theory is that hyperbolic surface vertex groups as described above are in some sense the \emph{only} examples of flexible vertex groups. We make this precise with the following definition.

\begin{defn}(Quadratically hanging vertex group \cite[Definition 5.13]{JSJ})\\ \label{QH}
Let $G_v$ be a vertex group for a splitting of a one-ended group $G$ over two-ended subgroups. We say that $G_v$ is \emph{quadratically hanging} (QH) if it surjects to a compact hyperbolic 2-orbifold group $\pi_1\Sigma$, with finite kernel, such that images of incident edge groups in $\pi_1\Sigma$ are contained in boundary subgroups. If $G$ is torsion-free, this reduces to saying that $G_v$ is the fundamental group of a compact hyperbolic surface with boundary, such that incident edge groups are contained in boundary subgroups.
\end{defn}

\begin{remk}\label{rem:Cvertexrigid}
For a one-ended group $G\in\mathscr{C}$, it follows from \cite[Theorem 6.2]{JSJ} that flexible vertex groups are always QH, thus vertex groups of $G \in \mathscr{C}^\bullet$ are always rigid. 
In fact there are only finitely many possibilities for rigid QH vertex groups \cite[5.12, 5.16 and 5.18]{JSJ}, and if $G$ is torsion-free then the pair of pants is the only possibility -- so the assumption that there are no QH vertex groups is very close to assuming that all vertex groups are rigid.
\end{remk}

\begin{lem}\label{lem:CindJSJ}
	The definition of $\mathscr{C}^\bullet$ is independent of the choice of JSJ decomposition: if $G\in\mathscr{C}^\bullet$, then the vertex stabilisers of \emph{any} JSJ tree $T$ for $G$ are virtually free and not QH.
\end{lem}
\begin{proof}
	The fact that the vertex stabilisers are not QH follows from \cite[Proposition 5]{JSJ}.
	As for the virtual freeness, $G\in\mathscr{C}$, so has some splitting over two-ended subgroups by acting on a tree $T_0$ with virtually free vertex stabilisers. By Remark \ref{rem:Cvertexrigid}, the vertex stabilisers of the JSJ tree $T$ are rigid, hence they are elliptic in $T_0$. Then each vertex stabiliser of $T$ is contained in a vertex stabiliser of $T_0$, so the former must be virtually free.
\end{proof}

\subsection{Trees of cylinders}\label{sec:Cyl}

In general there is no canonical JSJ decomposition.
Instead, we will use the tree of cylinders, and we will also show that it is preserved by quasi-isometries.

Subgroups $A, A' \leqslant G$ are \emph{commensurable} (in $G$) if $A \cap A'$ is a finite index subgroup of $A$ and $A'$.
Commensurability is an equivalence relation.
Let $G$ act on a tree $T$ with two-ended edge stabilisers. We say that two edges $e_1, e_2 \in ET$ are \emph{equivalent} if $G_{e_1}$ and $G_{e_2}$ are commensurable.
The union of all edges in an equivalence class gives a subtree (that is to say a connected subcomplex of $T$), which is called a \emph{cylinder}.
The \emph{tree of cylinders} $T_c$ is the bipartite tree with vertex set $V_0T_c \sqcup V_1T_c$, where $V_0T_c$ are the vertices of $T$ which lie in at least two cylinders, and $V_1T_c$ is the set of cylinders.
The edges of $T_c$ are of the form $(v, Y)$ where $v$ is a vertex in $T$ that lies in the cylinder $Y \subset T$.

\begin{notation}\label{not:cosets}
Let a finitely generated group $G$ act minimally on a tree $T$ without edge inversions, and let $\{v_1,...,v_n\}\subset VT$, $\{e_1,...,e_m\}\subset ET$ be orbit representatives for the vertices and edges, with $\{e_1,...,e_m\}$ closed under edge inversions. For $v\in VT$ in the same orbit as $v_i$, let $G(v)$ denote the left coset of $G_{v_i}$ consisting of $g\in G$ with $g(v_i)=v$. Similarly, define cosets $G(e)$ for $e\in ET$, noting that $G(e)=G(\bar{e})$. In the rest of this section we always pick orbit representatives for vertices and edges so that we can define the cosets $G(v)$ and $G(e)$; the choice of such representatives will not affect the results of this section, only the size of the constants contained within them.
\end{notation}

\begin{remk}\label{rem:cosetvstabiliser}
We always have $G(v)\sim G_v$ and $G(e)\sim G_e$, just not with a uniform constant (recall from Section \ref{sec:quasi-isometries} that $\sim$ denotes finite Hausdorff distance between subsets).
\end{remk}

The following theorem of Papasoglu says that quasi-isometries coarsely preserve vertex stabilisers of $JSJ$ trees.

\begin{thm}(Papasoglu \cite[Theorem 7.1]{Pap05})\\\label{thm:QIpreserveJSJ}
	Let $\psi:G\to G'$ be a quasi-isometry of finitely presented one-ended groups, with JSJ trees $T$ and $T'$. Then there exists a constant $D>0$, such that for each $v\in VT$ there exists $v'\in VT'$ with $\psi(G(v))\sim_DG'(v')$, and for each $e\in ET$ there exists $e'\in ET'$ with $\psi(G(e))\sim_D G'(e')$. Moreover, the type of vertex stabiliser is preserved, so $G_v$ is QH if and only if $G'_{v'}$ is QH.
\end{thm}

Note that one needs to use cosets rather than the vertex stabilisers themselves in order to obtain the uniform constant D (this was not stated correctly in \cite{Pap05}). In the following theorem we deduce from Papasoglu's theorem that quasi-isometries also preserve the tree of cylinders decomposition -- in fact this time we can say something even stronger than in Papasoglu's theorem, namely that a quasi-isometry induces an \emph{isomorphism} between trees of cylinders. In this sense we can think of trees of cylinders as being canonical. This theorem appears as \cite[Theorem 2.8]{CashMar}, but with the same mistake regarding vertex stabilisers versus cosets mentioned above. 
 Margolis~\cite[Theorem 2.9]{Margolis20} avoids this confusion in his statement; since we state the theorem in slightly different terms, we include a brief explanation of how we deduce our version from Margolis' version. 

\begin{thm}\label{thm:QIpreserveCyl}
	Let $\psi:G\to G'$ be a quasi-isometry of finitely presented one-ended groups, with JSJ trees $T$ and $T'$, and let $T_c$ and $T'_c$ be the corresponding trees of cylinders. Then there is a unique isomorphism $\hat{\psi}:T_c\to T'_c$ such that:
	\begin{enumerate}
		\item $\hat{\psi}(V_0 T_c)=V_0 T'_c$ and $\hat{\psi}(V_1 T_c)=V_1 T'_c$.
		\item There is a constant $K>0$, such that $\psi(G(v))\sim_KG'(\hat{\psi}(v))$ for $v\in VT_c$, and $\psi(G(e))\sim_KG'(\hat{\psi}(e))$ for $e\in ET_c$.
		\item $\psi(G_v)\sim G'_{\hat{\psi}(v)}$ for $v\in VT_c$ and $\psi(G_e)\sim G'_{\hat{\psi}(e)}$ for $e\in ET_c$. Moreover, the restrictions $\psi: G_v\to G'_{\hat{\psi}(v)}$ and $\psi:G_e\to G'_{\hat{\psi}(e)}$ are quasi-isometries with respect to the intrinsic metrics of the vertex and edge stabilisers.
	\end{enumerate}  
\end{thm}
\begin{proof}
	\cite[Theorem 2.9]{Margolis20} gives a unique isomorphism $\hat{\psi}:T_c\to T'_c$ that satisfies (1) and satisfies (2) for the vertex spaces in the trees of spaces corresponding to $(G,T_c)$ and $(G',T'_c)$. But \cite[Proposition 2.3]{Margolis20} allows us to transfer between the edge and vertex spaces in these trees of spaces and the corresponding cosets in $G$ and $G'$ as described in Notation \ref{not:cosets}. Hence we get a constant $K>0$, such that $\psi(G(v))\sim_KG'(\hat{\psi}(v))$ for $v\in VT_c$. Because of the tree structure of the trees of spaces, each edge space is $\sim$-equivalent to the intersection of the $R$-neighbourhoods of the adjacent vertex spaces for all $R\geq1$; so we deduce that $\psi$ coarsely maps edge spaces to edge spaces, and that $\psi(G(e))\sim_KG'(\hat{\psi}(e))$ for $e\in ET_c$ (possibly increasing $K$). Lastly, (3) follows from (2) and Remark \ref{rem:cosetvstabiliser}, and the moreover part follows from \cite[Lemma 2.1]{FarbMosher00}.
\end{proof}

\begin{cor}\label{cor:scrGactTc}
	If $G$ is a finitely presented one-ended group with JSJ tree $T$, then the group $\mathscr{G}$ of quasi-isometries of $G$ acts on the tree of cylinders $T_c$ by
	\begin{align*}
	\mathscr{G}&\to\Aut(T_c)\\
	[f]&\mapsto \hat{f}.
	\end{align*}
\end{cor}
\begin{proof}
	Properties (1)-(3) of Theorem \ref{thm:QIpreserveCyl} remain true if we perturb $f$ by a bounded amount, hence $\hat{f}$ is determined by the Hausdorff class $[f]$. If $[f_1],[f_2]\in\mathscr{G}$, then $\hat{f}_1\circ\hat{f}_2$ clearly satisfies properties (1)-(3) of Theorem \ref{thm:QIpreserveCyl} with respect to the quasi-isometry $f_1\circ f_2$, therefore $\widehat{f_1\circ f_2}=\hat{f}_1\circ\hat{f}_2$ by the uniqueness in Theorem \ref{thm:QIpreserveCyl}. 	
\end{proof}

\begin{remk}
	$G$ acts on itself by left translations, which induces a homomorphism $G\to\scrG$, and the restriction of the action of $\scrG$ on $T_c$ recovers the original action of $G$ on $T_c$.
\end{remk}

In the particular case that $G$ acts on $T$ with hyperbolic vertex stabilisers, we get that the edge stabilisers of $T_c$ are two-ended. This holds for example if $G\in\mathscr{C}^\bullet$ and $T$ is a JSJ tree.

\begin{lem}\label{lem:cyledgetwoended}
	Let $G$ act on a tree $T$ with two-ended edge stabilisers and hyperbolic vertex stabilisers. Let $H$ be the stabiliser of an edge $(v,Y)\in ET_c$. Then for every $e\in EY\subset ET$ incident at $v$, $G_e$ is a finite index subgroup of $H$ -- so in particular $H$ is two-ended. $H$ is also equal to its normaliser in $G_v$.
\end{lem}
\begin{proof}
	$H$ consists of those elements $h\in G_v$ such that $h(e)\in EY$, or equivalently that $G_e$ is commensurable to $hG_eh^{-1}$ in $G$, or equivalently that $G_e\sim hG_e$. Since $G_e$ is two-ended, the cosets $hG_e$ are all quasi-geodesics with uniform constants, so they must all be at uniform Hausdorff distance from $G_e$ by the Morse Lemma. This implies that there are only finitely many cosets $hG_e$ with $h\in H$, so $G_e$ has finite index in $H$. If $g\in G_v$ normalises $H$, then $G_e$ and $gG_eg^{-1}$ will both be finite index subgroups of $H$, so they will be commensurable and $g\in H$; hence $H$ is equal to its normaliser in $G_v$.
\end{proof}

\subsection{Rigid line patterns}\label{Riglinesec}

Given a vertex group of $G\in\mathscr{C}^\bullet$, we need to understand the structure of its incident edge groups from a coarse geometry perspective. For this we review the notion of rigid line pattern.

\begin{defn}(Line pattern)\\
	A \emph{line pattern} $\mathcal{L}$ on a metric space $X$ is a collection of bi-infinite quasi-geodesics in distinct $\sim$-equivalence classes. If $(X,\mathcal{L}_X)$ and $(Y,\mathcal{L}_Y)$ are spaces with line patterns, then we say that a quasi-isometry $f:X\to Y$ \emph{respects line patterns} if there is an associated bijection $f_*:\mathcal{L}_X\to\mathcal{L}_Y$ such that $ f (l)\sim f _*(l)$ for all $l\in\mathcal{L}_X$. In this case we write $ f :(X,\mathcal{L}_X)\to(Y,\mathcal{L}_Y)$. Observe that a composition of quasi-isometries respecting line patterns is itself a quasi-isometry respecting line patterns.
\end{defn}

\begin{defn}(Free group with line pattern, \cite{LinePat})\label{linepatdef}\\
	Consider a finitely generated free group $F$ of rank greater than one.
	Let  $\mathcal{H}$ be finite collection of cyclic subgroups of $F$.
	The \emph{line pattern} $\mathcal{L}=\mathcal{L}_{\mathcal{H}}$ generated by $\mathcal{H}$ is the collection of  quasi-geodesics corresponding to left cosets of the subgroups in $\mathcal{H}$.
	Note that the $(F, \mathcal{H})$ depends on a choice of finite generating sets, but all such choices are equivalent up to quasi-isometry respecting line patterns.
\end{defn}

\begin{remk}
	In \cite{LinePat}, all line patterns came from free groups as in Definition \ref{linepatdef}, and a quasi-isometry respecting line patterns was required to have $ f _*(l)$ at \emph{uniformly} bounded Hausdorff distance from $ f (l)$ for given free bases of the free groups. But this is equivalent to our definition because cosets of a cyclic subgroup will be uniform quasi-geodesics, and in a tree any uniform quasi-geodesic is at uniform Hausdorff distance from a unique geodesic. 
\end{remk}

\begin{defn}(Vertex group with induced line pattern)\label{vertexlinepat}\\
	If the graph of groups $(G,\Gamma)$ contains a non-abelian free vertex group $G_u$ whose incident edge groups are all cyclic, then the collection of $G_u$-conjugates of incident edge groups forms a line pattern $\mathcal{L}_u$ for $G_u$. 
	If $u$ lifts to a vertex $\tilde{u}$ in the Bass-Serre covering tree $T$ for $(G, \Gamma)$, then $G_{\tilde{u}}$ is non-abelian free and has cyclic incident edge groups, and the collection of these incident edge groups forms a line pattern $\mathcal{L}_{\tilde{u}}$ for $\tilde{u}$ (this time the collection of incident edge groups is already closed under conjugation in $G_{\tilde{u}}$).
\end{defn}

\begin{defn}(Rigid line pattern, \cite{LinePat})\label{riglinepatdef}\\	
	If $X$ is a space with line pattern $\mathcal{L}_X$, let $\mathcal{QI}(X,\mathcal{L}_X)$ denote the group of quasi-isometries from $X$ to itself that respect the line pattern $\mathcal{L}_X$ (formally an element of $\mathcal{QI}(X,\mathcal{L}_X)$ is an $\approx$-equivalence class of quasi-isometries, but when we write $ f \in\mathcal{QI}(X,\mathcal{L}_X)$ we will mean $ f $ to be a particular choice of quasi-isometry). Similarly, let $\Isom(X,\mathcal{L}_X)$ denote the group of isometries of $X$ that respect $\mathcal{L}_X$. We say that $(X,\mathcal{L}_X)$ is a \emph{rigid model space} if the natural map $\iota: \Isom(X,\mathcal{L}_X)\to\mathcal{QI}(X,\mathcal{L}_X)$ is an isomorphism.
	
	A free group with line pattern $(F,\mathcal{L})$ is \emph{rigid} if there is a quasi-isometry $\phi:(F,\mathcal{L})\to (X,\mathcal{L}_X)$ to a rigid model space. If the group $F$ is clear, then we will simply say that $\mathcal{L}$ is \emph{rigid}.
\end{defn}

Building on the work of \cite{CashMar}, Cashen proved the following characterization of rigidity for a free group with line patterns:

\begin{thm}(Cashen \cite[Theorem 4.29]{Cash})\\ \label{thm:rigidEquivalence}
 Let $F$ be a finitely generated free group and $\mathcal{H}$ a finite set of cyclic subgroups in $F$.
 Then we have three mutually exclusive cases:
 \begin{enumerate}
 	\item $(F, \mathcal{L}_{\mathcal{H}})$ is rigid,
 	\item $F$ is the fundamental group of a hyperbolic surface with boundary, with the boundary components corresponding to the subgroups $\mathcal{H}$,
 	\item $F$ is not of type (2), and admits a non-trivial free or cyclic splitting relative to $\mathcal{H}$.
 \end{enumerate}
\end{thm}

\begin{defn}($\phi$-conjugacy action)\label{phiconj}\\
	If $\phi:(F,\mathcal{L})\to (X,\mathcal{L}_X)$ is a quasi-isometry to a rigid model space, then we get an isomorphism $\mathcal{QI}(F,\mathcal{L})\to \Isom(X,\mathcal{L}_X)$ given by $ f \mapsto\iota^{-1}(\phi f \phi^{-1})$, where $\iota: \Isom(X,\mathcal{L}_X)\to\mathcal{QI}(X,\mathcal{L}_X)$ is the isomorphism as above. We call the corresponding isometric action of $\mathcal{QI}(F,\mathcal{L})$ on $(X,\mathcal{L}_X)$ the \emph{$\phi$-conjugacy action}.
\end{defn}
\begin{remk}\label{phiconjrem} 
	The $\phi$-conjugacy action is independent of $\phi$ in the sense that if $\phi_1,\phi_2:(F,\mathcal{L})\to (X,\mathcal{L}_X)$ are two different quasi-isometries, then the isometry $\iota^{-1}(\phi_2\phi_1^{-1}):(X,\mathcal{L}_X)\to(X,\mathcal{L}_X)$ is equivariant with respect to the $\phi_1$-conjugacy action on the left hand side and the $\phi_2$-conjugacy action on the right hand side. In particular, the translation length of an element $ f \in\mathcal{QI}(F,\mathcal{L})$ with respect to the $\phi$-conjugacy action is independent of $\phi$. Sometimes we will just say the action of $\mathcal{QI}(F,\mathcal{L})$ on $(X,\mathcal{L}_X)$ if we do not wish to refer to a particular $\phi$.
\end{remk}

\begin{remk}\label{enlargelinepat}
	If $\mathcal{L}\subset\mathcal{L}'$ are two line patterns on $F$, and $\mathcal{L}$ is rigid, then $\mathcal{L}'$ must also be rigid. This is because if $\phi:(F,\mathcal{L})\to (X,\mathcal{L}_X)$ is a quasi-isometry to a rigid model space, then $\phi:(F,\mathcal{L}')\to (X,\phi(\mathcal{L}'))$ is also a quasi-isometry to a rigid model space -- as $\mathcal{L}_X\subset\phi(\mathcal{L}')$ and $\mathcal{QI}(X,\phi(\mathcal{L}'))\leqslant\mathcal{QI}(X,\mathcal{L}_X)\cong \Isom(X,\mathcal{L}_X)$.
\end{remk}
\bigskip

The main theorem we will need about rigid line patterns is the following. Part (2) is due to Cashen-Macura, and part (3) is due to Hagen-Touikan (which also relies on the construction of Cashen-Macura).

\begin{thm}(Cashen-Macura \cite[Main Theorem]{LinePat}, Hagen-Touikan \cite[Theorem C]{Panel})\\\label{rigidequiv}
	Let $(F,\mathcal{L})$ be a free group with line pattern. The following are equivalent:
	\begin{enumerate}
		\item $\mathcal{L}$ is rigid.
		\item The \emph{decomposition space} $\mathcal{D_L}$, obtained from $\partial F$ by identifying the two limit points of each line $l\in\mathcal{L}$ and taking the quotient topology, is connected, has no cut points and no cut pairs.
		\item There is a quasi-isometry $\alpha:(F,\mathcal{L})\to (Y,\mathcal{L}_Y)$ to a rigid model space, where $Y$ is a locally finite tree with no leaves and $\mathcal{L}_Y$ is a collection of bi-infinite geodesics.
	\end{enumerate}
\end{thm}

We will call $(Y,\mathcal{L}_Y)$ from Theorem \ref{rigidequiv}(3) a \emph{rigid tree} for $(F,\mathcal{L})$. Note that distinct bi-infinite geodesics in $Y$ cannot be at finite Hausdorff distance, so the $\alpha$-conjugacy action of $\mathcal{QI}(F,\mathcal{L})$ on $Y$ isometrically maps each geodesic in $\mathcal{L}_Y$ onto another geodesic in $\mathcal{L}_Y$. $F$ acts on itself by left multiplication, preserving $\mathcal{L}$, and so we can view it as a subgroup of $\mathcal{QI}(F,\mathcal{L})$. The corresponding action of $F$ on a rigid tree $Y$ satisfies the following lemma.

\begin{lem}\label{freeco}
	Let $(F,\mathcal{L})$ be a rigid line pattern with a quasi-isometry  $\alpha:(F,\mathcal{L})\to (Y,\mathcal{L}_Y)$ to a rigid tree. Then the action of $F$ on $Y$ is free and cocompact, and $\alpha$ is at bounded distance from any orbit map of $F$.
\end{lem}
\begin{proof}
	By definition of the $\alpha$-conjugacy action, for $g\in F$ the diagram of quasi-isometries 
	\begin{equation}
	\begin{tikzcd}[
	ar symbol/.style = {draw=none,"#1" description,sloped},
	isomorphic/.style = {ar symbol={\cong}},
	equals/.style = {ar symbol={=}},
	subset/.style = {ar symbol={\subset}}
	]
	F\ar{d}{g}\ar{r}{\alpha}&Y\ar{d}{g}\\
	F\ar{r}{\alpha}&Y
	\end{tikzcd}
	\end{equation}
	commutes up to bounded distance. The two $g$ maps are actually isometries, so this diagram defines two quasi-isometries $F\to F$ at bounded distance from each other, with quasi-isometry constants only depending on $\alpha$. As $F$ has Cayley graph a regular tree, one can easily deduce that the distance between these two quasi-isometries $F\to F$ also just depends on $\alpha$. This implies that $\alpha$ is at bounded distance from any orbit map of $F$. It immediately follows that the action of $F$ on $Y$ is cocompact, and it must also be free because $F$ is torsion-free.
\end{proof}

\begin{remk}\label{rem:random}
	``Random line patterns'' are rigid line patterns in the following sense: working in the Cayley graph of $F$ with respect to a given free basis $\mathcal{B}$, and taking geodesic representatives for the lines in $\mathcal{L_H}$, if some $l\in\mathcal{L_H}$ contains every reduced word of length 3 as a subsegment, then $\mathcal{L_H}$ is rigid. It follows from \cite[Corollary 5.5]{CashenManning15} that $F$ does not split freely or cyclically relative to the subgroup corresponding to this $l$, so Theorem \ref{thm:rigidEquivalence} implies that $\mathcal{L_H}$ is rigid. (Note that the only possibility of being in case (2) of Theorem \ref{thm:rigidEquivalence}, but not splitting cyclically relative to $\mathcal{H}$, is if the hyperbolic surface is a pair of pants, but then $F$ will admit a free splitting relative to each subgroup in $\mathcal{H}$ individually.) In particular, if $w\in F$ is a random word of length $n$ with respect to $\mathcal{B}$, then the probability that $\mathcal{L}_{\{\langle w\rangle\}}$ is rigid tends to 1 exponentially quickly as $n\to\infty$.
\end{remk}

\subsection{Rigid decompositions are JSJ decompositions}\label{GammaJSJ}

In this section we explore the close relation between rigid line patterns and vertex groups of $G\in\mathscr{C}_{tf}^\bullet$.

\begin{lem} \label{lem:RigidImpliesRigid}
  Let $G \in \mathscr{C}_{tf}^\bullet$ with a JSJ tree $T$.
  Then for each $u \in V_0T_c$ the group $G_u$ is a non-abelian free group and the induced line pattern $(G_u,\mathcal{L}_u)$ is rigid.
\end{lem}

\begin{proof}
	The line pattern $(G_u, \mathcal{L}_u)$ must be in one of the three cases of Theorem~\ref{thm:rigidEquivalence}. We cannot be in case (2) because the splitting of $G$ has no QH vertex groups. $G_u$ cannot split freely relative to its incident edge groups because $G$ is one-ended. $G_u$ cannot admit a cyclic splitting relative to its incident edge groups by Remark \ref{rem:Cvertexrigid}, so case (3) can't happen either. Therefore we must be in case (1), which means that $(G_u,\mathcal{L}_u)$ is rigid.
\end{proof}

We know from Section \ref{sec:Cyl} that the group $\scrG$ of quasi-isometries $G\to G$ acts on the tree of cylinders $T_c$, and that each quasi-isometry restricts to maps between the vertex groups in $G$, we record here that these maps also respect the line patterns.

\begin{lem} \label{lem:restrictingToRigidVertices}
	Let $G \in \mathscr{C}_{tf}^\bullet$ and $u \in V_0 T_c$. 
	Then $[f] \in \mathscr{G}$ induces a $\approx$-class of quasi-isometries $[f]_u : (G_u,\calL_u) \rightarrow (G_{\hat{f}(u)},\calL_{\hat{f}(u)})$ that respect line patterns.
\end{lem}

\begin{proof}
	This follows immediately from Theorem~\ref{thm:QIpreserveCyl}(3).	
\end{proof}

\bigskip

We also have a converse to Lemma \ref{lem:RigidImpliesRigid} as follows.

\begin{prop}\label{prop:riglineisJSJ}
	Let $G$ be a finitely generated group that splits over two-ended subgroups by acting minimally on a tree $T$. 
	Suppose that the vertex stabilisers are all either 
	\begin{enumerate} 
	 \item virtually non-abelian free with incident edge stabilisers inducing rigid line patterns, 
	 \item virtually infinite cyclic, 
	\end{enumerate}
    with at least one vertex stabiliser of the first type.
	Then $G$ is one ended and $T$ is a JSJ tree for $G$ with no QH vertex groups.
	
\end{prop}

\begin{proof}
First suppose that $G$ is not one-ended.
Let $T_{DS}$ be a $G$-tree with finite edge stabilisers and one ended vertex stabilisers (the \emph{Dunwoody-Stallings decomposition}). Each vertex stabiliser $G_v$ for $T$ acts on a minimal subtree $S_v\subset T_{DS}$. If $u,v\in VT$ are the endpoints of an edge $e$, then $S_u$ and $S_v$ must intersect, else $G_e$ would stabilise the arc between them, contradicting the finiteness of edge stabilisers in $T_{DS}$. The union of all $S_v$ is then a $G$-invariant subtree of $T_{DS}$, and so by minimality it is the whole of $T_{DS}$. In particular, at least one of the $S_v$ is non-trivial. 

If all edge stabilisers for $T$ are elliptic in $T_{DS}$, then the type (2) vertex stabilisers are also elliptic in $T_{DS}$, and so there must be a type (1) vertex stabiliser $G_v$ that acts non-trivially on $S_v$ relative to its incident edge stabilisers, and the same is true of any finite index subgroup of $G_v$. But $G_v$ has a finite index subgroup with incident edge stabilisers inducing a rigid line pattern, contradicting Theorem \ref{thm:rigidEquivalence}. Hence at least some edge stabilisers for $T$ are not elliptic in $T_{DS}$, but such an edge stabiliser $G_e$ is two-ended, so must stabilise a unique axis $\ell_e\subset T_{DS}$, and moreover any finite index $\mathbb{Z}\leqslant G_e$ will act on $\ell_e$ by translations. Also note that $\ell_e\subset S_v$ for a vertex $v$ incident at $e$. We now have the following claim.\\

\begin{claim}
	There exists an edge $e_{DS}\in ET_{DS}$ and a type (1) vertex stabiliser $G_v$ such that $e_{DS}\subset\ell_e$ for a unique edge $e\in\lk(v)$. 
\end{claim}\\
\begin{claimproof}
	Suppose not. 
	Let $e_{DS}\in ET_{DS}$ be contained in at least one axis $\ell_e$. Given an axis $\ell_e$, if $G_e$ is incident at a vertex stabiliser $G_v$, if $G_v$ is type (1) then by assumption there is another $e'\in\lk(v)$ with $e_{DS}\subset\ell_{e'}$, while if $G_v$ is type (2) then all edge stabilisers incident at $G_v$ will be commensurable in $G$ and have the same axis $\ell_e$, so again there is another $e'\in\lk(v)$ with $e_{DS}\subset\ell_{e'}$. 
	Therefore, for any axis $\ell_e$ containing $e_{DS}$, there are two more edges incident at either end of $e$ whose stabilisers have axes that also contain $e_{DS}$.
	Thus $e_{DS}$ is contained in infinitely many axes $\ell_e$. 
	There are finitely many $G$-orbits of edges in $T$, so there exists $e\in ET$ with $e_{DS}\subset\ell_e$ and an infinite sequence $(g_n)$ in $G$ such that the edges $g_n(e)$ are all distinct and $e_{DS}\subset\ell_{g_n(e)}$ for all $n$. Noting that $g_n(\ell_e)=\ell_{g_n(e)}$, we can precompose the $g_n$ by elements of $G_e$ that translate along $\ell_e$ and assume that the edges $g_n(e_{DS})$ lie at bounded distance from $e_{DS}$. Passing to a subsequence of $(g_n)$, we can assume that the edges $g_n(e_{DS})$ are all at distance $d$ from $e_{DS}$ and all lie in the same component of $T_{DS}-e_{DS}$. But then there is an edge $e'_{DS}\subset\ell_e$ at distance $d$ from $e_{DS}$ such that $g_n(e'_{DS})=e_{DS}$ for all $n$, and so $e_{DS}$ has infinite stabiliser, a contradiction.
\end{claimproof}\\

Taking $e_{DS}$, $e$ and $G_v$ as from the claim, we will now convert the action of $G_v$ on $S_v$ into an action on a different tree $S$ that gives a splitting of $G_v$ over finite subgroups relative to its incident edge stabilisers, contradicting Theorem \ref{thm:rigidEquivalence} as before. $S$ will be bipartite with respect to vertex sets $VS=V_0S\sqcup V_1 S$, and is defined as follows:
\begin{itemize}
	\item $V_0S$ is the collection of components of $S_v-G_v\cdot e_{DS}$.
	\item $V_1S$ is the collection of axes $\ell_{g(e)}$ for $g\in G_v$.
	\item $U\in V_0S$ and $\ell_{g(e)}\in V_1 S$ form an edge if they intersect.
\end{itemize}
$S_v$ has a tree of spaces decomposition formed by the components $U\in V_0S$ and edges $g(e_{DS})$ for $g\in G_v$, and each edge $g(e_{DS})$ is contained in the unique axis $\ell_{g(e)}\in V_1S$, therefore $S$ is indeed a tree. The action of $G_v$ on $S_v$ induces an action on $S$. Each edge group $G_{g(e)}$ for $g\in G_v$ stabilises the axis $\ell_{g(e)}\in V_1S$, while each edge $e'\in\lk(v)-G_v\cdot e$ has axis $\ell_{e'}$ contained in some component $U\in V_0S$, and so $G_{e'}$ stabilises $U$. On the other hand, the $G_v$-stabiliser of an edge $(U,\ell_{g(e)})\in ES$ must stabilise (setwise) the two $G_v$-translates of $e_{DS}$ contained in $\ell_{g(e)}$ that touch $U$, and so this stabiliser must be finite. Therefore $S$ gives a splitting of $G_v$ over finite subgroups relative to its incident edge stabilisers, as required.
\bigskip

We now show that $T$ is a JSJ tree for $G$ with no QH vertex groups. Let $T_J$ be a JSJ tree for $G$ over two-ended subgroups. By \cite[Lemma 2.6(3)]{JSJ}, the edge stabilisers of $T$ are all elliptic in $T_J$, and hence so are the vertex stabilisers of type (2). For a type (1) vertex stabiliser $G_v$, we can apply Theorem \ref{thm:rigidEquivalence} to a finite index free subgroup of $G_v$ whose incident edge stabilisers induce a rigid line pattern, and deduce that $G_v$ is elliptic in $T_J$. Therefore each edge stabiliser of $T$ is either contained in an edge stabiliser of $T_J$, or has both its adjacent vertex stabilisers contained in the same vertex stabiliser $G^J_x$ of $T_J$. The second case can't happen, as then $G^J_x$ would be flexible, and hence QH (Remark \ref{rem:Cvertexrigid}); one can then argue that the vertex stabilisers of $T$ contained in $G^J_x$ would have line patterns coming from compact hyperbolic surfaces with boundary, contradicting rigidity of the line patterns by Theorem \ref{thm:rigidEquivalence}. We conclude that every edge stabiliser of $T$ is contained in an edge stabiliser of $T_J$, making $T$ universally elliptic. We already showed that the vertex stabilisers of $T$ are elliptic in $T_J$, hence they are elliptic in every universally elliptic tree for $G$, and so $T$ is a JSJ tree for $G$. Finally, there are no QH vertex stabilisers of $T$ by Theorem \ref{thm:rigidEquivalence}.
\end{proof}

\begin{exmp}\label{exmp:Cbullet}
Proposition \ref{prop:riglineisJSJ} allows us to construct explicit examples of groups in $\mathscr{C}^\bullet$, especially when combined with Remark \ref{rem:random}. For example if $\mathbb{F}_m$ and $\mathbb{F}_n$ are finitely generated free groups, and $1\neq w_1\in\mathbb{F}_m$, $1\neq w_2\in\mathbb{F}_n$ are not proper powers, and $w_1,w_2$ can each be represented by cyclically reduced words that contain every possible length three subword, then the following amalgam is in $\mathscr{C}^\bullet$.
$$G=\mathbb{F}_m*_{\mathbb{Z}}\mathbb{F}_n:=\langle \mathbb{F}_m,\mathbb{F}_n\mid w_1=w_2\rangle$$
The assumption that $w_1$ and $w_2$ are not proper powers ensures that $G$ is hyperbolic, and hence subgroup separable by Theorem \ref{thm:separableBalancedRelHyp}.

If instead we have $1\neq w_1,w_2\in\mathbb{F}_n$, but otherwise with the same properties, then the following HNN extension is in $\mathscr{C}^\bullet$.
$$G=\mathbb{F}_n*_{\mathbb{Z}}:=\langle \mathbb{F}_n,t\mid tw_1t^{-1}=w_2\rangle$$
If $g^{-1}w_1g\in\{w_2,w_2^{-1}\}$ for some $g\in G$, then $G$ is hyperbolic relative to $\langle w_1,gt\rangle$ -- which is isomorphic to either $\mathbb{Z}^2$ or the Klein bottle group (and the latter has an index two $\mathbb{Z}^2$ subgroup). Otherwise $G$ is hyperbolic. $G$ is subgroup separable in all of these cases by Theorem \ref{thm:separableBalancedRelHyp}.
\end{exmp}

\bigskip
\section{Balanced groups, separability, and torsion} \label{sec:ProofOfSeparableEquiv}

In this section we prove Theorem \ref{thm:separableBalancedRelHyp}. In \cite{omnipotence} Wise characterised subgroup separable groups in $\mathscr{C}_{tf}$ as being \emph{balanced}.
We generalise the notion of balanced in the obvious way to all groups in $\mathscr{C}$, and in Theorem~\ref{thm:torsionbalancedseparable} we prove that being balanced is equivalent to subgroup separability. The other implications of Theorem \ref{thm:separableBalancedRelHyp} are dealt with in Section \ref{sec:relhypvspecial}.

\begin{defn}(Separable and subgroup separable)\\
	A subgroup $H$ of a group $G$ is \emph{separable} if for any $g\in G-H$ there is a homomorphism $\rho:G\to\bar{G}$ to a finite group such that $\rho(g)\notin\rho(H)$. A group $G$ is \emph{subgroup separable} if all of its finitely generated subgroups are separable.
\end{defn}

\begin{remk}\label{rem:subgpsep}
	If $G$ is subgroup separable and $H\leqslant G$, then $H$ is subgroup separable. If $\hat{G}\leqslant G$ is finite index and $\hat{G}$ is subgroup separable, then $G$ is subgroup separable.
\end{remk}

The main way that we use subgroup separability in this paper is via the following proposition.

\begin{prop}\label{prop:sepvertexstabs}
	Let $G$ be a subgroup separable group acting on a tree $T$. Suppose $U\subset VT$ is a finite set of vertices, and for each $u\in U$ let $\dot{G}_u$ be a finite index subgroup of $G_u$. Then $G$ contains a finite index normal subgroup $\hat{G}$ such that $\hat{G}_u\leqslant \dot{G}_u$ for all $u\in U$. This also implies that $\hat{G}_{gu}=g\hat{G}_u g^{-1}\leqslant g\dot{G}_u g^{-1}\leqslant G_{gu}$ for all $u\in U$ and $g\in G$.
\end{prop}
\begin{proof}
	We know that $\dot{G}_u\leqslant G$ is separable, so for any $g\in G - \dot{G}_u$ there is a homomorphism $\rho:G\to\bar{G}$ to a finite group such that $\rho(g)\notin\rho(\dot{G}_u)$. By taking products of these homomorphisms, we can produce a homomorphism $\rho:G\to\bar{G}$ to a finite group such that $\rho(g_i)\notin\rho(\dot{G}_u)$ for $\{g_i\}$ a set of representatives for the left cosets of $\dot{G}_u$ in $G_u$ that are not equal to $\dot{G}_u$. This implies that $\ker\rho\cap G_u \leqslant\dot{G}_u$. The proposition then follows by taking products of these homomorphisms for all of the vertices in $U$, and setting $\hat{G}$ equal to the kernel.
\end{proof}

\begin{defn}(Balanced graph of groups)\label{balanced}\\
	A finite graph of groups $(G, \Gamma)$ with two-ended edge groups is \emph{balanced} if the following equation holds for any loop in $\Gamma$ given by edges $e_0,e_1,...,e_n=e_0$, where $\iota(e_i) = v_i$ and $\tau(e_i) = v_{i+1}$, 
	 and $\zeta_{e_{i-1}}(G_{e_{i-1}})$ is commensurable to $g_i \zeta_{\bar{e}_i}(G_{e_i}) g_i^{-1}$ in $G_{v_i}$ for some $g_i\in G_{v_i}$. 
	\begin{equation}\label{balanceprod}
	1 = \prod_{i=1}^n \frac{\big[g_i\zeta_{\bar{e}_i}(G_{e_i})g_i^{-1} : \zeta_{e_{i-1}}(G_{e_{i-1}})\cap g_i\zeta_{\bar{e}_i}(G_{e_i})g_i^{-1} \big]}
	{\big[ \zeta_{e_{i-1}}(G_{e_{i-1}}): \zeta_{e_{i-1}}(G_{e_{i-1}})\cap g_i\zeta_{\bar{e}_i}(G_{e_i})g_i^{-1}\big]}.
	\end{equation}	
\end{defn}

\begin{lem}\label{lem:balanceequiv}
	$(G,\Gamma)$ is balanced if and only if there is no relation $gh^pg^{-1}=h^q$ for $h$ an infinite order element of an edge group and $|p|\neq|q|$.
\end{lem}
\begin{proof}
	Let $T$ be the Bass-Serre tree corresponding to $(G,\Gamma)$. The edge loop of Definition \ref{balanced} corresponds to an edge path $e_0,e_1,...,e_n$ in $T$ such that the edge stabilisers $G_{e_i}$ are all commensurable in $G$ and there exists $g\in G$ with $g(e_0)=e_n$. The product (\ref{balanceprod}) becomes:
	\begin{equation}\label{liftbalanceprod}
	\prod_{i=1}^n \frac{\big[G_{e_i} : G_{e_i}\cap G_{e_{i-1}} \big]}
	{\big[ G_{e_{i-1}}: G_{e_i}\cap G_{e_{i-1}}\big]}=
	\frac{\big[G_{e_n} : G_{e_0}\cap G_{e_n} \big]}
	{\big[ G_{e_0}: G_{e_0}\cap G_{e_n}\big]}
	\end{equation}
	Let $h\in G_{e_0}$ be infinite order, so $ghg^{-1}\in G_{e_n}$, and $\langle h\rangle\leqslant G_{e_0}$ and $\langle ghg^{-1}\rangle\leqslant G_{e_n}$ are finite index subgroups. Suppose $\langle h\rangle\cap\langle ghg^{-1}\rangle$ is generated by $h^q=gh^pg^{-1}$. Then (\ref{liftbalanceprod}) is equal to
	\begin{align*}
\frac{\big[G_{e_n} : \langle gh^pg^{-1}\rangle \big]}
{\big[ G_{e_0}:\langle h^q\rangle\big]}&=\frac{|p|\big[G_{e_n} : \langle ghg^{-1}\rangle \big]}{|q|\big[ G_{e_0}:\langle h\rangle\big]}\\
&=\frac{|p|\big[gG_{e_0}g^{-1} : g\langle h\rangle g^{-1} \big]}{|q|\big[ G_{e_0}:\langle h\rangle\big]}\\
&=\frac{|p|}{|q|},
	\end{align*}
	thus completing the proof of the lemma.
\end{proof}
\begin{remk}\label{hypisbalanced}
	Hyperbolic groups and CAT(0) groups are always balanced as they cannot contain a relation $gh^pg^{-1}=h^q$ for $h$ an infinite order element and $|p|\neq|q|$ (see~\cite{BridsonHaefliger99}). 
\end{remk}

\begin{remk}\label{fibalanced}
	It follows from Lemma \ref{lem:balanceequiv} that, given $(\hat{G},\hat{\Gamma})\to(G,\Gamma)$ a finite cover of graphs of groups (or equivalently $\hat{G}\leqslant G$ finite index with the restricted action on the Bass-Serre tree $T$), $(\hat{G},\hat{\Gamma})$ is balanced if and only if $(G,\Gamma)$ is balanced.
\end{remk}

\begin{thm}(Wise \cite[Theorem 5.1]{omnipotence})\label{balancedseparable}\\
Suppose a finitely generated group $G$ splits as a finite graph of groups $(G,\Gamma)$, where the edge groups are cyclic and the vertex groups are free. Then $G$ is subgroup separable if and only if $(G,\Gamma)$ is balanced.
\end{thm}

We generalise Theorem~\ref{balancedseparable} to the following, which gives us the equivalence of \ref{item:separable} and \ref{item:balanced} in Theorem~\ref{thm:separableBalancedRelHyp}.

\begin{thm}\label{thm:torsionbalancedseparable}
Let $G \in \mathscr{C}$ split as a finite graph of groups $(G,\Gamma)$, where the edge groups are two-ended and the vertex groups are virtually free. Then $G$ is subgroup separable if and only if $(G,\Gamma)$ is balanced, and in this case $G$ is virtually torsion-free.
\end{thm}

\subsection{Removing torsion} \label{sec:virtuallyTorsionFree}

In this section we prove Theorem \ref{thm:torsionbalancedseparable}.
A key ingredient in the proof of Theorem~\ref{balancedseparable} is the \emph{omnipotence} of free groups.
 The omnipotence of free groups can be viewed as a special case of Wise's Malnormal Special Quotient Theorem (see ~\cite{WiseQCH, WiseRiches, AgolGrovesManning16}).
 In particular it will apply to virtually free groups.
 
 \begin{defn} \label{defn:almostMalnormal}
  Let $G$ be a group and $\mathcal{P}$ a collection of subgroups.
  The subgroups $\mathcal{P}$ are \emph{almost malnormal} if for $P, P' \in \mathcal{P}$ the intersection $P^g \cap P'$ being infinite implies that $P = P'$ and $g \in P$.
  We note that if a group is hyperbolic relative to $\mathcal{P}$, then it is an immediate consequence of Bowditch's fine graph condition for relative hyperbolicity~\cite{Bowditch12} that $\mathcal{P}$ is an almost malnormal family.
 \end{defn}

  \begin{thm}[Malnormal Special Quotient Theorem] \label{thm:MSQT}
   Let $G$ be a virtually special hyperbolic group.
   Let $\{H_1, \ldots, H_m \}$ be an almost malnormal collection of quasi-convex subgroups.
   Then there exist finite index subgroups $\dot{H}_i \trianglelefteq H_i$, such that for any further finite index subgroups $H_i' \leqslant \dot{H}_i$, the quotient  $G / \llangle H'_1, \ldots, H'_m \rrangle$ is hyperbolic and virtually special.
  \end{thm}

 \noindent The quotient $G / \llangle H'_1, \ldots, H'_m \rrangle$ is an example of a \emph{Dehn filling}.

\bigskip
The direction of Theorem \ref{thm:torsionbalancedseparable} where we assume that $G$ is subgroup separable is straightforward. Indeed if $(G,\Gamma)$ is not balanced then by Lemma \ref{lem:balanceequiv} we have a relation $gh^pg^{-1}=h^q$ for $h$ an infinite order element of an edge group and $|p|\neq|q|$. $\langle h^{|pq|}\rangle$ is separable in $G$, so there is a homomorphism $\rho:G\to \bar{G}$ to a finite group such that $\rho(h^i)\notin\rho(\langle h^{|pq|}\rangle)$ for $1\leq i<|pq|$, which implies that $\rho(h)$ has order $k|pq|$ for some integer $k$. But then $\rho(h^p)$ and $\rho(h^q)=\rho(gh^pg^{-1})$ are conjugate elements in $\bar{G}$ with distinct orders $k|q|$ and $k|p|$ respectively, a contradiction.

In the rest of this section we prove the other direction of Theorem \ref{thm:torsionbalancedseparable}, so suppose $G$ has a balanced graph of groups decomposition $(G,\Gamma)$ with virtually free vertex groups and two-ended edge groups. We will show that $G$ is virtually torsion-free, subgroup separability then follows from Remark \ref{fibalanced} and Theorem \ref{balancedseparable}. Note that some vertex groups in $(G,\Gamma)$ might be two-ended, and others infinite-ended, but this does not matter to us, as our arguments in this section will work for both.

Let $v\in V\Gamma$. We can assume that incident edge groups in $G_v$ that are commensurable up to conjugacy in $G_v$ are actually commensurable in $G_v$ (one can always modify a graph of groups to arrange this, without changing the fundamental group). Hence there exists an almost malnormal collection of maximal two-ended subgroups $\mathbb{P}_v$ in $G_v$, such that each incident subgroup $\zeta_e(G_e)$ is contained in exactly one $H\in\mathbb{P}_v$, call this subgroup $H_e$. Note that $H_{e_1}=H_{e_2}$ if and only if $\zeta_{e_1}(G_{e_1})$ and $\zeta_{e_2}(G_{e_2})$ are commensurable in $G_v$. These subgroups $H$ will also be quasi-convex in $G_v$ since $G_v$ is hyperbolic. We can thus apply Theorem \ref{thm:MSQT} to the collection $\mathbb{P}_v$ to produce finite index subgroups $\dot{H} \trianglelefteq H$ for each $H\in\mathbb{P}_v$. We may assume that the $\dot{H}$ are cyclic by passing to further finite index subgroups if necessary. We do this for each $v\in V\Gamma$.

\begin{lem}\label{lem:G'e}
	There exist finite index subgroups $G'_{\bar{e}}=G'_e\trianglelefteq G_e$ for each $e\in E\Gamma$ such that:
	\begin{enumerate}
		\item $\zeta_e(G'_e)\leqslant\dot{H}_e$ (in particular $G'_e\cong\mathbb{Z}$),
		\item if $\tau(e_1)=\tau(e_2)=v$ with $\zeta_{e_1}(G_{e_1})$ and $\zeta_{e_2}(G_{e_2})$ commensurable in $G_v$, then $\zeta_{e_1}(G'_{e_1})=\zeta_{e_2}(G'_{e_2})$,
		\item the normal subgroup $\llangle \zeta_e(G_{e}') \mid \tau(e) = v \rrangle \leqslant G_v$ is a free subgroup for each $v\in V\Gamma$,
		\item $\zeta_e$ induces an injection $G_e/G'_e\hookrightarrow G_v/\llangle \zeta_e(G_{e}') \mid \tau(e) = v \rrangle$.
	\end{enumerate}		
\end{lem}
\begin{proof}
	\begin{enumerate}
	\item This property holds provided we pick $G'_e\leqslant\zeta^{-1}_e(\dot{H}_e),\zeta^{-1}_{\bar{e}}(\dot{H}_{\bar{e}})$.
	
	\item The fact that $(G,\Gamma)$ is balanced implies there exist positive integers $K_e$ for $e\in E\Gamma$, with $K_e=K_{\bar{e}}$, such that
	\begin{equation}
	\frac{K_{e_1}}{[\zeta_{e_1}(G_{e_1}):\zeta_{e_1}(G_{e_1})\cap \zeta_{e_2}(G_{e_2})]}=\frac{K_{e_2}}{[\zeta_{e_2}(G_{e_2}):\zeta_{e_1}(G_{e_1})\cap \zeta_{e_2}(G_{e_2})]}\in\mathbb{N}
	\end{equation}
	whenever $\tau(e_1)=\tau(e_2)=v\in V\Gamma$ with $\zeta_{e_1}(G_{e_1})$ and $\zeta_{e_2}(G_{e_2})$ commensurable in $G_v$. If we choose the $G'_e$ such that $[G_e:G'_e]=NK_e$ for some fixed $N$, then
	\begin{equation}
	\frac{[\zeta_{e_1}(G'_{e_1}):\zeta_{e_1}(G'_{e_1})\cap \zeta_{e_2}(G'_{e_2})]}{[\zeta_{e_2}(G'_{e_2}):\zeta_{e_1}(G'_{e_1})\cap \zeta_{e_2}(G'_{e_2})]}=
	\frac{K_{e_2}[\zeta_{e_1}(G_{e_1}):\zeta_{e_1}(G'_{e_1})\cap \zeta_{e_2}(G'_{e_2})]}{K_{e_1}[\zeta_{e_2}(G_{e_2}):\zeta_{e_1}(G'_{e_1})\cap \zeta_{e_2}(G'_{e_2})]}=1,
	\end{equation}
	and as $\zeta_{e_1}(G'_{e_1}),\zeta_{e_2}(G'_{e_2})\leqslant\dot{H}_{e_1}=\dot{H}_{e_2}\cong\mathbb{Z}$ by (1), we deduce that $\zeta_{e_1}(G'_{e_1})=\zeta_{e_2}(G'_{e_2})$.
	
	\item \cite[Theorem 1]{delzant96} tells us that there exist integers $N_e$ such that, if $[\zeta^{-1}_e(\dot{H}_e):G'_e]$ is a multiple of $N_e$ for each $e\in E\Gamma$, then property (3) holds (\cite{delzant96} doesn't apply to the case where $G_v$ is two-ended, but in this case $\mathbb{P}_v$ will contain just one subgroup $H=G_v$ and (3) will follow from (1)).
	
	\item \cite[Theorem 1.1 (1)]{Osin07} tells us that property (4) holds provided each $[\zeta^{-1}_e(\dot{H}_e):G'_e]$ is sufficiently large.
	
	The four conditions described above can evidently be satisfied simultaneously, so the lemma follows.\qedhere
    \end{enumerate}
\end{proof}

Define $\bar{G}_v:=G_v/\llangle \zeta_e(G_{e}') \mid \tau(e) = v \rrangle$ for $v\in V\Gamma$. Since the $\dot{H}$ came from Theorem \ref{thm:MSQT}, Lemma \ref{lem:G'e}(1) implies that $\bar{G}_v$ is virtually special. As a result, there is a finite index torsion-free normal subgroup $\bar{G}'_v\trianglelefteq\bar{G}_v$. Let $G'_v$ be the preimage of $\bar{G}'_v$ under the quotient map $G_v\to\bar{G}_v$. The image of an incident edge group $\zeta_e(G_e)$ in $\bar{G}_v$ is finite, so has trivial intersection with $\bar{G}'_v$; Lemma \ref{lem:G'e}(4) then implies that \begin{equation}\label{G'vcap}
G'_v\cap\zeta_e(G_e)=\zeta_e(G'_e).
\end{equation}

Lemma \ref{lem:G'e}(3) implies that the kernel of $G_v\to\bar{G}_v$ is torsion-free, and $\bar{G}'_v\trianglelefteq\bar{G}_v$ is torsion-free by construction, hence $G'_v$ is torsion-free.

\begin{prop} \label{prop:virtuallyTorsionFree}
	Let $G \in \mathscr{C}$. If $G$ is balanced, then $G$ is virtually torsion-free and therefore supgroup separable.
\end{prop}
\begin{proof}
	We define a finite cover of graphs of groups $(\hat{G},\hat{\Gamma})\to(G,\Gamma)$, so that $\hat{G}\leqslant G$ is a finite index subgroup. The edge and vertex groups of $(\hat{G},\hat{\Gamma})$ will be copies of the $G'_e$ and $G'_v$ constructed earlier, which are torsion-free, so $\hat{G}$ will be torsion-free.
	
	The data for constructing the cover $(\hat{G},\hat{\Gamma})\to(G,\Gamma)$ is as follows.
	\begin{itemize}
		\item Have a surjective graph morphism $p:\hat{\Gamma}\to \Gamma$.
		\item For $\hat{v}\in V\hat{\Gamma}$ and $p(\hat{v})=v$, have an inclusion $\iota_{\hat{v}}:\hat{G}_{\hat{v}}\hookrightarrow G_v$ with image $G'_v$. For $\hat{e}\in E\hat{\Gamma}$ and $p(\hat{e})=e$, have an inclusion $\iota_{\hat{e}}:\hat{G}_{\hat{e}}\hookrightarrow G_e$ with image $G'_e$.
		\item If $\tau(\hat{e})=\hat{v}\in V\hat{\Gamma}$, $p(\hat{e})=e$ and $p(\hat{v})=v$, then there is $h_{\hat{e}}\in G_v$ such that the following diagram commutes
		\begin{equation}
		\begin{tikzcd}[
		ar symbol/.style = {draw=none,"#1" description,sloped},
		isomorphic/.style = {ar symbol={\cong}},
		equals/.style = {ar symbol={=}},
		subset/.style = {ar symbol={\subset}}
		]
		\hat{G}_{\hat{e}}\ar{rr}{\zeta_{\hat{e}}}\ar{d}{\iota_{\hat{e}}}&&\hat{G}_{\hat{v}}\ar{d}{\iota_{\hat{v}}}\\
		G_e\ar{r}{\zeta_e}&G_v\ar{r}{h_{\hat{e}}(-)h_{\hat{e}}^{-1}}&G_v.
		\end{tikzcd}
		\end{equation}
		Moreover, the elements $h_{\hat{e}}$ provide a complete set of double coset representatives $G'_vh_{\hat{e}}\zeta_e(G_e)$ as $\hat{e}$ ranges over edges in $p^{-1}(e)$ with $\tau(\hat{e})=\hat{v}$.
	\end{itemize}
One can check that this is indeed the correct data by thinking in terms of graphs of spaces and considering elevations of the various edge maps (we omit an explanation of this), or alternatively one can compare this data with \cite[Definitions 2.1 and 2.6]{Bass93}.

An alternative characterisation of the $h_{\hat{e}}$ (again with fixed $e$ and $\hat{v}$) is that they provide a complete set of coset representatives for the subgroup $G'_v\zeta_e(G_e)/G'_v$ in the finite quotient $G_v/G'_v$. Now
\begin{align}
\left|\frac{G'_v\zeta_e(G_e)}{G'_v}\right|&=[\zeta_e(G_e):\zeta_e(G_e)\cap G'_v]\\\nonumber
&=[\zeta_e(G_e):\zeta_e(G'_e)]&\text{by (\ref{G'vcap})}\\\nonumber
&=[G_e:G'_e],
\end{align}
so there will be $[G_v:G'_v]/[G_e:G'_e]$ such cosets, and hence the same number of $\hat{e}$.

As a result, we must satisfy the gluing equation
\begin{equation}\label{gluehatG}
|p^{-1}(v)|\frac{[G_v:G'_v]}{[G_e:G'_e]}=|p^{-1}(e)|
\end{equation}
whenever $\tau(e)=v\in V\Gamma$; and conversely, if we have numbers $|p^{-1}(v)|$ and $|p^{-1}(e)|$ that solve the equations (\ref{gluehatG}), then such a finite cover $(\hat{G},\hat{\Gamma})$ can be constructed. But such a solution is easy, just set
\begin{align}
|p^{-1}(v)|=\frac{M}{[G_v:G'_v]},&&|p^{-1}(e)|=\frac{M}{[G_e:G'_e]},
\end{align}
where $M$ is a common multiple of the $[G_v:G'_v]$ and $[G_e:G'_e]$.
\end{proof}

\subsection{Relative hyperbolicity and virtual specialness}\label{sec:relhypvspecial}

In this section we prove the other implications of Theorem \ref{thm:separableBalancedRelHyp}.

\begin{lem}\label{lem:cylinderproduct}
Let $G \in \mathscr{C}$ split as a finite balanced graph of groups $(G,\Gamma)$, where the edge groups are two-ended and the vertex groups are virtually free. Then, replacing $G$ by a finite index torsion-free subgroup, we can arrange that each cylinder $Y$ in the corresponding Bass-Serre tree $T$ has stabiliser $G_Y$ which admits a product splitting $G_Y=\mathbb{Z}\times F_n$ ($n\geq0$) such that the $\mathbb{Z}$ factor pointwise fixes $Y$ and the $F_n$ factor acts freely cocompactly on $Y$.
\end{lem}
\begin{proof}
	$G_Y$ acts cocompactly on $Y$, and all edge and vertex stabilisers are two-ended by Lemma \ref{lem:cyledgetwoended}.
	By Theorem \ref{balancedseparable} we can assume that $G$ is torsion-free, so then $G_Y$ splits as a graph of groups with all edge and vertex stabilisers isomorphic to $\mathbb{Z}$
	- such groups are called \emph{generalised Baumslag-Solitar groups} (or GBS groups). It follows from the proof of \cite[Proposition 2.6]{Levitt2007} that $G_Y$ contains a finite index subgroup $\dot{G}_Y$ which admits a product splitting $\dot{G}_Y=\mathbb{Z}\times F_n$ ($n\geq0$) such that the $\mathbb{Z}$ factor pointwise fixes $Y$ and the $F_n$ factor acts freely cocompactly on $Y$. Alternatively, we can apply Proposition \ref{prop:sepvertexstabs} to $(G_Y,Y)$ to produce a finite index subgroup $\dot{G}_Y\leqslant G_Y$ such that each vertex stabiliser of $\dot{G}_Y$ in $Y$ is equal to its incident edge stabilisers (recall that subgroup separability of $G$ implies subgroup separability of $G_Y$), this $\dot{G}_Y$ will admit a product splitting as above where the $\mathbb{Z}$ factor is equal to any vertex or edge stabiliser.
	
	Any finite index subgroup of $\dot{G}_Y$ will admit a similar product splitting, so we may apply Proposition \ref{prop:sepvertexstabs} to the action of $G$ on the tree of cylinders $T_c$ and a set of $G$-orbit representatives of cylinder vertices, and this will produce a finite index subgroup of $G$ satisfying the conclusions of the lemma.
\end{proof}

\begin{prop}\label{prop:hyprelcylinders}
	Let $G$ be a group acting on a tree $T$ with two-ended edge stabilisers and hyperbolic vertex stabilisers. Then $G$ is hyperbolic relative to its cylinder stabilisers.
\end{prop}
\begin{proof}
	Let $T_c$ be the tree of cylinders corresponding to $T$, and let $(G,\Gamma)$ be the quotient graph of groups for the action of $G$ on $T_c$. The partition $VT_c=V_0T_c\sqcup V_1 T_c$ induces a partition $V\Gamma=V_0\Gamma\sqcup V_1\Gamma$. We wish to show that $G$ is hyperbolic relative to its vertex groups $G_v$ for $v\in V_1\Gamma$ -- which we call its cylinder vertex groups. For the original tree $T$, two stabilisers of edges in different cylinders will have finite intersection, so for $u\in V_0\Gamma$ Lemma \ref{lem:cyledgetwoended} implies that different $G_u$-conjugates of edge groups incident at $G_u$ also have finite intersection. Then by \cite[Theorem 7.11]{Bowditch12} and Lemma \ref{lem:cyledgetwoended}, $G_u$ is hyperbolic relative to its incident edge groups. Next, for each $u\in V_0\Gamma$, let $(G^u,\Gamma^u)$ be the graph of groups obtained by amalgamating $G_u$ with its neighbouring cylinder vertex groups in $(G,\Gamma)$. By \cite[Theorem 0.1(2)]{Dahmani03}, $G^u$ is hyperbolic relative to its cylinder vertex groups in $(G^u,\Gamma^u)$. We can then join together the graphs of groups $(G^u,\Gamma^u)$ via a sequence of amalgamations and HNN extensions to recover the graph of groups $(G,\Gamma)$, and this will be hyperbolic relative to its cylinder vertex groups by \cite[Theorem 0.1(3)+(3')]{Dahmani03}.
\end{proof}

\begin{remk} \label{rem:RemovingCyclicPeripherals}
 Given the conclusion of Proposition~\ref{prop:hyprelcylinders}, we note that if a cylinder stabiliser is virtually infinite cyclic (and is therefore a hyperbolic group), we can remove it from the family of peripheral subgroups.
 (This is a special case of a more general result. See~\cite[Corollary 1.14]{DrutuSapir05}.)
 \end{remk}

\begin{proof}[Proof of Theorem~\ref{thm:separableBalancedRelHyp}]

 The equivalence of \ref{item:separable} and \ref{item:balanced} is Theorem~\ref{thm:torsionbalancedseparable}.
 %
%
 It remains to show the equivalence of  \ref{item:balanced}, \ref{item:relHyperbolic} and \ref{item:special}. Fix an action of $G$ on a tree $T$ with two-ended edge stabilisers and virtually free vertex stabilisers and let $(G,\Gamma)$ be the quotient graph of groups.

 Let's start by showing the equivalence of \ref{item:balanced}, that $(G,\Gamma)$ is balanced, and \ref{item:relHyperbolic}, that $G$ is hyperbolic relative to peripheral subgroups that are virtually  $\mathbb{Z}\times\mathbb{F}_n$ ($n\geq0$). \ref{item:balanced} implies \ref{item:relHyperbolic} by combining Lemma \ref{lem:cylinderproduct} and Proposition \ref{prop:hyprelcylinders}. Conversely, suppose for contradiction we have \ref{item:relHyperbolic} but not \ref{item:balanced}, then Lemma~\ref{lem:balanceequiv} gives us infinite order elements $h,g$ such that $gh^pg^{-1} = h^q$ with $|p|\neq|q|$.
 By~\cite[Corollary 4.21]{Osin06} the element $h$ must lie in a (conjugate of a) peripheral subgroup, call it $P$.
 Moreover, $g$ will also belong to $P$, otherwise $\langle gh^pg^{-1}\rangle= \langle h^q\rangle\leqslant P\cap gPg^{-1}$, contradicting the almost malnormality of the peripheral subgroups. 
 But then we contradict $P$ being virtually $\mathbb{Z}\times\mathbb{F}_n$.

 Next we'll show the equivalence of \ref{item:balanced}, that $(G,\Gamma)$ is balanced, and  \ref{item:special}, that $G$ is virtually special. Firstly suppose that $(G,\Gamma)$ is balanced, and let $T$ be the corresponding Bass-Serre tree; by Lemma \ref{lem:cylinderproduct} we may assume that $G$ is torsion-free and that its cylinder stabilisers are isomorphic to $\mathbb{Z}\times\mathbb{F}_n$, with all stabilisers of edges in the cylinder being equal to the $\mathbb{Z}$ factor. By \cite{HsuWise10}, $G$ is the fundamental group of a non-positively curved cube complex $X$; moreover, the $v$-arcs from \cite[Definition 10.1]{HsuWise10} are hyperplanes that correspond to the edge groups in $(G,\Gamma)$, so $X$ decomposes as a graph of cube complexes in the sense of \cite{HuangWise19} corresponding to $(G,\Gamma)$. We want to show that $X$ is virtually special. By \cite[Theorem 1.4]{HuangWise19} it is enough to show that $G$ has finite stature with respect to its vertex stabilisers in $T$ (finite stature is defined in \cite[Definition 1.2]{HuangWise19}). It suffices to show that for any $e_1,e_2\in ET$ either $G_{e_1}\cap G_{e_2}=G_{e_1}$ or $G_{e_1}\cap G_{e_2}=\{1\}$. Indeed if $e_1$ and $e_2$ belong to the same cylinder then $G_{e_1}\cap G_{e_2}=G_{e_1}$ by our assumption on the cylinders, and otherwise the edge groups are not commensurable so intersect trivially.
 
 Finally, suppose that $(G,\Gamma)$ is not balanced. Again, by Lemma~\ref{lem:balanceequiv}, $G$ contains infinite order elements $g,h$ with $gh^pg^{-1} = h^q$ and $|p|\neq|q|$, hence so will any finite index subgroup of $G$. 
 This implies that $G$ is not virtually cubulated -- as this would contradict the conjugation invariance of the combinatorial translation length of isometries of a CAT(0) cube complex (see~\cite{HaglundSemiSimple, Woodhouse17}).
\end{proof}

 \begin{remk} \label{rem:peripheralIsCylindrical}
 We observe that if we know $G$ is hyperbolic relative to a family $\mathcal{P}$ of virtually abelian peripheral subgroups (where $\mathcal{P}$ might not be the family of cylinder stabilisers), then the cylinder stabilisers will also be virtually abelian.
 Indeed by Theorem~\ref{thm:separableBalancedRelHyp} we know that the cylinder stabilisers are virtually $\mathbb{Z}\times\mathbb{F}_n$, so we just need to show that $n\leq 1$. Each cylinder stabiliser is undistorted (because $G$ is hyperbolic relative to its cylinder stabilisers) and unconstricted, so we may apply~\cite[Theorem 1.7]{DrutuSapir05} to conclude that each cylinder stabiliser is contained in a neighbourhood of a conjugate of some $P\in\mathcal{P}$ (note that \cite[Theorem 1.7]{DrutuSapir05} has a typo, $G'\to G$ should be a quasi-isometric embedding rather than a quasi-isometry).
 The observation then follows because there is no quasi-isometric embedding $\mathbb{Z}\times\mathbb{F}_n\to P$ if $n\geq2$ (for example because $\mathbb{Z}\times\mathbb{F}_n$ has exponential growth and $P$ has polynomial growth).
\end{remk}

\bigskip
\section{Leighton's theorem for graphs with coloured fins} \label{sec:LeightonAgain}

Leighton's Theorem for graphs with fins was proven by the second author \cite[Theorem 0.1]{FinLeighton}; in this section we build on this result by adding colours and orientations to the fins and arranging for the common finite cover to satisfy a symmetry property.
The orientations of the fins are particularly important.
In Sections \ref{sec:GraphOfSpaces} and \ref{sec:CommonFiniteCover} we will construct graphs of spaces by taking graphs with fins and gluing the ends of certain fins together by homeomorphisms. 
The homotopy type of such a graph of spaces will not only depend on which fins you glue together, but on the orientations of the fins that get matched up by the gluing.

\subsection{Definitions}

\begin{defn}(Graph with coloured fins)\label{graphcolfin}\\
 Let $X$ be a graph, which we now consider to be a $1$-dimensional cube complex. 
 Let $\Delta$ be a collection of combinatorial immersions $\gamma:S\to X$, where each $S$ is a circle or a bi-infinite line subdivided into $\ell(S)$ edges ($\ell(S)=\infty$ if $S$ is a bi-infinite line).
 A \emph{graph with fins} $\bfX$ is a non-positively curved square complex obtained by taking the mapping cylinder of  $$\cup_\Delta \gamma:\bigsqcup_\Delta S\to X.$$
 A graph with fins $\bfX$ is finite if it is a finite cube complex. The subset $$\bigsqcup_\Delta S \times \{1\} \subseteq \bfX$$
 is the \emph{boundary} of the graph with fins. Each component of the boundary, $S \times \{1\}$, is called a \emph{fin} -- for ease of notation we will always write $S$ instead of $S \times \{1\}$. The collection of fins is denoted $\partial \bfX$. Meanwhile, the subsets $S\times\{0\}$ lie in $X$. The natural retraction $r: \bfX \rightarrow X$ restricted to the boundary allows us to recover the collection $\Delta$.
 
 A fin $S\in\partial\bfX$ is a 1-manifold, so can be given an orientation $\tto$. The pair $(S,\tto)$ is an \emph{oriented fin}, and will often be written as $\bbS$. If $\bbS=(S,\tto)$ then we write $\bar{\bbS}=(S,\bar{\tto})$ for the fin with opposite orientation. The \emph{length} of $\bbS$ is $\ell(\bbS):=\ell(S)$. The collection of oriented fins is denoted $\partial_\tto\bfX$. If we have a colouring $\lambda: \partial_\tto \bfX \rightarrow \mathcal{C}$, then we say that $\bfX$ is a \emph{graph with coloured fins}.

\end{defn}

\begin{defn}(Coverings and automorphisms of graphs with coloured fins)\label{defn:covergraphfin}\\
	A \emph{covering of graphs with fins} $\Phi:\widehat{\mathbf{X}}\to \bfX$ is a covering of square complexes that restricts to a graph covering $\widehat{X}\to X$ -- we require $\bfX$ to be connected but $\widehat{\bfX}$ doesn't need to be. 
	
	The restriction of $\Phi$ to a fin $\hat{S}\in\partial\widehat{\mathbf{X}}$ is a covering $\hat{S}\to S$ of a fin $S\in\partial\bfX$. If $\hat{\bbS}=(\hat{S},\hat{\tto})$ and $\bbS=(S,\tto)$ are orientations respected by the covering, then we say that $\hat{\bbS}\to\bbS$ is a \emph{covering of oriented fins} (we will usually just say that $\hat{\bbS}\to\bbS$ is a covering). Thus we get a map $\Phi:\partial_\tto\widehat{\bfX}\to\partial_\tto \bfX$ where each $\hat{\bbS}\to\Phi(\hat{\bbS})$ is a covering. We call $\Phi:\widehat{\mathbf{X}}\to \bfX$ a \emph{covering of graphs with coloured fins} if the induced map $\Phi:\partial_\tto\widehat{\bfX}\to\partial_\tto \bfX$ preserves colours (both $\bfX$ and $\widehat{\bfX}$ must use the same set of colours $\mathcal{C}$). 
	
	A covering $\widehat{\mathbf{X}}\to \bfX$ is an \emph{isomorphism} if it is an isomorphism of square complexes. An isomorphism $\bf{X}\to\bf{X}$ is an \emph{automorphism}. 
	Let $\Aut(\bf X)$ denote the group of automorphisms of $\bf X$.
	We note that any automorphism of $\bfX$ also induces an automorphism on $\partial_\tto \bfX$.
	A covering $\widetilde{\mathbf{X}}\to \bfX$ is a \emph{universal covering} if $\widetilde{X}$ is a tree, or equivalently if $\widetilde{\bfX} \to \bfX$ is a universal covering of square complexes. In this case, the deck transformations of $\wt{\bfX} \to \bfX$ induce a subgroup of $\Aut(\widetilde{\mathbf{X}})$. 
\end{defn}

\begin{exmp}
Let $X$ be the bouquet of two circles -- the graph given by a single vertex and two edges.
We fix a generating set $\pi_1X = \langle x, y \rangle$ so that the generators $x$ and $y$ correspond to the two edges.
Let $\bfX$ be the graph with fins determined by the geodesic paths given by the set $\{ x, y, xy\}$.
In this example the oriented fins can be written out as $\partial_{\tto} \bfX = \{x, x^{-1}, y, y^{-1}, xy, y^{-1}x^{-1} \}$.
See Figure~\ref{fig:graphwithfins} for an illustration of $\bfX$.

 \begin{figure}[H]
 \centering
	\begin{overpic}[width=.55\textwidth,tics=5,]{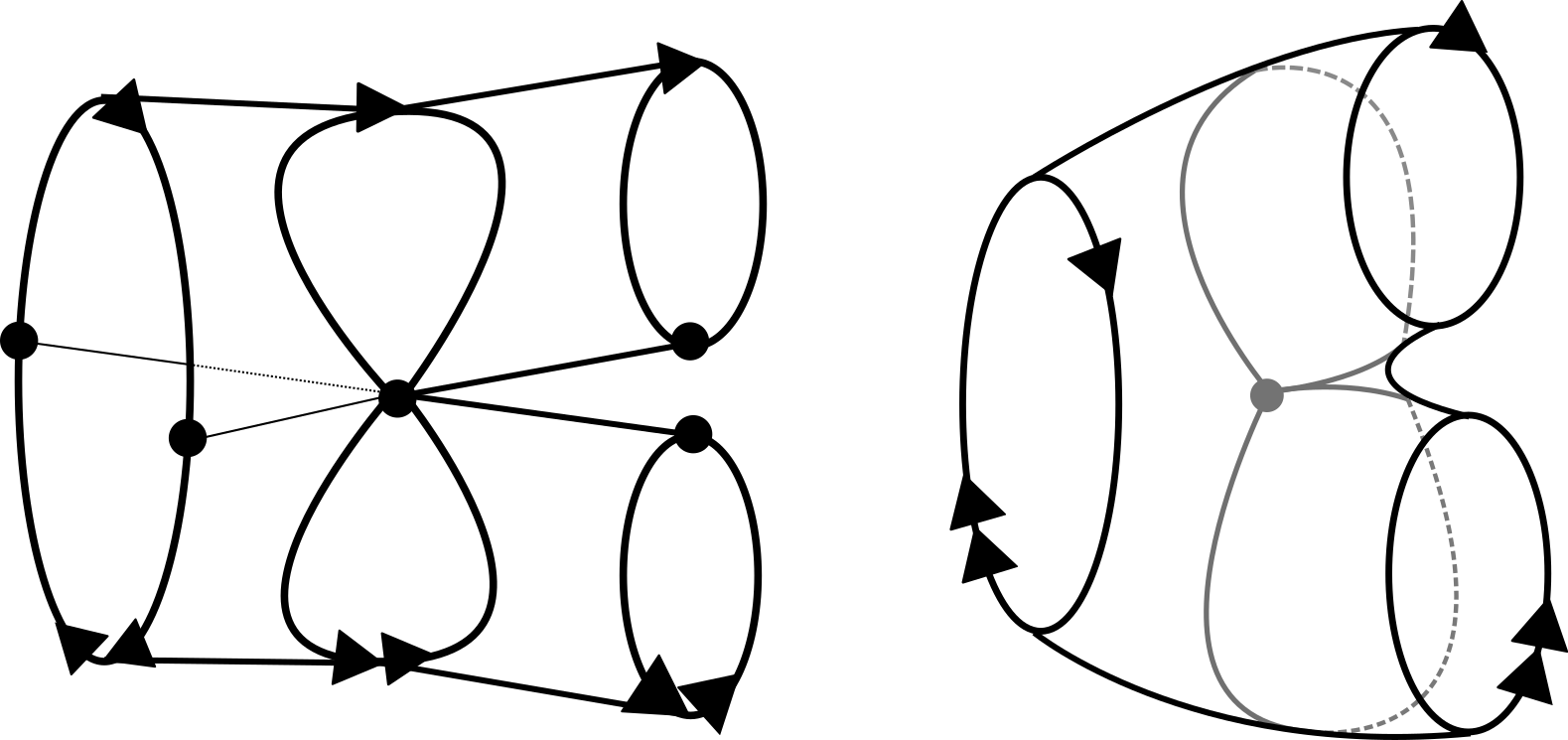} 
		\put(25,43){\Large$x$}
		\put(25,0){\Large$y$}
          \end{overpic}
	\caption{A graph with fins -- drawn again on the right to emphasize in this case it is homeomorphic to a surface with boundary.}
	\label{fig:graphwithfins}
    \end{figure}
 
\end{exmp}

\begin{remk}
	A graph with coloured fins $\mathbf{X}$ and a graph covering $\widehat{X}\to X$ uniquely determine a coloured fin structure $\widehat{\mathbf{X}}$ on $\widehat{X}$ and a covering $\widehat{\mathbf{X}}\to\mathbf{X}$.
\end{remk}

\begin{defn}
	If $\Phi_i:\widehat{\bfX}\to\bfX_i$ are coverings for $i=1,2$, and $\bbS_i\in\partial_\tto\bbX_i$ are oriented fins, then we write
	$$\partial_\tto\widehat{\bfX}(\bbS_1,\bbS_2):=\Phi_1^{-1}(\bbS_1)\cap\Phi_2^{-1}(\bbS_2)$$
	for the collection of oriented fins in $\widehat{\bfX}$ that cover both $\bbS_1$ and $\bbS_2$.
\end{defn} 

\begin{defn}(Density) \label{defn:density} \\
For $\bf X$ a finite graph with coloured fins and $c\in\mathcal{C}$ a colour, define the \emph{density} $\rho_c$ by
	\begin{equation}\label{density}
	\rho_c:=\sum_{\lambda(\bbS)=c}\ell(\bbS)/|X|,
	\end{equation}
	where $|X|$ is the number of vertices in $X$.
	Note that densities $\rho_c$ are preserved by finite coverings, and are therefore invariants of the commensurability class of $\mathbf{X}$.
\end{defn}

\subsection{The theorem} \label{sec:comfingraph}

\begin{thm}(Leighton's Theorem for graphs with coloured fins)\label{Leighton}\\
	Let $\mathbf{X}_1$ and $\mathbf{X}_2$ be graphs with coloured fins that have a common universal cover $\widetilde{\mathbf{X}}$. Denote the covering maps by $\Psi_i:\wt{\bfX}\to\bfX_i$ and let $\Gamma_1,\Gamma_2\leqslant\rm{Aut}(\widetilde{\mathbf{X}})$ be the corresponding deck transformation groups. Suppose that $\Aut(\widetilde{\mathbf{X}})$ acts transitively on the oriented fins of each colour in $\widetilde{\mathbf{X}}$.
	Then $\mathbf{X}_1$ and $\mathbf{X}_2$ have a common finite cover $\widehat{\mathbf{X}}$ such that 
	\begin{equation}\label{finequation}
	\sum_{\hat{\bbS}\in\partial_\tto\widehat{\bfX}(\bbS_1,\bbS_2)}\ell(\hat{\bbS})=\left(\frac{|\widehat{X}|}{\rho_c|X_1||X_2|}\right)\ell(\bbS_1)\ell(\bbS_2),
	\end{equation}
	for any $\bbS_i\in\partial_\tto\bfX_i$ of the same colour $c$.
\end{thm}

The rest of this section is devoted to proving this theorem, so fix $\mathbf{X}_1$, $\mathbf{X}_2$,  $\widetilde{\mathbf{X}}$, and $\Gamma_1,\Gamma_2< \Aut(\widetilde{\mathbf{X}})$ as above. For brevity we will write $H=\Aut(\widetilde{\mathbf{X}})$ for the rest of this section.
We will assume that the graphs $X_i$ are simplicial, and that $H$ doesn't invert edges in $\widetilde{X}$. 
We can achieve these properties by subdividing the edges of the graphs $X_i$ (and passing to an index two subgroup of $H$ if the $X_i$ are circles). Note that equation (\ref{finequation}) is preserved by subdividing edges in underlying graphs; indeed $|\widehat{X}|/|X_2|$ is the degree of $\widehat{X}\to X_2$, so is unchanged, and the quantities $\ell(\hat{\bbS})$, $\ell(\bbS_1)$, $\ell(\bbS_2)$ and $\rho_c|X_1|$ all increase by a factor of two.

\begin{defn}(Polyhedra and faces)\\
Let $\bfX$ be a graph with coloured fins.
A hyperplane in $\bfX$ is \emph{vertical} if it is dual to an edge in $X$ -- let $\mathscr{H}$ denote the set of vertical hyperplanes. 
Let ${\dot \bfX}$ denote the square complex obtained from $\bfX$ by subdividing along the vertical hyperplanes. 
A \emph{polyhedron} $(P,\phi)$ is a square complex $P$ equipped with a cubical embedding $\phi:P\to {\dot \bfX}$ such that $\phi(P)$ is the cubical neighbourhood in ${\dot{\bfX}}$ of a vertex $x\in X$. Alternatively, we can think of $\phi(P)$ as the closure of the component of $\bfX -\mathscr{H}$ containing $x$. 
We call $x$ the \emph{centre} of $\phi(P)$. 
A \emph{face} $(F,\varphi)$ is a finite tree $F$ equipped with a cubical embedding $\varphi:F\to {\dot \bfX}$ such that $\varphi(F)$ is a vertical hyperplane in $\bfX$ (which is a subcomplex in ${\dot \bfX}$). We say that $(F,\varphi)$ is a \emph{face} of $(P,\phi)$ if there is a commutative diagram of cubical embeddings
\begin{equation}
\begin{tikzcd}[
ar symbol/.style = {draw=none,"#1" description,sloped},
isomorphic/.style = {ar symbol={\cong}},
equals/.style = {ar symbol={=}},
subset/.style = {ar symbol={\subset}}
]
F\ar{r}\ar{dr}[swap]{\varphi}&P\ar{d}{\phi}\\
&{\dot \bfX}.
\end{tikzcd}
\end{equation}
Fixing an orientation on each edge in $X$, we have a notion of being on the \emph{left} or \emph{right} of a vertical  hyperplane in $\bfX$. We say that $(P,\phi)$ is on the \emph{left} (resp. \emph{right}) of a face $(F,\varphi)$ if $\phi(P)$ is on the left (resp. right) of $\varphi(F)$ (there is no ambiguity as we have assumed $X$ is simplicial). Up to isomorphism there is a unique polyhedron on the left and right of each face.\par 
 If $(P,\phi)$ and $(P',\phi')$ are polyhedra on the left and right of a face $(F,\varphi)$, then the polyhedra can be glued together along the embeddings of $F$ to make a new complex $P\cup P'$ that maps into $\bfX$ via $\phi\cup\phi'$.
\end{defn}

\begin{defn}(Polyhedral pairs and face pairs)\label{polypair}\\
	 A \emph{polyhedral pair} is a triple $\mathbf{P}=(P,\phi_1,\phi_2)$ where each pair $(P,\phi_i)$ is a polyhedron for $\mathbf{X}_i$. We say that $\bf P$ is $H$-\emph{admissible} if there is a commutative diagram as follows, which we will refer to as the \emph{admissibility diagram},
	\begin{equation}\label{Padmisible}
	\begin{tikzcd}[
	ar symbol/.style = {draw=none,"#1" description,sloped},
	isomorphic/.style = {ar symbol={\cong}},
	equals/.style = {ar symbol={=}},
	subset/.style = {ar symbol={\subset}}
	]
	{\widetilde{\bfX}}\ar{dd}[swap]{\Psi_1}\ar{rr}{h}&& {\widetilde{\bfX}}\ar{dd}{\Psi_2}\\
	&P\ar{ul}{\tilde{\phi}_1}\ar{ur}[swap]{\tilde{\phi}_2}\ar{dl}{\phi_1}\ar{dr}[swap]{\phi_2}\\
	\bfX_1&&\bfX_2,
	\end{tikzcd}
	\end{equation}
	where $\tilde{\phi}_i$ are lifts of the maps $\phi_i$ and $h\in H$. Note that the lifts $\tilde{\phi}_i$ are unique up to post-composition by $g_i\in\Gamma_i$, so if $\bf P$ is admissible then the diagram (\ref{Padmisible}) can be constructed for any lifts $\tilde{\phi}_i$.
	
	Similarly, a \emph{face pair} is a triple $\mathbf{F}=(F,\varphi_1,\varphi_2)$ where each pair $(F,\varphi_i)$ is a face for $\mathbf{X}_i$. We say that $\bf F$ is $H$-\emph{admissible} if there is a commutative diagram
	\begin{equation}\label{Fadmisible}
	\begin{tikzcd}[
	ar symbol/.style = {draw=none,"#1" description,sloped},
	isomorphic/.style = {ar symbol={\cong}},
	equals/.style = {ar symbol={=}},
	subset/.style = {ar symbol={\subset}}
	]
	\widetilde{\bfX}\ar{dd}[swap]{\Psi_1}\ar{rr}{h}&&\widetilde{\bfX}\ar{dd}{\Psi_2}\\
	&F\ar{ul}{\tilde{\varphi}_1}\ar{ur}[swap]{\tilde{\varphi}_2}\ar{dl}{\varphi_1}\ar{dr}[swap]{\varphi_2}\\
	\bfX_1&&\bfX_2,
	\end{tikzcd}
	\end{equation}
	where $\tilde{\varphi}_i$ are lifts of the maps $\varphi_i$ and $h\in H$. We say that a polyhedral pair $\mathbf{P}=(P,\phi_1,\phi_2)$ is on the \emph{left} (resp. \emph{right}) of a face pair $\mathbf{F}=(F,\varphi_1,\varphi_2)$ if $(P,\phi_i)$ is on the left (resp. right) of $(F,\varphi_i)$ for $i=1,2$ and with respect to the same embedding $F\to P$. Note that it is impossible for $(P,\phi_1)$ to be on the left of $(F,\varphi_1)$ and for $(P,\phi_2)$ to be on the right of $(F,\varphi_2)$ with respect to the same embedding $F\to P$ because $H$ has no edge-inversions. Let $\overleftarrow{\bf F}$ (resp. $\overrightarrow{\bf F}$) denote the set of admissible polyhedral pairs on the left (resp. right) of $\bf F$. Note that  $\overleftarrow{\bf F}$ and $\overrightarrow{\bf F}$ are finite since $\mathbf{X}_1$ and $\mathbf{X}_2$ are. If $\bf P\in\overleftarrow{\bf F}$ and $\mathbf{P}' \in\overrightarrow{\bf F}$ then we can glue together $P$ and $P'$ along the embeddings of $F$ to obtain a complex $P\cup P'$ with maps $\phi_1\cup\phi'_1$ and $\phi_2\cup\phi'_2$ to $\bfX_1$ and $\bfX_2$.
\end{defn}
\bigskip

Given a polyhedron $(P,\phi_1)$ for $\mathbf{X}_1$, we will be interested in counting the ways it can be extended to an admissible polyhedral pair $\mathbf{P}=(P,\phi_1,\phi_2)$, subject to forcing $\mathbf{P}\in\overleftarrow{\bf F}$ for a fixed face pair $\bf F$.

\begin{lem}\label{leftFchoice}
Let $(P,\phi_1)$ be a polyhedron for $\mathbf{X}_1$ and choose a lift $\tilde{\phi}_1:P\to\widetilde{\bfX}$ with image $\widetilde{P}$. 
Let $(P,\phi_1)$ be on the left (resp. right) of a face $(F,\varphi_1)$, and let $\tilde{\phi}_1(F)=\widetilde{F}$ (viewing $F$ as a subset of $P$). 
Suppose $(F,\varphi_1)$ extends to an admissible face pair $\mathbf{F}=(F,\varphi_1,\varphi_2)$. 
Then the choices $\phi_2$ such that $(P,\phi_1,\phi_2)\in\overleftarrow{\bf F}$ (resp. $\overrightarrow{\bf F}$) are in one to one correspondence with the quotient $H_{(\widetilde{F})}/H_{(\widetilde{P})}$ -- where $H_{(\widetilde{F})}$ and $H_{(\widetilde{P})}$ are the pointwise stabilisers of $\widetilde{F}$ and $\widetilde{P}$ respectively.
\end{lem}
\begin{proof}
	Assume $(P,\phi_1)$ is on the left of $(F,\varphi_1)$. Now $(F,\varphi_1,\varphi_2)$ fits into a commutative diagram (\ref{Fadmisible}) for some $h\in H$ and lifts $\tilde{\varphi}_i$, and we can choose $\varphi_1=\phi_1|_F$. Then any $\phi_2$ such that $(P,\phi_1,\phi_2)\in\overleftarrow{\bf F}$ will fit into an admissibility diagram
	\begin{equation}\label{Padmissiblelem}
	\begin{tikzcd}[
	ar symbol/.style = {draw=none,"#1" description,sloped},
	isomorphic/.style = {ar symbol={\cong}},
	equals/.style = {ar symbol={=}},
	subset/.style = {ar symbol={\subset}}
	]
	\widetilde{\bfX}\ar{dd}[swap]{\Psi_1}\ar{rr}{h'}&&\widetilde{\bfX}\ar{dd}{\Psi_2}\\
	&P\ar{ul}{\tilde{\phi}_1}\ar{ur}[swap]{\tilde{\phi}_2}\ar{dl}{\phi_1}\ar{dr}[swap]{\phi_2}\\
	\bfX_1&&\bfX_2,
	\end{tikzcd}
	\end{equation}
	for some $h'\in H$. As $\phi_2|_F=\varphi_2$, we know that $\tilde{\phi}_2|_F$ and $\tilde{\varphi}_2$ differ by an element of $\Gamma_2$; so by composing $h'$ with an element of $\Gamma_2$, we may assume that $\tilde{\phi}_2|_F=\tilde{\varphi}_2$. Then $h'|_{\widetilde{F}}=h|_{\widetilde{F}}$, hence $h'\in hH_{(\widetilde{F})}$. Conversely, any $h'\in hH_{(\widetilde{F})}$ defines a polyhedral pair $(P,\phi_1,\phi_2)\in\overleftarrow{\bf F}$ via (\ref{Padmissiblelem}). Finally, the map $\phi_2$ only depends on the coset $h'H_{(\widetilde{P})}$, again because of (\ref{Padmissiblelem}). This establishes the desired bijection between the choices $\phi_2$ and the quotient $H_{(\widetilde{F})}/H_{(\widetilde{P})}$.	
\end{proof}

 We want to take appropriate numbers of copies of each admissible polyhedral pair so that we can glue them all together along face pairs (as we described at the end of Definition \ref{polypair}) to form a common finite cover of $\mathbf{X}_1$ and $\mathbf{X}_2$. We formalise this with the following definition.
 \begin{defn}(Gluing Equations)\\
Let $\mathcal{P}$ be the (finite) collection of all admissible polyhedral pairs, and let $\omega:\mathcal{P}\to\mathbb{Z}_{>0}$ denote a weight function on $\mathcal{P}$. For each admissible face pair $\bf F$ we have the following \emph{Gluing Equation}:
\begin{equation}\label{gluingeq}
\sum_{\bf P\in\overleftarrow{\bf F}}\omega(\bf P)=\sum_{\bf P\in\overrightarrow{\bf F}}\omega(\bf P).
\end{equation}	
 \end{defn} 
 Given a solution, we can take $\omega(\bf P)$ copies of each $\bf P$, and glue them together along faces according to (arbitrary) bijections 
 \begin{equation}
\{(\mathbf{P},i)\mid\mathbf{P}\in\overleftarrow{\mathbf{F}},\,1\leq i\leq\omega(\mathbf{P})\}\leftrightarrow\{(\mathbf{P},i)\mid\mathbf{P}\in\overrightarrow{\mathbf{F}},\,1\leq i\leq\omega(\bf P)\},
 \end{equation}
 and this will give us a common finite cover of $\mathbf{X}_1$ and $\mathbf{X}_2$. For the moment we won't worry about colouring the oriented fins in this finite cover.
 
 To solve the gluing equations, we will consider the Haar measure $\mu$ for the group $H$. As $H$ contains a uniform lattice -- for example $\Gamma_1$ -- $H$ is unimodular and $\mu$ is both left and right $H$-invariant.
 Note that $\mu$ is positive on every open set and finite on every compact set, both of which apply to the stabilisers $H_{(\widetilde{P})}$ and $H_{(\widetilde{F})}$. 
 There are finitely many $H$-orbits of images of polyhedra $\widetilde{P}$ in $\widetilde{\bfX}$, and so, by $H$-invariance of $\mu$, there are finitely many values $\mu(H_{(\widetilde{P})})$; furthermore, the stabilisers $H_{(\widetilde{P})}$ are all commensurable in $H$, so by rescaling we can assume that all $\mu(H_{(\widetilde{P})})$ are positive integers.
  For each $\mathbf{P}=(P,\phi_1,\phi_2)\in\mathcal{P}$, choose a lift $\tilde{\phi}_1:P\to\widetilde{\mathbf{X}}$ with image $\widetilde{P}$, and set
 \begin{equation}\label{Haarsol}
 \omega(\mathbf{P})=\mu(H_{(\widetilde{P})}).
 \end{equation}
 Observe that $\omega(\mathbf{P})$ is independent of the choice of lift $\tilde{\phi}_1$ because of the left and right $H$-invariance of $\mu$.\par 
 \begin{lem}
 	The Haar measure weight function (\ref{Haarsol}) solves the Gluing Equations (\ref{gluingeq}).
 \end{lem}
\begin{proof}
Given an admissible face pair $\mathbf{F}=(F,\varphi_1,\varphi_2)$, let $(P,\phi_1)$ be the polyhedron on the left of $(F,\varphi_1)$. All $\mathbf{P}\in\overleftarrow{\mathbf{F}}$ can be obtained by choosing a map $\phi_2$ such that $(P,\phi_1,\phi_2)\in\overleftarrow{\bf F}$, and by Lemma \ref{leftFchoice} there are $H_{(\widetilde{F})}/H_{(\widetilde{P})}$ such choices, where $\widetilde{F}\subset\widetilde{P}\subset\widetilde{\mathbf{X}}$ comes from a lift of $(P,\phi_1)$. Substituting (\ref{Haarsol}) into the left hand side of (\ref{gluingeq}) then gives us
\begin{align*}
\sum_{\bf P\in\overleftarrow{\bf F}}\omega(\bf P)&=\sum_{\bf P\in\overleftarrow{\bf F}}\mu(H_{(\widetilde{P})})\\
&=|H_{(\widetilde{F})}:H_{(\widetilde{P})}|\mu(H_{(\widetilde{P})})\\
&=\mu(H_{(\widetilde{F})}).
\end{align*}
Observe that this only depends on $\bf F$, and so by a symmetric argument we get the same value if we substitute (\ref{Haarsol}) into the right hand side of (\ref{gluingeq}).
\end{proof}
 \bigskip
 
 We have now constructed a common finite cover of $\mathbf{X}_1$ and $\mathbf{X}_2$, call it $\widehat{\bfX}$ say. Denote the covering maps by $\Phi_i:\widehat{\bfX}\to\bfX_i$. We colour the oriented fins of $\widehat{\bfX}$ by pulling back the colours from $\bfX_1$ and $\bfX_2$, which is well-defined by the following lemma. This makes the $\Phi_i$ coverings of graphs with coloured fins.
 
 \begin{lem}
 	If we pull back the colours on $\partial_\tto\bfX_1$ to $\partial_\tto\widehat{\bfX}$, then the covering $\Phi_2:\widehat{\bfX}\to\bfX_2$ preserves colours.
 \end{lem}
\begin{proof}
	Take an oriented fin $\hat{\bbS}=(\hat{S},\hat{\tto})\in\partial_\tto\widehat{\bfX}$ with $\Phi_i(\hat{\bbS})=\bbS_i$. We must show that $\bbS_1$ and $\bbS_2$ have the same colour. Let $\tilde{\bbS}\in\partial_\tto\widetilde{\bfX}$ be a lift of $\bbS_1$ to $\widetilde{\bfX}$. If $\hat{\bbS}$ crosses a polyhedral pair $\bfP$ in $\widehat{\bfX}$, then we have an admissibility diagram (\ref{Padmisible}) with some $h\in H$. Restricting to the fins, we get the following commutative diagram.
	
	\begin{equation}
	\begin{tikzcd}[
	ar symbol/.style = {draw=none,"#1" description,sloped},
	isomorphic/.style = {ar symbol={\cong}},
	equals/.style = {ar symbol={=}},
	subset/.style = {ar symbol={\subset}}
	]
	{\tilde{\bbS}}\ar{dd}[swap]{\Psi_1}\ar{rr}{h}&& {h\tilde{\bbS}}\ar{dd}{\Psi_2}\\
	&P\ar{ul}{\tilde{\phi}_1}\ar{ur}[swap]{\tilde{\phi}_2}\ar{dl}{\phi_1}\ar{dr}[swap]{\phi_2}\\
	\bbS_1&&\bbS_2
	\end{tikzcd}
	\end{equation}
	As $h:\widetilde{\bfX}\to\widetilde{\bfX}$ preserves colours of oriented fins, we see that $\bbS_1$, $\tilde{\bbS}$, $h\tilde{\bbS}$ and $\bbS_2$ all have the same colour, as required.
\end{proof}

\bigskip

We now turn to proving equation (\ref{finequation}) from Theorem \ref{Leighton}. We will need the following definition. 

\begin{defn}(Arcs and oriented arcs)\\
	Let $(P,\phi)$ be a polyhedron for a graph with fins $\bfX$. An \emph{arc} in $(P,\phi)$ is a component of $\phi^{-1}(\partial\bfX)$, and its image in $\bfX$ is also called an \emph{arc}. 
	If $A\subset\bfX$ is an arc, then it is contained in a unique fin $S\in\partial\bfX$, it is homeomorphic to an interval, and it contains exactly one vertex of $S$. 
	An arc $A$ can be given an orientation $\tto$ as a 1-manifold, making it an \emph{oriented arc} $\bbA = (A, \tto)$. 
	If $\bbA$ is an oriented arc contained in an oriented fin $\bbS$ such that the orientations agree, then we say that $\bbA$ is an \emph{oriented subarc of $\bbS$}.
	
	If $\Phi:\widehat{\bfX}\to\bfX$ is a covering of graphs with fins, then it maps each arc in $\widehat{\bfX}$ homeomorphically to an arc in $\bfX$. Moreover, if $\hat{A}$ is an arc in $\widehat{\bfX}$ with orientation $\hat{\bbA}$, and $A=\Phi(\hat{A})$ is its image in $\bfX$, then we get an induced orientation $\bbA$ on $A$ such that $\hat{\bbA}\to\bbA$ is an orientation preserving homeomorphism, and we write $\bbA=\Phi(\hat{\bbA})$. Note that if $\hat{\bbA}$ is an oriented subarc of $\hat{\bbS}$, then $\Phi(\hat{\bbA})$ is an oriented subarc of $\Phi(\hat{\bbS})$. 
\end{defn}

Similarly, an \emph{arc} in a polyhedral pair $\bfP=(P,\phi_1,\phi_2)$ is a component of $\phi_1^{-1}(\partial\bfX_1)=\phi_2^{-1}(\partial\bfX_2)$. 
As $\widehat{\bfX}$ is built from the polyhedral pairs in $\mathcal{P}$, we see that the arcs in $\widehat{\bfX}$ correspond exactly to the arcs in elements of $\mathcal{P}$.
	We let $\partial_{\tto} \bfP$ denote the set of oriented arcs in $\bfP$.

Fix oriented fins $\bbS_1\in\partial_\tto \bfX_1$ and $\bbS_2\in\partial_\tto \bfX_2$ of the same colour $c$. Our goal is to sum lengths of fins in $\partial_\tto\widehat{\bfX}(\bbS_1,\bbS_2)$; we will do this by counting admissible polyhedral pairs whose images contain certain oriented subarcs of $\bbS_1$ and $\bbS_2$.

\begin{defn}\label{Acount}
	Let $\bbA_i$ be oriented subarcs of the oriented fins $\bbS_i$, and define 
	\begin{equation}
	P(\bbA_1,\bbA_2):=\big\{ \bfP = (P,\phi_1,\phi_2)\in\mathcal{P} \mid \exists \bbA \in \partial_{\tto} \bfP,\, \;\phi_i(\bbA)=\bbA_i \big\}, 
	\end{equation}
	the collection of polyhedral pairs containing oriented arcs that map to the $\bbA_i$. Then define
	\begin{equation}
	A(\bbA_1,\bbS_2):=\{\bbA_2\mid \bbA_2\text{ is an oriented subarc of $\bbS_2$ with }P(\bbA_1,\bbA_2)\neq\emptyset\}.
	\end{equation}
\end{defn}

In order to enumerate the elements of $P(\bbA_1,\bbA_2)$, we fix a polyhedron $(P,\phi_1)$ for $\mathbf{X}_1$ that contains $\bbA_1$ in its image $\phi_1(P)$ (this polyhedron will be unique up to isomorphism). 
Let $\bbA = (A, \tto)$ be the (unique) oriented arc in $P$ with $\phi_1(\bbA)=\bbA_1$. 
Enumerating the elements of $P(\bbA_1,\bbA_2)$ is now equivalent to enumerating maps $\phi_2:P\to {\bfX}_2$ such that $(P,\phi_1,\phi_2) \in P(\bbA_1,\bbA_2)$. 
For the following two lemmas it will also be helpful to fix a lift $\tilde{\phi}_1:P\to\widetilde{\bfX}$ of $\phi_1$, and letting $\widetilde{P} := \tilde{\phi}_1(P)$, $\tilde{A} := \tilde{\phi}_1(A)$ and $\tilde{\bbA}:=\tilde{\phi}_1(\bbA)$.

\begin{lem}\label{Asize}
	If $P(\bbA_1,\bbA_2)$ is non-empty, then it is in bijection with $H_{(\tilde{A})}/H_{(\widetilde{P})}$ -- where $H_{(\tilde{A})}$ is the pointwise stabiliser of $\tilde{A}$.
\end{lem}
\begin{proof}	
	The proof is very similar to Lemma \ref{leftFchoice}. Suppose $(P,\phi_1,\phi_2)\in P(\bbA_1,\bbA_2)$. We then get an admissibility diagram (\ref{Padmisible}) for some $h\in H$. Any other $(P,\phi_1,\phi'_2)\in P(\bbA_1,\bbA_2)$ also gives an admissibility diagram, but with $h$ replaced by $h'$ and $\tilde{\phi}_2$ replaced by $\tilde{\phi}'_2$. Since $\phi_2(\bbA)=\phi'_2(\bbA)=\bbA_2$, we know that $h(\tilde{\bbA})$ and $h'(\tilde{\bbA})$ differ by an element of $\Gamma_2$, so by composing $h'$ with an element of $\Gamma_2$ we may assume that $h(\tilde{\bbA})=h'(\tilde{\bbA})$. 
	This implies that $h'\in hH_{(\tilde{A})}$, and conversely any $h'\in hH_{(\tilde{A})}$ defines a map $\phi'_2$ with $(P,\phi_1,\phi'_2)\in P(\bbA_1,\bbA_2)$. 
	Finally, the map $\phi'_2$ only depends on the coset $h'G_{(\widetilde{P})}$, and so we obtain a bijection between the choices $\phi'_2$ and the quotient $H_{(\tilde{A})}/H_{(\widetilde{P})}$.	
\end{proof}

\begin{lem}\label{Bsize}
	The ratio $|A(\bbA_1,\bbS_2)|/\ell(\bbS_2)$ only depends on the colour $c$.
\end{lem}
\begin{proof}
	Let $\tilde{\bbS}_2\in\partial\widetilde{\bfX}$ be an oriented fin that covers $\bbS_2$, and let $\tilde{\bbA}_2$ be an oriented subarc of $\tilde{\bbS}_2$. 
	
	We claim that $\Psi_2(\tilde{\bbA}_2)\in A(\bbA_1,\bbS_2)$ if and only if $\tilde{\bbA}_2$ is a $H$-translate of $\tilde{\bbA}$. 
	Indeed, if $\bbA_2:=\Psi_2(\tilde{\bbA}_2)\in A(\bbA_1,\bbS_2)$, then there is an admissible $(P,\phi_1,\phi_2)$ with $\phi_2(\bbA)=\bbA_2$, and it has an associated admissibility diagram (\ref{Padmisible}) for some $h\in H$. Then $h(\tilde{\bbA})$ and $\tilde{\bbA}_2$ will both be a lifts of $\bbA_2$, so by composing $h$ with an element of $\Gamma_2$ we may assume that $h(\tilde{\bbA})=\tilde{\bbA}_2$. Conversely, if $h(\tilde{\bbA})=\tilde{\bbA}_2$ for some $h\in H$, then we get an admissibility diagram (\ref{Padmisible}), and $\Psi_2(\tilde{\bbA}_2)=\phi_2(\bbA)\in A(\bbA_1,\bbS_2)$.
	
	Thus the proportion of oriented subarcs of $\bbS_2$ that lie in $A(\bbA_1,\bbS_2)$ is equal to the proportion of oriented subarcs of $\tilde{S}_2$ that lie in the $H$-orbit of $\tilde{\bbS}$. In turn this is equal to the smallest positive translation length of elements of $H_{\wt{\bbS}}$. It follows that it is independent of the choice of $\tilde{\bbS}_2$ that covers $\bbS_2$, in fact it only depends on the $H$-orbit of $\tilde{\bbS}_2$, thus only depends on the colour of the oriented fin. 
\end{proof}

For $\hat{\bbS}\in\partial_\tto\widehat{\bfX}(\bbS_1,\bbS_2)$, the proportion of oriented subarcs $\bbA$ of $\hat{\bbS}$ that descend to $\bbA_1$ is $1/\ell(\bbS_1)$, and any such $\bbA$ must lie in some  $(P,\phi_1,\phi_2)\in\mathcal{P}$ that forms a piece of $\widehat{\bfX}$, with $\phi_1(\bbA)=\bbA_1$ and $\phi_2(\bbA)=\bbA_2$ some oriented subarc of $\bbS_2$. There are $\omega(P,\phi_1,\phi_2)$ copies of $(P,\phi_1,\phi_2)$ in $\widehat{\bfX}$, thus we can make the following computation:

\begin{align}\label{Kc}
\sum_{\hat{S}\in\partial_\tto\widehat{\bfX}(\bbS_1,\bbS_2)}\ell(\hat{\bbS})&=\ell(\bbS_1)\sum_{\substack{\bbA_2\subset \bbS_2\\ (P,\phi_1,\phi_2)\in P(\bbA_1,\bbA_2)}}\omega(P,\phi_1,\phi_2)\nonumber\\
&=\ell(\bbS_1)\sum_{\bbA_2\subset \bbS_2}|P(\bbA_1,\bbA_2)|\mu(H_{(\widetilde{P})})\nonumber\\
&=\ell(\bbS_1)|A(\bbA_1,\bbS_2)||H_{(\tilde{A})}:H_{(\widetilde{P})}|\mu(H_{(\widetilde{P})})&\text{by Lemma \ref{Asize}}\nonumber\\
&=\ell(\bbS_1)\ell(\bbS_2)\left(\frac{|A(\bbA_1,\bbS_2)|}{\ell(\bbS_2)}\right)\mu(H_{(\tilde{A})})\nonumber\\
&=K_{\tilde{\bbA}}\,\ell(\bbS_1)\ell(\bbS_2),
\end{align}

\noindent where $K_{\tilde{\bbA}}$ only depends on the $H$-orbit of $\tilde{\bbA}$ by Lemma \ref{Bsize}. The key point now is that the oriented fins in $\widetilde{\bfX}$ of colour $c$ are all in the same $H$-orbit. As a result, if we had chosen different $\bbS_1$ and $\bbS_2$ of colour $c$, then by choosing a suitable oriented subarc $\bbA_1$ in $\bbS_1$ we can arrange for the oriented subarc $\tilde{\bbA}$ to be in the same $H$-orbit as before. Thus $K_{\tilde{\bbA}}$ in fact only depends on $c$, and we will write it as $K_c$.

To complete the proof of equation (\ref{finequation}) it remains to compute a formula for $K_c$. We do this by summing (\ref{Kc}) over all $\bbS_1\in\partial_\tto \bfX_1$ and $\bbS_2\in\partial_\tto \bfX_2$ of colour $c$:

\begin{align*}
\sum_{\substack{\lambda_1(\bbS_1)=\lambda_2(\bbS_2)=c\\\hat{\bbS}\in\partial_\tto\widehat{\bfX}(\bbS_1,\bbS_2)}}\ell(\hat{\bbS})=K_c\left(\sum_{\lambda_1(\bbS_1)=c}\ell(\bbS_1)\right)\left(\sum_{\lambda_2(\bbS_2)=c}\ell(\bbS_2)\right)
\end{align*}
We can then substitute in the definition of the density $\rho_c$ (Definition \ref{density}), which will be the same for $\mathbf{X}_1$, $\mathbf{X}_2$ and $\widehat{\mathbf{X}}$ since they are all commensurable, to obtain:
\begin{equation*}
\rho_c|\widehat{X}|=K_c\rho_c|X_1|\rho_c|X_2|
\end{equation*}
This gives the required formula for $K_c$:
\begin{equation*}
K_c=\frac{|\widehat{X}|}{\rho_c|X_1||X_2|}
\end{equation*}

\bigskip
\section{Building graphs of spaces} \label{sec:GraphOfSpaces}

In this section and the next we prove Theorem \ref{thm:main}. In this section we build graphs of spaces for the two groups that share a number of properties, while in the next section we construct a common finite cover.

From now on let $G$ be the group from Theorem \ref{thm:main}, i.e. a group in $\mathscr{C}^\bullet$ that is hyperbolic relative to virtually abelian subgroups, and let $\psi:G\to G'$ be a fixed quasi-isometry to another finitely generated group.
Recall from the introduction that $\mathscr{C}^\bullet$ is quasi-isometrically rigid, and being hyperbolic relative to virtually abelian subgroups is preserved by quasi-isometry as well because of \cite[Theorem 1.6]{DrutuSapir05} and the quasi-isometric rigidity of abelian groups \cite{Pansu83}. Hence $G'$ is also a group in $\mathscr{C}^\bullet$ that is hyperbolic relative to virtually abelian subgroups.
Let $T$ and $T'$ be JSJ trees for $G$ and $G'$ and let $T_c$ and $T_c'$ be their associated trees of cylinders.
By Theorem \ref{thm:QIpreserveCyl}, there is an isomorphism $\hat{\psi}: T_c \rightarrow T_c'$ such that $\psi$ restricts to quasi-isometries $G_v\to G'_{\hat{\psi}(v)}$ and $G_e\to G'_{\hat{\psi}(e)}$ for $v\in VT_c$ and $e\in ET_c$. We know from Remark~\ref{rem:peripheralIsCylindrical} that the cylinder stabilisers for $G$ are virtually abelian, so if $T_c$ is just a single vertex then $G$ is either virtually free or virtually abelian, and such groups are already known to be quasi-isometrically rigid. So we may assume that $T_c$ is not a single vertex.

\begin{notation}
	Recall that the group $\scrG$ of Hausdorff equivalence classes of quasi-isometries $G\to G$ acts on the tree of cylinders $T_c$ by Corollary \ref{cor:scrGactTc}, and similarly $\scrG'$ acts on $T'_c$.	
	From now on it will be convenient to identify $\mathscr{G}$ with $\mathscr{G}'$ via the isomorphism $[f]\mapsto [\psi f \psi^{-1}]$, and to identify $T_c$ with $T'_c$ via the isomorphism $\hat \psi$ given by Theorem~\ref{thm:QIpreserveCyl}.
	
	$G$ acts on itself by left multiplication, so we have a homomorphism $G\to\scrG$.
	Any pair of edges that are not incident to a cylinder vertex in $T_c$ will have stabilizers with trivial intersection. 
	We can deduce from this that $G$ acts on $T_c$ faithfully and the homomorphism $G\to\scrG$ is injective, thus we can think of $G$ as a subgroup of $\scrG$. Similarly, $G'$ is a subgroup of $\scrG'$, so also a subgroup of $\scrG$ by the above identification (i.e. when we say $G'$ is a subgroup of $\mathscr{G}$ we mean the subgroup $\psi^{-1} G' \psi$). This means that $G'$ acts on the tree $T_c$; since this action is conjugate to the action of $G'$ on $T'_c$, we know that $T_c$ is a tree of cylinders for $G'$.
\end{notation}

\begin{notation}
	Recall that the tree of cylinders has a partition of the vertex set $VT_c=V_0 T_c\sqcup V_1 T_c$. $V_0T_c$ corresponds to vertices of $T$ that lie in more than one cylinder; in our case all vertex groups of the JSJ decomposition are rigid, so we will refer to vertices $u\in V_0 T_c$ as \emph{rigid vertices}. Note that the stabilisers of rigid vertices will all be virtually non-abelian free. The vertices $V_1 T_c$ correspond to cylinders in $T$; in Section \ref{sec:Cyl} we denoted cylinders by $Y\subset T$ or $Y\in V_1 T_c$, but as we no longer need to work with them as subtrees of $T$ we will instead denote them by $v\in V_1 T_c$, and refer to them as \emph{cylindrical vertices}. By Lemma \ref{lem:cylinderproduct} and Remark \ref{rem:peripheralIsCylindrical} the stabilisers of cylindrical vertices will be virtually $\mathbb{Z}$ or $\mathbb{Z}^2$.
\end{notation}

\subsection{Cylindrical factors and orientations}

By Theorem~\ref{thm:separableBalancedRelHyp}, we know that $G$ and $G'$ are both balanced.

\begin{lem}\label{lem:cylfactor}
	By passing to a finite index subgroup of $G$, we can assume that $G$ is torsion-free, and that for each $v\in V_1 T_c$ there is subgroup $\mathbb{Z}\cong\mathbb{Z}_v\leqslant G_v$, such that either $G_v = \mathbb{Z}_v$ or $G_v = \mathbb{Z}_v \times \mathbb{Z}$, and $G_e=\mathbb{Z}_v$ for any $e\in\lk(v)$. 
\end{lem}
\begin{proof}
	 By Lemma \ref{lem:cylinderproduct} we can pass to a finite index torsion-free subgroup of $G$ such that, for each $v\in V_1 T_c$ representing a cylinder $Y\subset T$, we get a product splitting $G_v = \mathbb{Z}_v$ or $G_v = \mathbb{Z}_v \times \mathbb{Z}$, where the $\mathbb{Z}_v$ factor pointwise fixes $Y$, and in the second case the second factor acts freely cocompactly on $Y$. For any $e\in\lk(v)$, $G_e$ will act elliptically on $T$, hence it will be a subgroup of $\mathbb{Z}_v$, but we also know that $\mathbb{Z}_v$ fixes $e\in EY$, so $G_e=\mathbb{Z}_v$.
\end{proof}

The subgroup $\mathbb{Z}_v$ from Lemma~\ref{lem:cylfactor} is called the \emph{cylindrical factor} of $G_v$. 
Similarly, we apply Lemma~\ref{lem:cylfactor} to $G'$ to make it torsion-free and get cylindrical factors $\mathbb{Z}'_v$ for the vertex stabilisers $G'_v$, where $v\in V_1 T_c$.

\begin{defn}(Oriented cylinders and edge groups)\\\label{defn:orientedcylinders}
	An \emph{orientation} $\cO$ on a cylindrical factor $\Zv$ is a choice of one of its two ends. The pair $(v,\cO)$ is called an \emph{oriented cylinder}. If $e\in\lk(v)$ then $G_e=\Zv$, so $\cO$ is also a choice of end of the subgroup $G_e$, and we call the pair $(e,\cO)$ an \emph{oriented edge group}.
	Let $\bar{\cO}$ denote the opposite end of $\Zv$.
	
	The group $\scrG$ of Hausdorff equivalence classes of quasi-isometries $G\to G$ acts on the tree of cylinders $T_c$ by Corollary \ref{cor:scrGactTc}. For $[f]\in\scrG$ that acts by $\hat{f}\in\Aut(T_c)$ we also have $f(G_e)\sim G_{\hat{f}(e)}$ for $e\in ET_c$ by Theorem \ref{thm:QIpreserveCyl}(3), hence $\scrG$ acts on the set of oriented edge groups, and also on the set of oriented cylinders. To avoid over-counting we will always consider edges with terminus in a cylindrical vertex, let $E_1 T_c\subset ET_c$ denote this set of edges. We denote $\scrG$-orbits using square brackets, so for $e\in E_1 T_c$ and $v\in V_1 T_c$ we have:
		\begin{align*}
		[e,\cO]:=\scrG\cdot(e,\cO)\\
		[v,\cO]:=\scrG\cdot(v,\cO)
		\end{align*}
		Let $\mathcal{C}$ denote the set of $\scrG$-orbits $[e,\cO]$ for $e\in E_1 T_c$, which we will think of as a colouring of the oriented edge groups. If $[f]\cdot(e_1,\cO_1)=(e_2,\cO_2)$ then
		\begin{equation}\label{actorientflip}
[f]\cdot(e_1,\bar{\cO}_1)=(e_2,\bar{\cO}_2),
		\end{equation}
		so for $c=[e,\cO]$ we can define $\bar{c}:=[e,\bar{\cO}]$.
\end{defn}

\begin{notation}\label{not:scrGorbits}
	We will also use square brackets to denote $\scrG$-orbits in $E_1T_c$ and $V_0 T_c$: $[e]:=\scrG\cdot e$ for $e\in E_1T_c$ and $[u]:=\scrG\cdot u$ for $u\in V_0 T_c$. Note that $[e,\cO]$ determines $[e]$.
\end{notation}

\subsection{Cylinder numbers and ratios}\label{sec:CylNumbers}

\begin{defn}(Cylinder numbers and ratios)\label{defn:torusratio}\\
	Let $v \in V_1T_c$ be a cylindrical vertex. For an orientation $\cO$ on $\Zv$ define
		$$\lk(v,\cO):=\{(e,\cO)\mid e\in\lk(v)\}.$$
	 For a colour $c\in\mathcal{C}$, define
	$$\lk(v,\cO,c):=\{(e,\cO)\mid e\in\lk(v), [e,\cO]=c\}.$$	
	
	The \emph{cylinder number} $t_c(v,\cO)$ is the number of $G_v$-orbits of oriented edges in $\lk(v,\cO,c)$.
	The \emph{cylinder ratio} of $(v,\cO)$ is
	\begin{equation*}
		t(v,\cO)=[c\mapsto t_c(v,\cO)],
	\end{equation*}
	where the brackets indicate that we only define the function up to rescaling. 
	Similarly, let $t'_c(v,\cO)$ be the number of $G'_v$-orbits of oriented edge groups in $\lk(v,\cO,c)$, and $t'(v,\cO)=[c\mapsto t'_c(v,\cO)]$.
\end{defn}

The motivation for cylinder numbers is the following, which will be made more precise later on. In a graph of spaces for the splitting of $G$ induced by $T_c$, we can take the vertex space for $G_v$ to be a circle or a torus, and the edge spaces for incident edge stabilisers to be circles; the orientation $\cO$ induces orientations on the edge spaces as 1-manifolds, so $\mathcal{C}$ gives a colouring of oriented edge spaces.
 The cylinder number $t_c(v,\cO)$ is just the number of edge spaces incident at the vertex space for $G_v$, with orientation induced by $\cO$, of colour $c$.

We note that  $v$ is a finite valence vertex in the case $G_v \cong \mathbb{Z}$, and by Lemma \ref{lem:cylfactor} the stabiliser $G_v$ fixes $\lk(v)$, so the cylinder number $t_c(v,\cO)$ is just the size of $\lk(v,\cO,c)$.
So in this case not only is $t(v,\cO) = t'(v,\cO)$, but $t_c(v,\cO) = t_c'(v,\cO)$.
In the case that $G_v \cong \mathbb{Z}^2$, although $t_c(v,\cO)$ is not in general equal to $t_c'(v,\cO)$, we will show that the cylinder ratios are in fact equal and that we can pass to finite index subgroups of $G$ and $G'$ such that the cylinder numbers are equal.

\begin{remk}\label{remk:cylstabpreserveO}
	Consider a cylindrical vertex stabiliser $G_v=\Zv\times\mathbb{Z}$ and fix an orientation $\cO$ on $\Zv$. For an edge $e\in\lk(v)$, any $g\in G_v$ maps $G_e=\mathbb{Z}_v\times\{0\}$ to some coset $\mathbb{Z}_v\times\{n\}$ by a translation of $\Zv\times\mathbb{Z}$, so the induced quasi-isometry $G_e\to G_{ge}=G_e$ is at bounded distance from the identity, and the orientation $\cO$ is preserved
	\begin{align}\label{cylstabpreserveO}
	g\cdot(e,\cO)=(ge,\cO)\\\nonumber
	g\cdot(v,\cO)=(v,\cO).
	\end{align}
	Equations (\ref{cylstabpreserveO}) also hold for $g\in G'_v$ by considering its action on $G'_v=\mathbb{Z}'_v\times\mathbb{Z}$. So $G_v$- and $G'_v$-orbits in $\lk(v,\cO)$ just correspond to $G_v$- and $G'_v$-orbits in $\lk(v)$.
\end{remk}

\begin{lem}\label{lem:torusratio}
	\begin{enumerate}
		\item $t(v,\cO)=t'(v,\cO)$ for all $\mathbb{Z}^2$ cylinders $v \in V_1 T_c$ with orientation $\cO$.
		\item $t(v,\cO)$ only depends on the $\mathscr{G}$-orbit $[v,\cO]$.
	\end{enumerate}	
\end{lem}
\begin{proof}
	We will only give a proof of (1), but (2) can be proven by the same argument applied to a quasi-isometry $[f]\in\mathscr{G}$ instead of the quasi-isometry $\psi:G\to G'$.
	
	Let $v \in V_1T_c$ with $G_v \cong \mathbb{Z}^2$.
	Since $\psi$ induces a quasi-isometry from $G_v$ to $G'_v$ it follows $G_v' \cong \mathbb{Z}^2$ also.
	If $e \in\lk(v)$, then $G_e$ is equal to the cylindrical factor $\mathbb{Z}_v \leqslant G_v$.
	Let $e_1, \ldots, e_N \in \lk(v)$ be $G_v$-orbit representatives of the edges.
	There exists $g \in G_v$ that corresponds to the generator of the second factor in the decomposition $G_v \cong \mathbb{Z}_v \times \mathbb{Z}$.
	It follows that $G_v = \bigcup_{k} g^k \mathbb{Z}_v$ and $\lk(v) = \bigcup_{i,k} g^k e_i$.
	Then we have a function $n: \lk(v) \rightarrow \mathbb{Z}$ given by $n(g^ke_i) = k$.
	Similarly for $G_v'$ we let $e'_1, \ldots, e'_{N'} \in \lk(v)$ be $G_v'$-orbit representatives, $g' \in G_v'$ be an element generating the second factor in the decomposition $G_v' = \mathbb{Z}'_v \times \mathbb{Z}$, to obtain $n' : \lk(v) \rightarrow \mathbb{Z}$ given by $n((g')^ke'_i) = k$.
	
	We now observe that there exists some $L >0$ such that $G(e_i) \sim_L \mathbb{Z}_v \times \{0\}$ for all $i$, where $G(e_i)$ is the coset corresponding to $e_i$ from Notation \ref{not:cosets}.
	By $G$-invariance of the metric, for all $e \in \lk(v)$ we have
	\begin{equation}\label{cosetcoset}
		G(e) \sim_L \mathbb{Z}_v \times \{n(e)\}.
	\end{equation}

	We now consider the following five pseudo-metrics on $\rm{lk}(v)$.
	\begin{align}\label{lkYmetrics}
		d(e_1,e_2)&=\begin{dcases*}
			|n(e_1)-n(e_2)| &(a)\\
			d_H(G(e_1),G(e_2)) &(b)\\
			d_H(\psi(G(e_1)),\psi(G(e_2))) &(c)\\
			d_H(G'(e_1),G'(e_2))&(d)\\
			|n'(e_1)-n'(e_2)|&(e)
		\end{dcases*}
	\end{align}\\
	\begin{claim}
		Pseudo metrics (a)-(e) are all equivalent up to quasi-isometry.
	\end{claim}\\
	\begin{claimproof}
		With respect to the standard generators of $\mathbb{Z}^2$, the Hausdorff distance $d_H(\mathbb{Z}\times\{n_1\},\mathbb{Z}\times\{n_2\})$ is simply $|n_1-n_2|$. Cylinder stabilisers are quasi-isometrically embedded in $G$ (peripheral subgroups are always quasi-convex~\cite[Lemma 4.15]{DrutuSapir05}), and quasi-isometries coarsely preserve Hausdorff distance between subsets, so (\ref{cosetcoset}) implies that metrics (a) and (b) are equivalent. Similarly, (d) and (e) are equivalent. (b) and (c) are equivalent because $\psi$ is a quasi-isometry, and finally (c) and (d) are equivalent precisely because of Theorem \ref{thm:QIpreserveCyl}(2).
	\end{claimproof}\\
	
	The maps $n,n':\rm{lk}(v)\to\mathbb{Z}$ are both surjective, so the equivalence of metrics (a) and (e) gives us a quasi-isometry $\nu:\mathbb{Z}\to\mathbb{Z}$ such that
	\begin{equation}\label{nu}
		n'\approx \nu\circ n.
	\end{equation}  
	After perturbing $\nu$ by bounded distance, we can assume that it is monotonic; indeed, if $\lim_{i\to\pm\infty}\nu(i)=\pm\infty$ then $i\mapsto \max_{j\leq i} \nu(j)$ is increasing and at bounded distance from $\nu$.
	
	For each $a\in\mathbb{Z}$, $n^{-1}(a)$ has one edge from each $G_v$-orbit, and so by Remark \ref{remk:cylstabpreserveO} we have that $|\{(e,\cO)\in\rm{lk}(v,\cO,c)\mid n(e)=a\}|$ is equal to $t_c(v,\cO)$ from Definition \ref{defn:torusratio}. For $c_1,c_2\in\mathcal{C}$ (with $\rm{lk}(v,\cO,c_i)\neq\emptyset$) we have
	\begin{equation}\label{torusratioeq}
		\frac{t_{c_1}(v,\cO)}{t_{c_2}(v,\cO)}=\lim_{b-a\to\infty}\frac{|\{(e,\cO)\in\rm{lk}(v,\cO,c_1)\mid n(e)\in[a,b]\}|}{|\{(e,\cO)\in\rm{lk}(v,\cO,c_2)\mid n(e)\in[a,b]\}|},
	\end{equation}
	and a similar equation holds for $t'_{c_i}$. By (\ref{nu}), and the fact that $\nu$ is monotonic, we see that
	\begin{equation}\label{nn'ratio}
		\lim_{b-a\to\infty}\frac{|\{(e,\cO)\in\rm{lk}(v,\cO,c)\mid n(e)\in[a,b]\}|}{|\{(e,\cO)\in\rm{lk}(v,\cO,c)\mid n'(e)\in[\nu(a),\nu(b)]\}|}=1
	\end{equation}
	for any $c\in\mathcal{C}$. Combining (\ref{torusratioeq}) with (\ref{nn'ratio}) we deduce that $t_{c_1}(v,\cO)/t_{c_2}(v,\cO)=t'_{c_1}(v,\cO)/t'_{c_2}(v,\cO)$. Hence $t(v,\cO)=t'(v,\cO)$ as required.
\end{proof}
\bigskip

\noindent The next task is to prove the following lemma.

\begin{lem}\label{lem:torusmatch}
	There exist finite index subgroups $\hat{G}\trianglelefteq G$ and $\hat{G}'\trianglelefteq G'$ and integers $N_c [v,\cO]$ such that $N_c [v,\cO]=N_{\bar{c}}[v,\bar{\cO}]$, and for each oriented cylinder $(v,\cO)$ the cylinder numbers of $\hat{G}$ and $\hat{G}'$ both equal the numbers $N_c [v,\cO]$: $$\hat{t}_c(v,\cO)=\hat{t}'_c(v,\cO)=N_c [v,\cO].$$
\end{lem}

\noindent We will need the following remark.

\begin{remk}\label{remk:linkflipmap}
	For an oriented cylinder $(v,\cO)$ and a colour $c\in\mathcal{C}$, we have a bijection
	\begin{align}\label{cyllinkflip}
	\lk(v,\cO,c)&\to\lk(v,\bar{\cO},\bar{c})\\\nonumber
	(e,\cO)&\mapsto(e,\bar{\cO}).
	\end{align}
	Remark \ref{remk:cylstabpreserveO} tells us that $G_v$ acts on both $\lk(v,\cO,c)$ and $\lk(v,\bar{\cO},\bar{c})$, so by (\ref{actorientflip}) we know that the map (\ref{cyllinkflip}) is $G_v$-equivariant. It follows that
	\begin{equation}
t_c(v,\cO)=t_{\bar{c}}(v,\bar{\cO}).\label{cylnumberflip}
	\end{equation}
\end{remk}

As discussed earlier, if a cylindrical vertex stabiliser is cyclic, then the cylinder numbers are already equal, and will be stable under passing to finite index subgroups, so Lemma \ref{lem:torusmatch} is all about modifying the $\mathbb{Z}^2$ cylinders.
If $v \in V_1T_c$ is a cylindrical vertex with stabiliser $G_v \cong \mathbb{Z}_v \times \mathbb{Z}$, then let $\pi_v: G_v\to\mathbb{Z}$ denote the projection onto the second factor in the product decomposition. 
Any finite index subgroup $\hat{G} \trianglelefteq G$ will have stabiliser $\hat{G}_v$ finite index in $G_v$. 
Suppose for a moment we have  $\hat{G} \trianglelefteq G$ finite index such that
 $\pi_v(\hat{G}_v) = N\mathbb{Z}$. 
Then each $G_v$-orbit of edges in $\rm{lk}(v)$ would split into $N$ many  $\hat{G}_v$-orbits, so by Remark \ref{remk:cylstabpreserveO} the cylinder numbers for $\hat{G}$ and $G$ would be related by 
\begin{equation}\label{cylnumberscale}
\hat{t}_c(v,\cO)=Nt_c(v,\cO).
\end{equation}

It follows readily from Definition \ref{defn:torusratio} that $t_c(v,\cO)$ only depends on $c$ and the $G$-orbit of $(v,\cO)$, hence there are only finitely many cylinder numbers. Furthermore, Lemma \ref{lem:torusratio} says that the cylinder ratio $t(v,\cO)$ only depends on the $\mathscr{G}$-orbit $[v,\cO]$.
Therefore, for each $[v,\cO]$ we can pick numbers $N_c [v,\cO]$ that are in the ratio $t(v,\cO)$, and that are common multiples of all cylinder numbers. By (\ref{cylnumberflip}) we can assume that $N_c[v,\cO]=N_{\bar{c}}[v,\bar{\cO}]$. Again by (\ref{cylnumberflip}), we deduce that
\begin{equation}\label{torusmatch}
N_v:=\frac{N_c[v,\cO]}{t_c(v,\cO)}=\frac{N_{\bar{c}}[v,\bar{\cO}]}{t_{\bar{c}}(v,\bar{\cO})}
\end{equation}
only depends on $v$, in fact it only depends on the $G$-orbit of $v$ because $t_c(v,\cO)$ only depends on $c$ and the $G$-orbit of $(v,\cO)$. We define integers $N'_v$ similarly, such that $N'_v$ only depends on the $G'$-orbit of $v$.

By (\ref{cylnumberscale}) and (\ref{torusmatch}), Lemma \ref{lem:torusmatch} will follow if we can construct finite index $\hat{G}\trianglelefteq G$ and $\hat{G}'\trianglelefteq G'$ such that 
\begin{align}\label{torusstretch}
	\pi_v(\hat{G}_v)=N_v\mathbb{Z}\nonumber\\
	\pi_v(\hat{G}'_v)=N'_v\mathbb{Z}
\end{align} 
for each $v$. 
Note that it is enough to have (\ref{torusstretch}) hold for a set of $G$-orbit representatives of $v \in V_1T_c$ with $G_v \cong \mathbb{Z}^2$ (and similarly for $G'$) because $\hat{G}$ is normal in $G$ (and because each map $\pi_v:G_v\to\mathbb{Z}$ is determined by the edge stabilisers incident to $G_v$, up to a factor of $\pm1$, so they are preserved by conjugation in $G$). During this construction, we are allowed to multiply all the $N_v,N'_v, N_c [v,\cO]$ by some fixed constant, as this preserves equation (\ref{torusmatch}), or in other words we are allowed to assume that they are multiples of any given finite set of integers.

This construction could be done in an elementary way by building explicit finite covers of graphs of spaces, but for a cleaner approach we will make use of the relatively hyperbolic version of the Malnormal Special Quotient Theorem, which is as follows.

\begin{thm}(Einstein \cite[Theorem 2]{Ein19})\\ \label{thm:RHMSQT}
	Let $G$ be a virtually special group that is hyperbolic relative to subgroups $\{P_1, \ldots, P_m \}$.
	Then there exist finite index subgroups $\dot{P}_i \trianglelefteq P_i$, such that for any further finite index subgroups $\ddot{P}_i \trianglelefteq \dot{P}_i$, the quotient  $G / \llangle \ddot{P}_1, \ldots, \ddot{P}_m \rrangle$ is hyperbolic and virtually special.
\end{thm}

\begin{proof}[Proof of Lemma \ref{lem:torusmatch}]
	We apply Theorem \ref{thm:RHMSQT} to $G$ with peripheral subgroups $\{P_1, \ldots, P_m \}$ being the stabilisers of a set of $G$-orbit representatives $\{v_1,\ldots,v_m\}$ of cylinder vertices in $T_c$. Note that $G$ is virtually special by Theorem~\ref{thm:separableBalancedRelHyp} and it is hyperbolic relative to its cylinder stabilisers by Proposition \ref{prop:hyprelcylinders}. Moreover, by Remark \ref{rem:RemovingCyclicPeripherals} we may disregard any cyclic peripheral subgroups and assume that each $P_i$ is isomorphic to $\mathbb{Z}^2$.
	We now wish to find finite index subgroups $\ddot{P}_i \trianglelefteq \dot{P}_i$ such that the following two properties hold:
	\begin{enumerate}		
		\item The induced maps $P_i/\ddot{P_i}\to G/\llangle \ddot{P}_1,\ldots,\ddot{P}_m\rrangle$ are injections.
		\item $\pi_{v_i}(\ddot{P}_i)=N_{v_i}\mathbb{Z}$
	\end{enumerate}
    \cite[Theorem 1.1 (1)]{Osin07} tells us that property (1) holds provided that the subgroups $\ddot{P}_i \trianglelefteq \dot{P}_i$ miss a given finite set $\mathfrak{F}$ of non-trivial elements of $G$. This is easy to arrange since the subgroups $P_i$ are  residually finite. Suppose that after arranging property (1) we have that $\pi_{v_i}(\ddot{P}_i)=N_i\mathbb{Z}$. As discussed above, we may assume that $N_{v_i}$ is a multiple of $N_i$ for each $i$, so we can arrange property (2) by replacing each $\ddot{P}_i$ with $\ddot{P}_i\cap\pi_{v_i}^{-1}(N_{v_i}\mathbb{Z})$.
	
	We then define $\bar{G}:=G/\llangle \ddot{P}_1,\ldots,\ddot{P}_m\rrangle$. 
	Theorem \ref{thm:RHMSQT} implies that $\bar{G}$ is virtually special, hence it has a finite index, torsion-free, normal subgroup $\hat{\bar{G}}\trianglelefteq\bar{G}$. 
	Set $\hat{G}$ to be the preimage of $\hat{\bar{G}}$ under the quotient map $G\to\bar{G}$. The image of a peripheral subgroup $P_i$ in $\bar{G}$ is finite, so has trivial intersection with $\hat{\bar{G}}$. 
	Property (1) then implies that $\hat{G}\cap P_i=\ddot{P}_i$ for each $i$. And so property (2) tells us that $\hat{G}$ satisfies the first equation of (\ref{torusstretch}).
	
	By the same argument, there exists $\hat{G}'\trianglelefteq G'$ finite index that satisfies the second equation of (\ref{torusstretch}).
\end{proof}

By Lemma \ref{lem:torusmatch}, we can assume going forward that for each oriented cylinder $(v,\cO)$ the cylinder numbers of $G$ and $G'$ both equal the numbers $N_c [v,\cO]$: 
\begin{equation}\label{cylnumbermatch}
t_c(v,\cO)=t'_c(v,\cO)=N_c [v,\cO].
\end{equation}

\subsection{A tree of trees with fins}\label{sec:treeoftrees}

For each rigid vertex $u \in V_0 T_c$, recall that the incident edge groups for the stabiliser $G_u$ induce a line pattern $\mathcal{L}_u$ (Definition \ref{vertexlinepat}). By Lemma \ref{lem:RigidImpliesRigid} this line pattern will be rigid, and so by Theorem \ref{rigidequiv} there is a quasi-isometry to a tree with line pattern that is a rigid model space:
\[
\alpha_u : (G_u, \calL_u) \rightarrow (Y_u, \calL_u).
\]
Recall Lemma \ref{lem:restrictingToRigidVertices}, which says that any $[f]\in\scrG$ induces a $\approx$-class of quasi-isometries 
\begin{equation}\label{linepatternmap}
[f]_u : (G_u,\calL_u) \rightarrow (G_{\hat{f}(u)},\calL_{\hat{f}(u)})
\end{equation} 
that respect line patterns.
So for each $\scrG$-orbit of vertices $u$, the free groups with line patterns $(G_u, \calL_u)$ are all quasi-isometric, and hence we may choose the  rigid model spaces $(Y_u, \calL_u)$ to be isometric.
We can encode the line pattern $\calL_u$ in the tree $Y_u$ as a set of fins to obtain a quasi-isometry to a graph with fins (see Definition \ref{graphcolfin}):
\[
\beta_u : (G_u, \calL_u) \rightarrow (\bfY_u, \partial \bfY_u).
\]                                                                                 
Since the underlying graph $Y_u$ is a tree, we will refer to $(\bfY_u, \partial \bfY_u)$ as a \emph{tree with fins}. Note that $(\bfY_u, \partial \bfY_u)$ also serves as a rigid model space, and its group of isometries is precisely its automorphism group in the sense of Definition \ref{defn:covergraphfin}. Moreover, the isometry type of the rigid model space $(Y_u, \calL_u)$ only depends on the $\scrG$-orbit $[u]$, and so the isomorphism type of the tree with fins $(\bfY_u, \partial \bfY_u)$ also just depends on $[u]$. Combining these two facts with (\ref{linepatternmap}) yields the following lemma.

\begin{lem}\label{lem:treefinmap}
For each $[f] \in \mathscr{G}$ and $u \in V_0 T_c$, there is a unique isomorphism  
$$\morp{f}_u : (\bfY_u, \partial \bfY_u) \rightarrow (\bfY_{\hat{f}(u)}, \partial \bfY_{\hat{f}(u)})$$ 
such that $\morp{f}_u \approx \beta_{\hat{f}(u)} \circ [f]_u \circ \beta_u^{-1}$. 
\end{lem} 
\begin{proof}
We know that $u$ and $\hat{f}(u)$ are in the same $\scrG$-orbit, so $(\bfY_u, \partial \bfY_u)$ and $(\bfY_{\hat{f}(u)}, \partial \bfY_{\hat{f}(u)})$ are isomorphic. As these trees with fins are rigid model spaces, the line-pattern-preserving quasi-isometry (or more precisely $\approx$-class of quasi-isometries) $\beta_{\hat{f}(u)} \circ [f]_u \circ \beta_u^{-1}:(\bfY_u, \partial \bfY_u) \rightarrow (\bfY_{\hat{f}(u)}, \partial \bfY_{\hat{f}(u)})$ between them is finite Hausdorff distance from a unique isometry $\morp{f}_u$.
\end{proof}

This gives us the data to define an action of $\scrG$ on the disjoint union of the $\bfY_u$ -- which we think of as a ``tree of trees with fins''.

\begin{lem}\label{lem:scrGfinaction}
	The maps $\morp{f}_u$ define an action of $\scrG$ on the graph with fins $\bfY:=\sqcup_{u\in V_0 T_c} \bfY_u$.
\end{lem}
\begin{proof}
	We know from Corollary \ref{cor:scrGactTc} that $\scrG$ acts on the rigid vertices $V_0 T_c$. It follows from Lemma \ref{lem:treefinmap} that we have a well-defined map $\scrG\to\Aut(\bfY)$, we must show that this is a homomorphism. It is clear that $\rm{id}_G$ maps to the identity, so it remains to show that this map respects composition. Let $[f_1],[f_2]\in\scrG$ with $\hat{f}_1(u_1)=u_2$ and $\hat{f}_2(u_2)=u_3$. We know that $[f_2\circ f_1]_{u_1}\approx [f_2]_{u_2}\circ [f_1]_{u_1}$ as these maps come from restricting the quasi-isometries to the vertex groups, so it follows from Lemma \ref{lem:treefinmap} that $\morp{f_2\circ f_1}_{u_1}\approx\morp{f_2}_{u_2}\circ\morp{f_1}_{u_1}$, but this second $\approx$ must be an equality since both sides are isometries between rigid model spaces.
\end{proof}

For $u\in V_0 T_c$ we know that the lines in $\calL_u$ correspond to the incident edge stabilisers $G_e$, and this is a one-to-one correspondence because no two incident edge stabilisers are commensurable in $G_u$ (as they come from different cylinders). In turn these lines correspond via $\beta_u$ to the fins of $\bfY_u$. Let $S_e\in\partial\bfY_u$ be the fin corresponding to $G_e$ (with $\iota(e)=u$). A choice of end $\cO$ of the edge stabiliser $G_e$ defines an oriented edge group $(e,\cO)$, which will correspond via $\beta_u$ to a choice of end of the fin $S_e$, or equivalently a choice of orientation $\bbS_e=(S_e,\tto)$ of the fin as a 1-manifold, as in Definition \ref{graphcolfin}. It follows from the way we defined the $\scrG$-action on $\bfY$ that the action of $\scrG$ on oriented edge groups is conjugate to the action of $\scrG$ on oriented fins in  $\bfY$. We defined $\mathcal{C}$ to be the set of $\scrG$-orbits of oriented edge groups, so this also corresponds to $\scrG$-orbits of oriented fins, which we will think of as a colouring of the oriented fins $\lambda:\partial_\tto\bfY\to\mathcal{C}$. This makes $\bfY$ and each of the $\bfY_u$ into graphs with coloured fins, and the $\scrG$-action obviously preserves colours.

\begin{remk}\label{remk:GuG'uaction}
	For a rigid vertex $u\in V_0 T_c$ the action of $\scrG$ on $Y$ restricts to an action of $G_u$ on $\bfY_u$, where $g\in G_u$ acts by $\morp{g}_u$. It follows from the definition of $\morp{g}_u$ that this action of $G_u$ on $\bfY_u$ is the $\beta_u$-conjugacy action in the sense of Definition \ref{phiconj}. Similarly, the quasi-isometry $\psi:G\to G'$ restricts to a quasi-isometry $\psi:G_u\to G'_u$, and the action of $G'_u$ on $\bfY_u$ is the $\beta_u \psi^{-1}$-conjugacy action. It then follows from Lemma \ref{freeco} that the actions of $G_u$ and $G'_u$ on $\bfY$ are free and cocompact, that $\beta_u:G_u\to\bfY_u$ is Hausdorff equivalent to any orbit map of $G_u$, and that $\beta_u\psi^{-1}:G'_u\to\bfY_u$ is Hausdorff equivalent to any orbit map of $G'_u$.
\end{remk}

\begin{remk}
	The space $\bfY$ is disconnected, so it is tempting to try and connect it up into some simply connected metric space that's quasi-isometric to $G$, and that admits an action of $\scrG$ quasi-conjugate to its action on $G$. The natural way to try and do this is to take a copy of $\mathbb{R}$ or $\mathbb{R}^2$ for each cylindrical vertex $v\in V_1 T$, and glue them to the appropriate fins in $\bfY$ according to how the edge stabilisers $G_e$ embed in the vertex stabilisers $G_v$.
	There is no real advantage in doing this however, because the action of $\mathscr{G}$ would not be isometric -- it would induce isometries between the vertex spaces as it does for $\bfY$, but in general it would act via ``shearing'' maps between the edge spaces.
	Such a construction was used however in~\cite{BehrstockNeumann08}.
\end{remk}

\subsection{Stretch ratio}\label{sec:stretchratio}

\begin{defn}(Stretch ratio)\\\label{defn:stretchratio}                                                                                                                                           
	Let $v \in V_1T_c$ be a cylindrical vertex and let $g \in \Zv$ be a non-trivial element.
	Let $e \in E_1T_c$ be an edge with $\tau(e) = v$ and $\iota(e) = u \in V_0T_c$, then the automorphism
	$${\morp{g}_u : (\bfY_u , \partial \bfY_u) \rightarrow (\bfY_u , \partial \bfY_u)},$$
	acts by translation on the fin $S_e$.
	
    Let $r_e$ be the translation length of $\morp{g}_u$, which is equal to the distance that it translates along the fin $S_e$. Note that $r_e\neq 0$.
    
 The \emph{stretch ratio} of $v \in V_1T_c$ is the function $\lk(v)\to\mathbb{Q}$ given by $e\mapsto r_e$ determined by $g \in \Zv$, but as we are only interested in the ratio between the $r_e$ terms we will only consider this function to be defined up to scaling. 
 We will denote this equivalence class of functions by
 \[
  \Str(v)=[e\mapsto r_e].
 \]
\end{defn}
  
The stretch ratio does not depend on the choice of non-trivial element $g \in\Zv$, since each element is a power of a fixed generator, and the translation lengths scale linearly by the power. 

We can also define the stretch ratio for $v \in V_1T_c$ with respect to $G'$ by using elements $g'\in\Zv'$.
It is a result of Cashen-Martin~\cite{CashMar} that the stretch ratios defined using $G$ and $G'$ will coincide.
Their result is more general, but the two consequences that will be relevant to us are the following. We also include a proof because the result is slightly simpler in our setting, and it highlights how we make use of rigid model spaces.

\begin{lem}(Cashen-Martin \cite[Proposition 5.14]{CashMar})\label{stretch}

	\begin{enumerate}
		\item The stretch ratio $\Str(v)$ is the same for $G$ and $G'$.
		\item There exist integers $r_{[e]}$ for $e\in E_1T_c$, where $[e]$ denotes the $\scrG$-orbit of $e$, such that Str$(v)=[e\mapsto r_{[e]}]$ for all $v\in V_1 T_c$.
	\end{enumerate}	
\end{lem}

\noindent We recall that a \emph{coarse $M$-similitude} is a function $f: X \rightarrow Y$ between metric spaces such that
\[
 M d_X(x_1, x_2) - \epsilon \leq d_Y(f(x_1), f(x_2)) \leqslant M d_X(x_1, x_2) + \epsilon
\]
for all $x_1, x_2 \in X$ and some fixed $\epsilon \geq 0$. We make four remarks about such an $f$:
\begin{itemize}
	\item Any map Hausdorff equivalent to $f$ will also be a coarse $M$-similitude.	
	\item If $f:X\to Y$ is a quasi-isometry, then its quasi-inverse $f^{-1}$ will be a coarse $M^{-1}$-similitude.
	\item If $g:Y\to Z$ is a coarse $N$-similitude, then $g\circ f$ is a coarse $MN$-similitude.
	\item An equivariant quasi-isometry of $\mathbb{Z}$ into a tree will be a coarse $M$-similitude, where $M$ is determined by the translation length along the axis.
\end{itemize}

\begin{proof}[Proof of Lemma \ref{stretch}]
	Let $e \in E_1T_c$ be an edge with $\iota(e) = u\in V_0 T_c$ and $\tau(e) = v\in V_1 T_c$. We know from Remark \ref{remk:GuG'uaction} that $\beta_u:G_u\to\bfY_u$ is Hausdorff equivalent to any orbit map of $G_u$. We also know that $\Zv=G_e \leqslant G_u$ acts on $\bfY_u$ by translating along the fin $S_e$, say the translation length of a generator is $r_e$, so it follows that (up to Hausdorff equivalence) $\beta_u$ restricts to a coarse $r_e$-similitude $\Zv\to S_e$.
	
    Similarly, we know from Remark \ref{remk:GuG'uaction} that $\beta_u \circ \psi^{-1}:G_u'\to\bfY_u$ is Hausdorff equivalent to any orbit map of $G'_u$, and that $\Zv'=G'_e \leqslant G'_u$ acts on $\bfY_u$ by translating along the fin $S_e$, with translation length of a generator being $r'_e$ say. So it follows that (up to Hausdorff equivalence) $\beta_u \psi^{-1}$ restricts to a coarse $r'_e$-similitude $\Zv'\to S_e$. Composing the two coarse similitudes tells us that $\psi:\Zv\to\Zv'$ is a coarse $r_e/r'_e$-similitude. But the map $\psi:\Zv\to\Zv'$ doesn't depend on the choice of $e$, so the ratio $r_e/r'_e$ is the same for all edges $e\in\lk(v)$ -- thus proving (1).
    
    For (2), we must show that the action of $\scrG$ preserves stretch ratio. More precisely, if $ [f] \in\mathscr{G}$ and Str$(v)=[e\mapsto r_e]$, then we must show that 
    \begin{equation}\label{scrGpreservestretch}
\Str(\hat{ f }(v)) = [\hat{ f }(e)\mapsto r_e\mid e\in\rm{lk}(v)].
    \end{equation}
    
    Observe that, for $e\in\lk(v)$ with $\iota(e)=u$, we have the following diagram that commutes up to Hausdorff equivalence.

   \begin{equation}
	\begin{tikzcd}[
	ar symbol/.style = {draw=none,"#1" description,sloped},
	isomorphic/.style = {ar symbol={\cong}},
	equals/.style = {ar symbol={=}},
	subset/.style = {ar symbol={\subset}}
	]
	\Zv\ar{d}{ f }\ar[hook]{r}&G_{u}\ar{d}{ f }\ar{r}{\beta_{u}}&\bfY_{u} \ar{d}{ \morp{f}_{u}}\\
	\mathbb{Z}_{\hat{f}(v)}\ar[hook]{r} & G_{\hat{f}(u)} \ar{r}{\beta_{\hat{f}(u)}} & \bfY_{\hat{f}(v)}                                                              
	\end{tikzcd}
	\end{equation}
	
	We know that $\mathbb{Z}_{\hat{f}(v)}$ acts on $\bfY_{\hat{f}(v)} $ by translating along the fin $S_{\hat{f}(e)}$, with the translation length of a generator being $r_{\hat{f}(e)}$ say. And as before $\beta_{\hat{f}(u)}$ restricts to a coarse $r_{\hat{f}(e)}$-similitude $\mathbb{Z}_{\hat{f}(v)}\to S_{\hat{f}(e)}$. But we know that $\morp{f}_u$ restricts to an isometry $S_e\to S_{\hat{f}(e)}$, so composing coarse similitudes implies that $f:\Zv\to\mathbb{Z}_{\hat{f}(v)}$ is a coarse $r_e/r_{\hat{f}(e)}$-similitude. As before we note that the ratio $r_e/r_{\hat{f}(e)}$ must be the same for all edges $e\in\lk(v)$, which completes the proof of (\ref{scrGpreservestretch}).
\end{proof}

\begin{remk}
	The $\scrG$-invariance of the cylinder ratios and stretch ratios coming from Lemmas \ref{lem:torusratio} and \ref{stretch} can be interpreted in terms of the geometry of quasi-isometries between $\mathbb{Z}^2$ cylindrical vertex groups. It implies that a quasi-isometry $[f]\in\scrG$ with $\hat{f}(v_1)=v_2\in V_1 T_c$ restricts to a quasi-isometry $G_{v_1}=\mathbb{Z}_{v_1}\times\mathbb{Z}\to G_{v_2}=\mathbb{Z}_{v_2}\times\mathbb{Z}$ that (up to Hausdorff equivalence) sends cosets of $\mathbb{Z}_{v_1}$ to cosets of $\mathbb{Z}_{v_2}$, stretching each of them by coarse similitudes of the same factor (because stretch ratios are preserved), and the induced map between the second factors $\pi_{v_2}\circ f:\{0\}\times\mathbb{Z}\to\{0\}\times\mathbb{Z}$ is also a coarse similitude (because torus ratios are preserved), where $\pi_{v_2}:\mathbb{Z}_{v_2}\times\mathbb{Z}\to\mathbb{Z}$ is projection to the second factor. Moreover, the factors of these coarse similitudes are determined by $v$ and $v'$. This means that there is not much choice for $f:G_{v_1}\to G_{v_2}$ up to Hausdorff equivalence, in fact it is determined by the Hausdorff class of the map $\chi_{v_2}\circ f:\{0\}\times\mathbb{Z}\to\mathbb{Z}_{v_2}$, where $\chi_{v_2}:\mathbb{Z}_{v_2}\times\mathbb{Z}\to\mathbb{Z}_{v_2}$ is projection to the first factor. These observations are not important for the proof of our theorem, so we give no further explanations.
\end{remk}

\subsection{Constructing graphs of spaces from graphs with fins}\label{sec:graphsofspacesforGG'}

Consider the graphs of groups $(G,\Gamma)$ and $(G',\Gamma')$ for $G$ and $G'$ given by their respective actions on  $T_c \cong T_c'$.
The vertices in $\Gamma$ and $\Gamma'$ are either \emph{rigid}
or \emph{cylindrical} according to their lifts in $T_c$, and have corresponding vertex partitions $V\Gamma =V_0\Gamma \sqcup V_1 \Gamma$ and $V\Gamma'=V_0\Gamma' \sqcup V_1\Gamma'$. As for $T_c$, we always consider edges with terminus a cylindrical vertex, and we write $E_1\Gamma$ and $E_1\Gamma'$ for the sets of these edges. We also colour edges and rigid vertices according to the $\scrG$-orbits of their lifts in $T_c$, so we write $[e]:=[\tilde{e}]$ for $e\in E_1\Gamma\sqcup E_1\Gamma'$ with lift $\tilde{e}\in E_1T_c$, and $[u]:=[\tilde{u}]$ for $u\in V_0\Gamma\sqcup V_0\Gamma'$ with lift $\tilde{u}\in V_0 T_c$ (using Notation \ref{not:scrGorbits}).

We now build graphs of spaces $(\mathcal{X},\Gamma)$ and $(\mathcal{X}',\Gamma')$,  for $(G,\Gamma)$ and $(G',\Gamma')$ respectively, following the conventions given in Section~\ref{sec:BassSerre}. 

\begin{defn}\label{defn:GoScalXcalX'}
	(Graphs of spaces $(\mathcal{X},\Gamma)$ and $(\mathcal{X}',\Gamma')$)\\

For each rigid vertex $u\in V_0\Gamma$, take a lift $\tilde{u} \in  V_0 T_c$, and consider the action of $G_{\tilde{u}}$ on its corresponding tree with coloured fins $\mathbf{Y}_{\tilde{u}}$ as described in Remark \ref{remk:GuG'uaction}. This action is free and cocompact, so the quotient $\mathbf{X}_u:=\mathbf{Y}_{\tilde u}/G_{\tilde{u}}$ is a finite graph with coloured fins. The colouring $\lambda:\partial_\tto\bfY_{\tilde{u}}\to\mathcal{C}$ descends to a colouring $\lambda:\partial_\tto\bfX_u\to\mathcal{C}$. 
The fundamental group $\pi_1 \bfX_u$ is identified with the deck transformations $G_{\tilde{u}}$ of the covering $\bfY_{\tilde{u}}\to\bfX_u$, which in turn is identified with the vertex group $G_u$ of $(G,\Gamma)$.
We let $\mathcal{X}_u = \bfX_u$.

This is independent of the choice of lift $\tilde{u}$, because if $\tilde{u}_1$ and $\tilde{u}_2$ are two lifts of $u$ with $g(\tilde{u}_1)=\tilde{u}_2$ ($g\in G$), then $\morp{g}_{\tilde{u}_1}:\bfY_{\tilde{u}_1}\to\bfY_{\tilde{u}_2}$ is an isomorphism that is equivariant with respect to the actions of $G_{\tilde{u}_1}$ and $G_{\tilde{u}_2}$ respectively via the conjugation $h\in G_{\tilde{u}_1}\mapsto ghg^{-1}\in G_{\tilde{u}_2}$.

For each cylindrical vertex $v\in V_1\Gamma$ we let $\mathcal{X}_v$ be homeomorphic to a circle $S^1$ if $G_v \cong \mathbb{Z}$ or a torus $S^1 \times S^1$ if $G_v \cong \mathbb{Z}^2$ and identify $\pi_1 \mathcal{X}_v$ with $G_v$. We have $G_v\cong G_{\tilde{v}}=\mathbb{Z}_{\tilde{v}}\times\mathbb{Z}$ or $\mathbb{Z}_{\tilde{v}}$ for any lift $\tilde{v}\in V_1 T_c$ of $v$, and since the cylindrical factor $\mathbb{Z}_{\tilde{v}}$ is preserved by $G$-conjugation we can define the \emph{cylindrical factor} $\Zv\leqslant G_v$. We then fix a \emph{cylindrical fibre} $S_v \subseteq \calX_v$, a subspace homeomorphic to a circle whose embedding gives the embedding of the cylindrical factor.
Note that in the case $G_v \cong \mathbb{Z}$ we have $S_v = \calX_v$.

Let $e \in E_1\Gamma$ be an edge such that $\iota(e) = u\in V_0\Gamma$ and $\tau(e) = v\in V_1\Gamma$. 
By construction, the fins in $\bfX_u$ correspond to $G_{\tilde u}$-orbits of fins in $\bfY_{\tilde u}$, which in turn correspond to $G_{\tilde u}$-orbits of edges in $\lk(\tilde u)$. Hence we get one fin $S_e \in \partial \bfX_u$ for each edge $e$ with $\iota(e)=u$, and for each lift $\tilde{e}$ with $\iota(\tilde{e})=\tilde{u}$ the covering $\bfY_{\tilde{u}}\to\bfX_u$ restricts to a covering $S_{\tilde{e}}\to S_e$ of fins.
On the level of fundamental groups, the fin $S_e$ corresponds to the $G_u$-conjugacy class of the image $\zeta_{\bar{e}}(G_e) \leqslant G_u$.
Having an orientation $\bbS_e$ of the fin $S_e$ corresponds to choosing an orientation of the fin $S_{\tilde{e}}$, which corresponds to a choice of end $\cO$ of $G_{\tilde{e}}$. Then the colour of the oriented fin is $\lambda(\bbS_e)=[\tilde{e},\cO]$, while the colour of the edge is $[e]=[\tilde{e}]$ -- in particular $\lambda(\bbS_e)$ determines $[e]$.
Let $\mathcal{X}_e$ be homeomorphic to the circle and let $\phi_{\bar{e}}: \mathcal{X}_e \rightarrow \mathcal{X}_u$ be the homeomorphism onto $S_e \subseteq \bfX_u$ that induces $\zeta_{\bar{e}}$, and let $\phi_e : \mathcal{X}_e \rightarrow \mathcal{X}_v$ be the homeomorphism onto the cylindrical fibre $S_v \subseteq \mathcal{X}_v$ that induces $\zeta_e$.
Having determined the vertex spaces $\{\mathcal{X}_v \mid v \in V\Gamma\}$, edge spaces $\{\mathcal{X}_e \mid e \in E\Gamma \}$, and attaching maps $\{ \phi_e \mid e \in E\Gamma\}$, we obtain the graph of spaces $(\mathcal{X}, \Gamma)$.

We construct $(\mathcal{X}', \Gamma')$ similarly. So we have a vertex space $\calX'_{u'}=\bfX_{u'}:=\bfY_{\tilde{u}'}/G_{\tilde{u}'}$ for a rigid vertex $u'\in V_0 \Gamma'$ with a lift $\tilde{u}'\in V_0 T_c$, and for $e'$ with $\iota(e')=u'$ we have a fin $S_{e'}\in\partial\bfX_{u'}$. For a cylindrical vertex $v'\in V_1\Gamma'$ we let $\calX'_{v'}$ be a torus containing a cylindrical fibre $S_{v'}\cong S^1$ corresponding to the cylindrical factor $\mathbb{Z}'_{v'}\leqslant G_{v'}$. For $e'\in E_1\Gamma'$ an edge with $\iota(e')=u'\in V_0\Gamma'$ and $\tau(e')=v'\in V_1\Gamma'$, we let $\calX'_{e'}$ be a circle, and $\phi'_{\bar{e}'}:\calX'_{e'}\to\calX'_{u'}$, $\phi'_{e'}:\calX'_{e'}\to\calX'_{v'}$ maps that are homeomorphisms onto $S_{e'}$ and $S_{v'}$ respectively.
\end{defn}

\begin{defn}\label{defn:orientations}
	(Orientations)\\
	Let $v\in V_1 \Gamma$ be a cylindrical vertex with lift $\tilde{v}\in V_1 T_c$. Because we have identified $\pi_1 S_v$ with $\mathbb{Z}_{\tilde{v}}$, a choice of end $\cO$ on $\mathbb{Z}_{\tilde{v}}$ induces an orientation $\tto$ on $S_v$ as a 1-manifold. In keeping with Definition \ref{graphcolfin}, we use the notation $\bbS_v=(S_v,\tto)$, and we call this an \emph{oriented cylindrical fibre}. We colour oriented cylindrical fibres according to the $\scrG$-orbit of the corresponding oriented cylinders, and denote these colours with square brackets, so $[\bbS_v]:=[\tilde{v},\cO]$ for any lift $\tilde{v}$ of $v$ and choice of end $\cO$ of $\mathbb{Z}_{\tilde{v}}$ that induces the orientation $\bbS_v$. Note that different lifts $\tilde{v}$ will give oriented cylinders in the same $G$-orbit, so the colouring on $\bbS_v$ is well-defined. 
	
	Similarly, we can put orientations on the edge spaces $\bbX_e=(\calX_e,\tto)$, and of course we already have the notion of oriented fin $\bbS_e=(S_e,\tto)$. We colour oriented edge spaces according to the $\scrG$-orbit of the corresponding oriented edge groups, and denote these colours with square brackets, so $[\bbX_e]:=[\tilde{e},\cO]\in\mathcal{C}$ for any lift $\tilde{e}$ of $e$ and choice of end $\cO$ of $G_{\tilde{e}}$ that induces the orientation $\bbX_e$. As for oriented fins we use bars to denote the opposite orientation, so $\bar{\bbS}_v$ is the opposite orientation to $\bbS_v$ and $\bar{\bbX}_e$ is the opposite orientation to $\bbX_e$. When $\phi_e:\bbX_e\to\bbS_v$ is orientation preserving we write $\phi_e(\bbX_e)=\bbS_e$, and when $\phi_{\bar{e}}:\bbX_e\to\bbS_e$ is orientation preserving we write $\phi_{\bar{e}}(\bbX_e)=\bbS_e$. We make analogous definitions for $\bbS_{v'}=(S_{v'},\tto)$ and $\bbX_{e'}=(\calX'_{e'},\tto)$ in $\calX'$.
\end{defn}

At this point we have ways of defining orientations on several different objects, so we should take a moment to check that these orientations are compatible by chasing the definitions. Suppose $\tilde{e}$ is a lift of an edge $e\in E_1\Gamma$ with $\tau(e)=v\in V_1\Gamma$, $\tau(\tilde{e})=\tilde{v}$, $\iota(e)=u\in V_0\Gamma$ and $\iota(\tilde{e})=\tilde{u}$. If $\cO$ is a choice of end of $\mathbb{Z}_{\tilde{v}}=G_{\tilde{e}}$ then we get an oriented cylindrical fibre $\bbS_v$ as above, but also an oriented fin $\bbS_{\tilde{e}}$ as in Section \ref{sec:treeoftrees}, which descends to an oriented fin $\bbS_e\in\partial_\tto\bfX_u$. So we have a diagram

\begin{equation}
\begin{tikzcd}[
ar symbol/.style = {draw=none,"#1" description,sloped},
isomorphic/.style = {ar symbol={\cong}},
equals/.style = {ar symbol={=}},
subset/.style = {ar symbol={\subset}}
]
(\tilde{v},\cO)\ar[dashed]{d}&(\tilde{e},\cO)\ar[dashed]{l}\ar[dashed]{r}&\bbS_{\tilde{e}}\ar{d}\\
\bbS_v&\bbX_e\ar{l}[swap]{\phi_e}\ar{r}{\phi_{\bar{e}}}&\bbS_e,                                                
\end{tikzcd}
\end{equation}
where the dotted arrows represent one orientation inducing another, and the solid arrows are orientation preserving maps of 1-manifolds. The colours are also compatible, so $[\bbS_v]=[\tilde{v},\cO]$ and $[\bbX_e]=[\tilde{e},\cO]=\lambda(\bbS_{\tilde{e}})=\lambda(S_e)$.

\begin{defn}\label{defn:stretchratiodown}
	(Stretch ratio)\\
	For a rigid vertex $\tilde{u}\in V_0 T_c$ and $\iota(\tilde{e})=\tilde{u}$, in Definition \ref{defn:stretchratio} we set $r_{\tilde{e}}$ to be the translation length of a generator of $g\in G_{\tilde{e}}$ acting on $\bfY_{\tilde{u}}$. We know that $G_{\tilde{e}}$ is the $G_{\tilde{u}}$-stabiliser of the fin $S_{\tilde{e}}$, and that the quotient of $S_{\tilde{e}}$ is the fin $S_e\in\partial\bfX_u$, where $\tilde{u}$ and $\tilde{e}$ descend to $u$ and $e$ in $\Gamma$, and so $r_{\tilde{e}}=\ell(S_e)$. For $\tilde{v}\in V_1 T_c$, we defined the stretch ratio $\Str(\tilde{v})$ to be the ratio of the numbers $r_{\tilde{e}}$ for $\tilde{e}\in\lk(\tilde{v})$, thus it makes sense to define the  \emph{stretch ratio of $v\in V_1 \Gamma\sqcup V_1\Gamma'$} to be the class of functions $\Str(v):=[e\in$ $\lk(v)\mapsto\ell(S_e)]$.
\end{defn}

Lemma \ref{stretch} tells us that the stretch ratio depends only on the $\scrG$-orbits of the edges. More precisely, there are numbers $r_{[\tilde{e}]}$ such that $\Str(\tilde{v})=[\tilde{e}\mapsto r_{[\tilde{e}]}]$ for $\tilde{v}\in V_1 T_c$, which implies that
\begin{equation}\label{stretchequation}
\Str(v)=[e\mapsto r_{[e]}]
\end{equation}
for $v\in V_1\Gamma\sqcup V_1\Gamma'$.

\subsection{Density coefficients}

\begin{defn}
Given our graph of spaces $(\mathcal{X}, \Gamma)$ 
we define the \emph{volume} of $\mathcal{X}$ to be the following sum (recall that $|X_u|$ is the number of vertices in the graph $X_u$):
\begin{align}
|\mathcal{X}|&:=\sum_{u\in V_0\Gamma}|X_u|.\nonumber
\end{align}

\noindent Given a rigid vertex $u\in V_0\Gamma$, we define the \emph{density} of the colour $[u]$ in $(\mathcal{X}, \Gamma)$, denoted 
$\rho_{[u]}$, to be the value
\begin{align}\label{udensity}
\rho_{[u]}&:=\sum_{u_*\in V_0\Gamma,\,[u_*]=[u]}|X_{u_*}|/|\mathcal{X}|.
\end{align}
\end{defn}

\begin{remk}
We can also consider the density of $[u]$ in $(\mathcal{X}', \Gamma')$, since the vertices of $V_0\Gamma'$ are labelled with the same colours, but \emph{prima facie} there is no reason to believe that they will be equal.
However, because density is preserved by finite covers of graphs of spaces, after we have constructed a common finite cover $\widehat{\mathcal{X}}$ we will know that $\rho_{[u]}$ gives the same value whether defined with $\Gamma$ or $\Gamma'$.
\end{remk}

We recall Definition~\ref{defn:density}, the notion of the density $\rho_c$ of a colour $c$ given a graph with coloured fins. 
The following lemma relates the local notion of density of a colour in a particular vertex space $\mathbf{X}_u$, with the global density of the vertex spaces of that particular colour.

\begin{lem} \label{lem:densityEquation}
Let $\bbS_e\in\partial_\tto\bfX_u$ be an oriented fin of colour $c$.
Then 
 \begin{align}
  \sum_{\lambda(\bbS_{e_*})=c,\,e_*\in E_1\Gamma}\ell(\bbS_{e_*}) = \rho_c\rho_{[u]}|\mathcal{X}|
 \end{align}
\end{lem}

\begin{proof}
	All $\mathbf{X}_{u_*}$ containing an oriented fin of colour $c$ have $[u_*]=[u]$ and are covered by $\mathbf{Y}_{\tilde{u}}$ for some $\tilde{u}\in V_0 T_c$ a lift of $u$. Hence by Theorem \ref{Leighton}, all these $\mathbf{X}_{u_*}$ have a common finite cover, and so they all have the same density $\rho_c$. We can then make the following computation:
 \begin{align*}
  \sum_{\lambda(\bbS_{e_*})=c,\,e_*\in E_1\Gamma}\ell(\bbS_{e_*}) 
          & = \sum_{\substack{u_* \in V_0\Gamma, \\ [u_*] = [u]} } \; \Bigg[ \; \sum_{\substack{\bbS_{e_*}\in\partial_\tto\bfX_{u_*}, \\ \lambda(\bbS_{e_*}) = c}} \ell(\bbS_{e_*}) \Bigg]\\
   & = \sum_{\substack{u_* \in V_0\Gamma, \\ [u_*] = [u]} } \rho_c | X_{u_*} | \\
   & = \rho_c \rho_{[u]} |\mathcal{X}|.\qedhere
 \end{align*}
   
\end{proof}

\bigskip
\section{A common finite cover} \label{sec:CommonFiniteCover}

In this section we complete the proof of Theorem \ref{thm:main} by constructing a common finite cover of the graphs of spaces $\mathcal{X}$ and $\mathcal{X}'$ from the previous section.

\subsection{A template for our desired common cover} \label{sec:template}

More precisely, we will construct finite covers $\widehat{\mathcal{X}} \rightarrow \mathcal{X}$ and $\widehat{\mathcal{X}}' \rightarrow \mathcal{X}'$ such that 
$\widehat{\calX}$ and $\widehat{\calX}'$ are homotopy equivalent.
This will be achieved by constructing $\widehat{\calX}$ and $\widehat{\calX}'$ such that their induced decompositions are over graphs $\widehat{\Gamma}$ and $\widehat{\Gamma}'$ that are type and colour preserving.
Indeed if we identify $\widehat{\Gamma}$ and $\widehat{\Gamma}'$, then we will have homeomorphic vertex spaces $\widehat{\calX}_v \cong \widehat{\calX}'_v$ for all $v \in V\widehat{\Gamma}$ and homeomorphic edge spaces $\widehat{\calX}_e \cong \widehat{\calX}'_e$ for all $e \in E\widehat{\Gamma}$.
The attaching maps $\hat{\phi}_e, \hat{\phi}_e' : \widehat{\calX}_e \rightarrow \widehat{\calX}_v$ will be homotopic for all $e \in E \widehat{\Gamma}$.
By a standard result in topology the graphs of spaces will therefore be homotopic.
Commensurability of $G$ and $G'$ will follow.

\subsection{Common covers of vertex and edge spaces}\label{sec:vertexedgecovers}

In this section we define the vertex and edge spaces of $\widehat{\calX}$ and $\widehat{\calX}'$.

\begin{defn}(Common covers of rigid vertex spaces)\\
For rigid vertices $u\in V_0\Gamma$ and $u'\in V_0\Gamma'$ of the same colour $[u]=[u']$, we describe how to produce a common cover $\widehat{\bfX}_{u,u'}$ of the graphs with fins $\mathcal{X}_ u = \mathbf{X}_u$ and $ \mathcal{X}'_{u'} = \mathbf{X}_{u'}$. These two graphs with fins are defined by the quotients $\mathbf{Y}_{\tilde u}/G_{\tilde u}$ and $\mathbf{Y}_{\tilde{u}'}/G'_{\tilde{u}'}$, where $\tilde{u}$ and $\tilde{u}'$ are lifts of $u$ and $u'$ respectively to $T_c$. As $[\tilde u]=[u]=[u']=[\tilde{u}']$, we know that there exists $[f]\in\scrG$ with $\hat{f}(\tilde{u})=\tilde{u}'$ and $\morp{f}_{\tilde{u}}:\bfY_{\tilde{u}}\to\bfY_{\tilde{u}'}$ an isomorphism. We know that the action of $[f]^{-1}G'_{\tilde{u}'}[f]\leqslant\scrG_{\tilde{u}}$ on $\bfY_{\tilde{u}}$ is conjugate to the action of $G'_{\tilde{u}'}$ on $\bfY_{\tilde{u}'}$ via $\morp{f}_{\tilde{u}}$, so $\bfX_{u'}\cong \bfY_{\tilde{u}}/[f]^{-1}G'_{\tilde{u}'}[f]$. 
We can then apply Theorem \ref{Leighton} to $\bfY_{\tilde{u}}$, with $\Gamma_1,\Gamma_2 \leqslant\Aut(\bfY_{\tilde{u}})$ the images of $G_{\tilde{u}},[f]^{-1}G'_{\tilde{u}'}[f]\leqslant\scrG_{\tilde{u}}$ under the homomorphism $\scrG_{\tilde{u}}\to\Aut(\bfY_{\tilde{u}})$, to produce a common finite cover $\widehat{\bfX}_{u,u'}$ of $\bfX_u$ and $\bfX_{u'}$ that satisfies equation (\ref{finequation}). Note that the colours of oriented fins in $\bfY_{\tilde{u}}$ were defined to correspond to $\scrG$-orbits (Section \ref{sec:treeoftrees}), so $\scrG_{\tilde{u}}$ acts transitively on the oriented fins of each colour in $\bfY_{\tilde{u}}$.
Additionally note that, while the definitions of $\bfX_u$ and $\bfX_{u'}$ did not depend on the choice of lifts $\tilde{u}$ and $\tilde{u}'$, the definition of $\widehat{\bfX}_{u,u'}$ does depend on these choices, and also on the choice of $[f]\in\scrG$.
\end{defn}

The following lemma is a direct application of omnipotence of free groups.

\begin{lem}\label{lem:samefinlength}
 We can choose integers $\ell_{[e]}$ for $e\in E_1\Gamma\sqcup E_1\Gamma'$ and replace each $\widehat{\bfX}_{u,u'}$ with a finite cover, such that the length of a fin $\hat{S}\in\partial\widehat{\bfX}_{u,u'}$ that covers a fin $S_e\in\partial\bfX_u\sqcup\partial\bfX_{u'}$ is $\ell_{[e]}$.
 Moreover, for a vertex $v \in V_1\Gamma \sqcup V_1\Gamma'$ we have $\Str(v) = [e\mapsto \ell_{[e]}]$, or equivalently there is an integer $d_v$ such that
  \begin{equation}\label{omnistretch}
\ell_{[e]}=d_v\ell(S_e),
  \end{equation}
  for all $e\in\lk(v)$ -- so the degree of the covering $\hat{S}\to S_e$ is $d_v$ and depends only on $v$.
\end{lem}

\begin{proof}
By omnipotence of free groups \cite[Theorem 3.5]{omnipotence}, there exists $N>0$ such that for any $k:\partial\widehat{\bfX}_{u,u'}\to\mathbb{N}$ there exists a normal cover $\Phi:\overline{\bfX}\to\widehat{\bfX}_{u,u'}$ such that the length of any fin in $\Phi^{-1}(S)$ is $Nk(S)$. If $\hat{S}\in\partial\widehat{\bfX}_{u,u'}$ covers fins $S_e\in\partial\bfX_u$ and $S_{e'}\in\partial\bfX_{u'}$ then $[e]=[e']$, because $S_e$ and $S_{e'}$ will have orientations of the same colour, and the colour of an oriented fin determines the colour of the corresponding edge (see Definition \ref{defn:GoScalXcalX'}).
Therefore, we can replace the $\widehat{\bfX}_{u,u'}$ with further finite covers and assume that the length of a fin covering $S_e$ is $\ell_{[e]}$. We know from (\ref{stretchequation}) that $\Str(v)=[e\mapsto r_{[e]}]$, so if we set $\ell_{[e]}=Nr_{[e]}$, then we have that
\begin{equation}\label{Xhatfinratio}
\Str(v) = [e\mapsto \ell_{[e]}]
\end{equation}
for a vertex $v\in V_1\Gamma \sqcup V_1\Gamma'$.
Note that equation (\ref{finequation}) from Theorem \ref{Leighton} is preserved by passing to a further finite cover.
\end{proof}
\bigskip

We also need common finite covers for the cylindrical vertex spaces.

\begin{defn}(Common covers of cylindrical vertex spaces)\\\label{defn:comcovercylvertex}
	Given cylindrical vertices $v\in V_1\Gamma$ and $v'\in V_1\Gamma'$ and oriented cylindrical fibres $\bbS_v$ and $\bbS_{v'}$ of the same colour (see Definition \ref{defn:orientations}), we let $\hat{\bbS}(\bbS_v,\bbS_{v'})$ be an oriented circle equipped with orientation preserving covering maps to $\bbS_v$ and $\bbS_{v'}$ of degrees $d_v$ and $d_{v'}$ respectively (where $d_v$ and $d_{v'}$ come from (\ref{omnistretch})). We extend each  $\hat{\bbS}(\bbS_v,\bbS_{v'})$ to a common cover $\widehat{\calX}(\bbS_v,\bbS_{v'})$ of the vertex spaces $\calX_v$ and $\calX'_{v'}$. If $G_v\cong G'_{v'}\cong\mathbb{Z}$ then no extension is necessary, while if $G_v\cong G'_{v'}\cong\mathbb{Z}^2$ then we make $\widehat{\calX}(\bbS_v,\bbS_{v'})$ a torus containing $\hat{\bbS}(\bbS_v,\bbS_{v'})$ as an embedded circle, so that $\widehat{\calX}(\bbS_v,\bbS_{v'})$ is the cover corresponding to the subgroups $d_v\Zv\times\mathbb{Z}\leqslant\Zv\times\mathbb{Z}=G_v=\pi_1(\calX_v)$ and $d_{v'}\mathbb{Z}_{v'}\times\mathbb{Z}\leqslant\mathbb{Z}_{v'}\times\mathbb{Z}=G_{v'}=\pi_1(\calX'_{v'})$. We consider $\hat{\bbS}(\bar{\bbS}_v,\bar{\bbS}_{v'})$ to be the same embedded circle as $\hat{\bbS}(\bbS_v,\bbS_{v'})$ but with orientation reversed, while $\widehat{\calX}(\bbS_v,\bbS_{v'})=\widehat{\calX}(\bar{\bbS}_v,\bar{\bbS}_{v'})$ is just a space with no orientation.
	Thus we obtain a pair of common covers for each pair of vertices $v$ and $v'$. 
	See Figure~\ref{fig:commonCylinderCover} for an illustration.
	\end{defn}

	\begin{figure}[H]
 \centering
	\begin{overpic}[width=.3\textwidth,tics=5,]{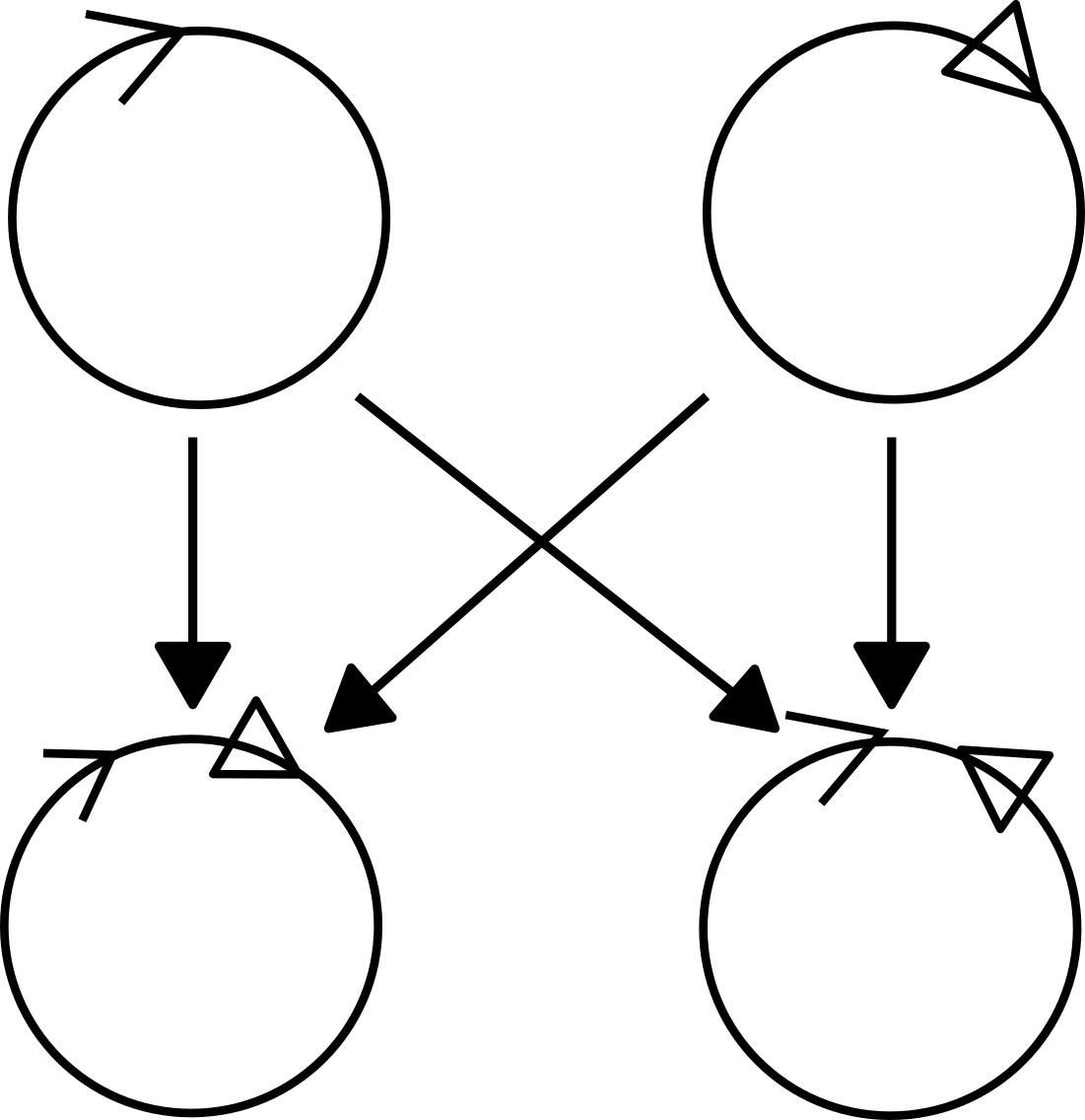} 
	  \put(-9,16){$\bbS_v$}
  	  \put(99,16){$\bbS_{v'}$}
      \put(-32,80){$\hat{\bbS}(\; \bbS_v,\bbS_{v'})$}
      \put(100,80){$\hat{\bbS}(\; \bbS_v,\bar{\bbS}_{v'})$}
          \end{overpic}
	\caption{Each cylindrical fibre has the clockwise orientation and the covering maps are determined by the arrows.
	Note that if we take the anticlockwise orientations we obtain $\hat{\bbS}(\; \bar{\bbS}_v,\bar{\bbS}_{v'})$ and $\hat{\bbS}(\; \bar{\bbS}_v, {\bbS}_{v'})$.
	Thus $\widehat{\calX}(\bbS_v,\bbS_{v'})=\widehat{\calX}(\bar{\bbS}_v,\bar{\bbS}_{v'})$ and $\widehat{\calX}(\bbS_v,\bar{\bbS}_{v'})=\widehat{\calX}(\bar{\bbS}_v,{\bbS}_{v'})$.}
	\label{fig:commonCylinderCover}
    \end{figure}

\begin{defn}(Common covers of edge spaces)\\\label{defn:comcoveredge}
	If $e\in E_1\Gamma$ and $e'\in E_1\Gamma'$ are edges with $\tau(e)=v\in V_1\Gamma$ and $\tau(e')=v'\in V_1 \Gamma'$, and $\bbX_e$ and $\bbX_{e'}$ are orientations of the same colour, then we define $\widehat{\bbX}(\bbX_e,\bbX_{e'})$ to be an oriented circle equipped with orientation preserving covering maps to $\bbX_e$ and $\bbX_{e'}$ of degrees $d_v$ and $d_{v'}$ respectively. We identify $\widehat{\bbX}(\bbX_e,\bbX_{e'})$ and $\widehat{\bbX}(\bar{\bbX}_e,\bar{\bbX}_{e'})$ as two orientations of the same space $\widehat{\calX}(\bbX_e,\bbX_{e'})=\widehat{\calX}(\bar{\bbX}_e,\bar{\bbX}_{e'})$.
	So again we obtain a pair of common covers for each pair of edges.
\end{defn}

\subsection{Link maps}

Having defined common covers of the edge and vertex spaces, we now need to glue them together, or rather enumerate the possible ways of gluing them together. The following definition will be used to describe the ways of gluing a cylindrical vertex space $\widehat{\calX}(\bbS_v,\bbS_{v'})$ to edge spaces $\widehat{\calX}(\bbX_e,\bbX_{e'})$.

\begin{defn}(Link maps)\\\label{den:linkmaps}
	Let $v\in V_1\Gamma$ and $v'\in V_1\Gamma'$ be cylindrical vertices and consider oriented cylindrical fibres $\bbS_v$ and $\bbS_{v'}$ of the same colour. This induces orientations $\bbX_e$ and $\bbX_{e'}$ on the incident edge spaces. A \emph{link map} from $\bbS_v$ to $\bbS_{v'}$ is a colour preserving bijection between the incident oriented edge spaces, so in symbols it is a bijection
	$$\sigma:\lk(v)\to\lk(v')$$
	such that $[\bbX_e]=[\bbX_{\sigma(e)}]$ for all $e\in\lk(v)$. We let $\rm{LkMap}(\bbS_v,\bbS_{v'})$ be the set of all link maps from $\bbS_v$ to $\bbS_{v'}$.
\end{defn}

\begin{lem}\label{lem:linkmapsexist}
Let $c\in\mathcal{C}$. The number of $e\in\lk(v)$ with $[\bbX_e]=c$ is equal to the number of $e'\in\lk(v')$ with $[\bbX_{e'}]=c$ is equal to $N_c[\bbS_v]$. In particular, $\rm{LkMap}(\bbS_v,\bbS_{v'})$ is non-empty.
\end{lem}
\begin{proof}
The oriented cylindrical fibre $\bbS_v$ corresponds to a choice of end $\cO$ of a cylindrical factor $\mathbb{Z}_{\tilde{v}}$ for $\tilde{v}$ a lift of $v$. The incident oriented edge spaces $\bbX_e$ correspond to $G_{\tilde{v}}$-orbits of oriented edge groups $(\tilde{e},\cO)$ with $\tilde{e}\in\lk(\tilde{v})$. Moreover, for $c\in\mathcal{C}$, the number of $G_{\tilde{v}}$-orbits of oriented edge groups $(\tilde{e},\cO)$ of colour $c$ is equal to the cylinder number $t_c(v,\cO)$ by Definition \ref{defn:torusratio}. Hence the number of incident oriented edge spaces $\bbX_e$ of colour $c$ is also equal to $t_c(v,\cO)$, and by (\ref{cylnumbermatch}) we have
$$t_c(v,\cO)=N_c[v,\cO]=N_c[\bbS_v],$$
so it only depends on the colours $c$ and $[\bbS_v]$. Again by (\ref{cylnumbermatch}), we know that the number of oriented edge spaces incident to $\bbS_{v'}$ of colour $c$ is equal to $N_c[\bbS_{v'}]=N_c[\bbS_v]$.
\end{proof}

\begin{remk}\label{remk:linkmapflip}
	$\sigma:\lk(v)\to\lk(v')$ defines a link map from $\bbS_v$ to $\bbS_{v'}$ if and only if it defines a link map from $\bar{\bbS}_v$ to $\bar{\bbS}_{v'}$. This is because $\bar{\bbS}_v$ and $\bar{\bbS}_{v'}$ induce the orientations $\bar{\bbX}_e$ and $\bar{\bbX}_{e'}$ on the incident edge spaces, so if $\sigma$ defines a link map from $\bbS_v$ to $\bbS_{v'}$ then $[\bar{\bbX}_e]=\overline{[\bbX_e]}=\overline{[\bbX_{\sigma(e)}]}=[\bar{\bbS}_{\sigma(e)}]$ for each $e\in\lk(v)$.
\end{remk}
\bigskip

Given a link map $\sigma:\lk(v)\to\lk(v')$ from $\bbS_v$ to $\bbS_{v'}$ and $e\in\lk(v)$ with $\sigma(e)=e'$, suppose $\iota(e)=u$ and $\iota(e')=u'$. Let $\phi_{\bar{e}}(\bbX_e)=\bbS_e\in\partial_\tto\bfX_u$ and $\phi_{\bar{e}'}(\bbX_{e'})=\bbS_{e'}\in\partial_\tto\bfX_{u'}$. The fins $\bbS_e$ and $\bbS_{e'}$ both have colours equal to $[\bbX_e]=[\bbX_{e'}]$, so equation (\ref{finequation}) implies that there exists $\hat{\bbS}\in \partial_\tto \widehat{\bfX}_{u,u'}(\bbS_e,\bbS_{e'})$ that covers both of them. Equation (\ref{omnistretch}) tells us that these covering maps of fins have degrees $d_v$ and $d_{v'}$ respectively, so we get two commutative diagrams as follows.

\begin{equation}\label{linkgluev}
\begin{tikzcd}[
ar symbol/.style = {draw=none,"#1" description,sloped},
isomorphic/.style = {ar symbol={\cong}},
equals/.style = {ar symbol={=}},
subset/.style = {ar symbol={\subset}}
]
\widehat{\calX}(\bbS_v,\bbS_{v'})\ar{d}&\hat{\bbS}(\bbS_v,\bbS_{v'}) \ar{d}{d_v} \ar[l,hook'] & \widehat{\bbX}(\bbX_e,\bbX_{e'}) \ar{l}[swap]{\sim} \ar{r}{\sim} \ar{d}{d_v} &\hat{\bbS}\ar[r,hook]\ar{d}{d_v} &\widehat{\mathbf{X}}_{u,u'}\ar{d}\\
\calX_v&\bbS_v\ar[l,hook'] & \bbX_e \ar{l}{\sim}[swap]{\phi_e}\ar{r}{\phi_{\bar{e}}}[swap]{\sim}& \bbS_e\ar[r,hook]&\bfX_u
\end{tikzcd}
\end{equation}
\begin{equation}\label{linkgluev'}
\begin{tikzcd}[
ar symbol/.style = {draw=none,"#1" description,sloped},
isomorphic/.style = {ar symbol={\cong}},
equals/.style = {ar symbol={=}},
subset/.style = {ar symbol={\subset}}
]
\widehat{\calX}(\bbS_v,\bbS_{v'})\ar{d}&\hat{\bbS}(\bbS_v,\bbS_{v'}) \ar{d}{d_{v'}} \ar[l,hook']  & \widehat{\bbX}(\bbX_e,\bbX_{e'}) \ar{l}[swap]{\sim} \ar{r}{\sim} \ar{d}{d_{v'}} &\hat{\bbS}\ar[r,hook]\ar{d}{d_{v'}} &\widehat{\mathbf{X}}_{u,u'}\ar{d}\\
\calX'_{v'} &\bbS_{v'}\ar[l,hook'] & \bbX_{e'} \ar{l}{\sim}[swap]{\phi'_{e'}}\ar{r}{\phi'_{\bar{e}'}}[swap]{\sim}& \bbS_{e'}\ar[r,hook]&\bfX_{u'}
\end{tikzcd}
\end{equation}

In these diagrams a homeomorphism is indicated by $\sim$. Also note that the middle six spaces in each diagram have associated orientations, which are preserved by the maps between them.

These diagrams give us the right local data to define edge maps in the covers $\widehat{\calX}$ and $\widehat{\calX}'$. The vertical maps are the coverings from vertex and edge spaces of $\widehat{\calX}$ and $\widehat{\calX}'$ to vertex and edge spaces of $\calX$ and $\calX'$, as defined in Section \ref{sec:vertexedgecovers}. Then the top row of (\ref{linkgluev}) can be used to define edge maps of $\widehat{\calX}$ while the top row of (\ref{linkgluev'}) can be used to define edge maps of $\widehat{\calX}'$. The two maps from $\widehat{\bbX}(\bbX_e,\bbX_{e'})$ to $\hat{\bbS}$ are both orientation preserving homeomorphisms of circles, hence they are homotopic, similarly the two maps from $\widehat{\bbX}(\bbX_e,\bbX_{e'})$ to $\hat{\bbS}(\bbS_v,\bbS_{v'})$ are homotopic. 
See Figure~\ref{fig:coverConstruction}.
From now on we will only care about these edge maps up to homotopy, so we will just talk about a single cover $\widehat{\calX}$ of $\calX$ and $\calX'$.

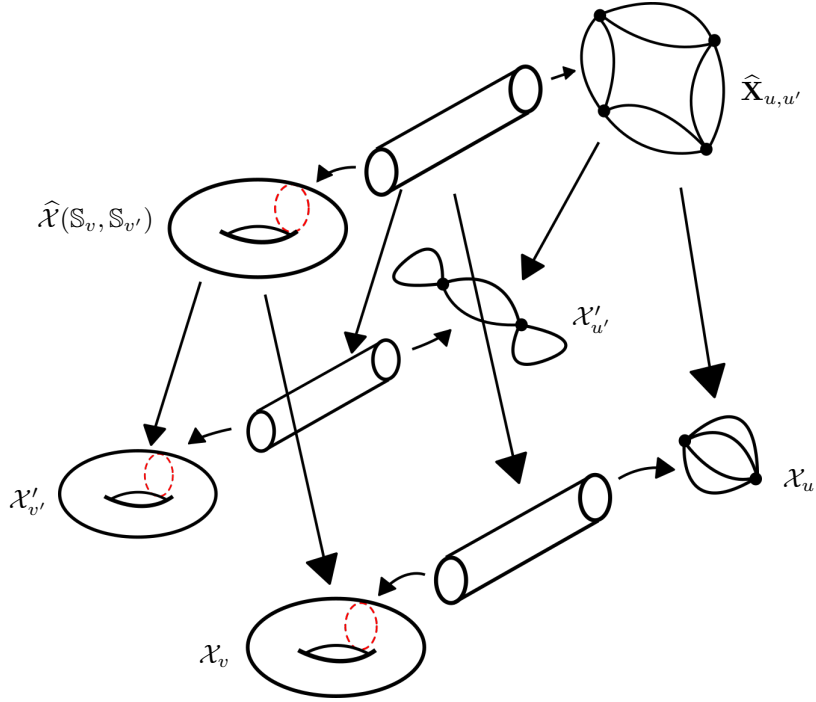
\begin{figure}[H]
 \centering
	\begin{overpic}[width=.6\textwidth,tics=5,]{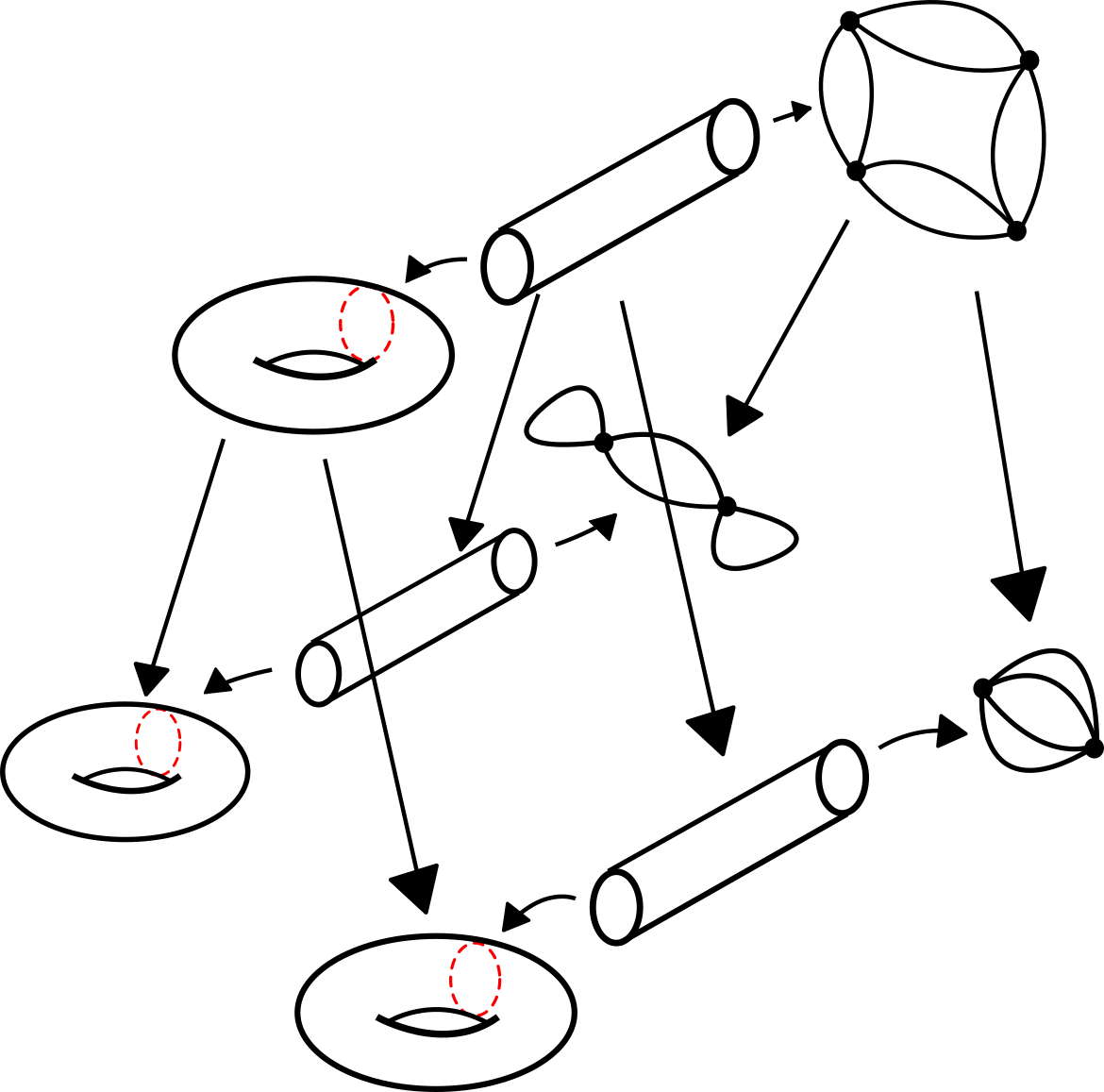} 
	 \put(20,5){$\calX_v$}
	 \put(-7,27){$\calX_{v'}'$}
	 \put(-3,67){$\widehat{\calX}(\bbS_v, \bbS_{v'})$}
	 \put(73,53){$\calX_{u'}'$}
	 \put(103,30){$\calX_{u}$}
	 \put(97,85){$\widehat{\bfX}_{u,u'}$}
	\end{overpic}

	\caption{An illustration of how the common cover is constructed.
	The arrows in the diagram commute, and the dashed lines in the tori denote the cylindrical fibres.}
	\label{fig:coverConstruction}
    \end{figure}

\begin{remk}\label{remk:linkglueflip}
	Under the replacement $\bbS_v,\bbS_{v'},\bbX_e,\bbX_{e'},\bbS_e,\bbS_{e'},\hat{\bbS}\mapsto\bar{\bbS}_v,\bar{\bbS}_{v'},\bar{\bbX}_e,\bar{\bbX}_{e'},\bar{\bbS}_e,\bar{\bbS}_{e'},\bar{\hat{\bbS}}$, diagrams (\ref{linkgluev}) and (\ref{linkgluev'}) will consist of the same spaces and maps, the orientations of the spaces will just reverse. So when using $\sigma$ to construct the local data of edge maps in $\widehat{\calX}$, it doesn't matter whether we regard $\sigma$ as a link map from $\bbS_v$ to $\bbS_{v'}$ or as a link map from $\bar{\bbS}_v$ to $\bar{\bbS}_{v'}$.
\end{remk}

\subsection{From local common covers to global}\label{sec:LocaltoGlobal}

A finite common cover $\mathcal{\widehat{X}}$ of $\mathcal{X}$ and $\mathcal{X}'$ will be constructed by taking as vertex spaces $\omega(u,u')$ copies of each $\widehat{\mathbf{X}}_{u,u'}$ and $\omega(\bbS_v,\bbS_{v'})$ copies of each $\widehat{\calX}(\bbS_v,\bbS_{v'})$, and as edge spaces $\omega(\bbX_e,\bbX_{e'})$ copies of each $\widehat{\calX}(\bbX_e,\bbX_{e'})$. We require $\omega(\bbS_v,\bbS_{v'})=\omega(\bar{\bbS}_v,\bar{\bbS}_{v'})$ and $\omega(\bbX_e,\bbX_{e'})=\omega(\bar{\bbX}_e,\bar{\bbX}_{e'})$ because $\widehat{\calX}(\bbS_v,\bbS_{v'})=\widehat{\calX}(\bar{\bbS}_v,\bar{\bbS}_{v'})$ and $\widehat{\calX}(\bbX_e,\bbX_{e'})=\widehat{\calX}(\bar{\bbX}_e,\bar{\bbX}_{e'})$. To each copy of $\widehat{\calX}(\bbS_v,\bbS_{v'})$ we associate a link map $\sigma\in\rm{LkMap}(\bbS_v,\bbS_{v'})$, and then for each $e\in\lk(v)$ we glue an edge space $\widehat{\calX}(\bbX_e,\bbX_{e'})$ to $\widehat{\calX}(\bbS_v,\bbS_{v'})$ and also to a vertex space $\widehat{\mathbf{X}}_{u,u'}$, all according to the diagrams (\ref{linkgluev}) and (\ref{linkgluev'}) (so $e'=\sigma(e)$, $u=\iota(e)$ and $u'=\iota(e')$). By Remark \ref{remk:linkglueflip} it doesn't matter whether we regard $\sigma$ as lying in $\rm{LkMap}(\bbS_v,\bbS_{v'})$ or $\rm{LkMap}(\bar{\bbS}_v,\bar{\bbS}_{v'})$. The different $\sigma\in\rm{LkMap}(\bbS_v,\bbS_{v'})$ will be evenly distributed across the $\omega(\bbS_v,\bbS_{v'})$ copies of $\widehat{\calX}(\bbS_v,\bbS_{v'})$ (so in particular $|\rm{LkMap}(\bbS_v,\bbS_{v'})|$ will divide $\omega(\bbS_v,\bbS_{v'})$).

For this to form a cover of $\calX$ and $\calX'$, we must ensure that each edge space $\widehat{\calX}(\bbX_e,\bbX_{e'})$ gets used exactly once, and that each fin in each vertex space $\widehat{\mathbf{X}}_{u,u'}$ has exactly one edge space glued to it. This requirement can be captured by a set of Gluing Equations, which we describe in the following lemma.

\begin{lem}(Gluing Equations)\\\label{lem:GluingEquations}
	We can form a common finite cover $\hat{\calX}$ of $\calX$ and $\calX'$ by the above gluing instructions if the following Gluing Equations have a positive solution:
	
	\begin{equation}\label{GluingEquations}
	\frac{\omega(\bbS_v,\bbS_{v'})}{N_c[\bbS_v]}=\omega(\bbX_e,\bbX_{e'})=\omega(u,u')|\partial_\tto \widehat{\bfX}_{u,u'}(\bbS_e,\bbS_{e'})|
	\end{equation}
	Here $\bbS_v$ and $\bbS_{v'}$ are oriented cylindrical fibres from $\calX$ and $\calX'$ of the same colour; $e\in\lk(v)$ and $e'\in\lk(v')$ are edges such that the edge spaces with induced orientations $\bbX_e$ and $\bbX_{e'}$ have the same colour $c\in\mathcal{C}$; $\iota(e)=u$ and $\iota(e')=u'$; and $\phi_{\bar{e}}(\bbX_e)=\bbS_e\in\partial_\tto\bfX_u$ and $\phi_{\bar{e}'}(\bbX_{e'})=\bbS_{e'}\in\partial_\tto\bfX_{u'}$ are the oriented fins corresponding to $\bbX_e$ and $\bbX_{e'}$.
\end{lem}
\begin{proof}
By Lemma \ref{lem:linkmapsexist}, there are $N_c[\bbS_v]$ edges $e'_*\in\lk(v')$ whose oriented edge spaces $\bbX_{e'_*}$ have colour $c$, and any choice $e\mapsto e'_*$ can be extended to a link map $\sigma:\lk(v)\to\lk(v')$. Moreover, the number of possible extensions is independent of $e'_*$, thus the proportion of link maps $\sigma\in\rm{LkMap}(\bbS_v,\bbS_{v'})$ with $\sigma(e)=e'$ is $1/N_c[\bbS_v]$. By the local gluing data of (\ref{linkgluev}) and (\ref{linkgluev'}), a copy of $\widehat{\calX}(\bbX_e,\bbX_{e'})$ is used in the construction of $\widehat{\calX}$ precisely when a link map $\sigma\in\rm{LkMap}(\bbS_v,\bbS_{v'})$ with $\sigma(e)=e'$ is associated to a vertex space $\widehat{\calX}(\bbS_v,\bbS_{v'})$. This explains the first equality in (\ref{GluingEquations}).

For the second equality in (\ref{GluingEquations}), note that the local gluing data of (\ref{linkgluev}) and (\ref{linkgluev'}) glues each copy of an oriented edge space $\widehat{\bbX}(\bbX_e,\bbX_{e'})$ to an oriented fin $\hat{\bbS}\in\partial_\tto \widehat{\bfX}_{u,u'}(\bbS_e,\bbS_{e'})$, for one of the $\omega(u,u')$ copies of $\widehat{\bfX}_{u,u'}$; and these are the only edge spaces that could be glued to $\hat{\bbS}$ because $\bbX_e$ and $\bbX_{e'}$ are the unique oriented edge spaces that attach to the oriented fins $\bbS_e$ and $\bbS_{e'}$.

Of course we also need $\omega(\bbS_v,\bbS_{v'})$, $\omega(\bbX_e,\bbX_{e'})$ and $\omega(u,u')$ to be positive integers, and for $|\rm{LkMap}(\bbS_v,\bbS_{v'})|$ to divide $\omega(\bbS_v,\bbS_{v'})$, but this can be achieved by scaling the solution suitably.
\end{proof}

\bigskip

Lemma \ref{lem:samefinlength} tells us that all fins in $\widehat{\mathbf{X}}_{u,u'}$ that cover $S_e\in\partial\bfX_u$ have length $\ell_{[e]}$, so Theorem \ref{Leighton} tells us that we can substitute
\begin{align*}
|\partial_\tto\widehat{\bfX}_{u,u'}(\bbS_e,\bbS_{e'})|=\left(\frac{|\widehat{X}_{u,u'}|}{\rho_c|X_u||X_{u'}|}\right)\frac{\ell(\bbS_e)\ell(\bbS_{e'})}{\ell_{[e]}}
\end{align*}
into equations (\ref{GluingEquations}). Thus we can solve the gluing equations by taking
\begin{equation}\label{gluesolu}
 \omega(u,u') =\frac{|X_u||X_{u'}|}{\rho_{[u]}|\widehat{X}_{u,u'}|}\textrm{, and }
 \frac{\omega(\bbS_v,\bbS_{v'})}{N_c[\bbS_v]}=\omega(\bbX_e,\bbX_{e'})= \frac{\ell(\bbS_e)\ell(\bbS_{e'})}{\ell_{[e]}\rho_c\rho_{[u]}}.
\end{equation}

It remains to show that this solution is well-defined. Note that the replacement $\bbS_v,\bbS_{v'}\mapsto \bar{\bbS}_v,\bar{\bbS}_{v'}$ will flip the orientations on the fins $\bbS_e$ and $\bbS_{e'}$, and the colour $c$ will turn to $\bar{c}$; but this will not change the lengths of the fins, and $N_{\bar{c}}[\bar{\bbS}_v]=N_c[\bbS_v]$ by Lemma \ref{lem:torusmatch}; and $\rho_{\bar{c}}=\rho_c$ because by definition this is proportional to the sum of lengths of oriented fins of colour $\bar{c}$, which equals the sum of lengths of oriented fins of colour $c$ since these are different orientations of the same fins. Hence $\omega(\bbS_v,\bbS_{v'})=\omega(\bar{\bbS}_v,\bar{\bbS}_{v'})$ and $\omega(\bbX_e,\bbX_{e'})=\omega(\bar{\bbX}_e,\bar{\bbX}_{e'})$ as required.

It is easy to see that the formula for $\omega(u,u')$ depends only on $u$ and $u'$, and that the formula for $\omega(\bbX_e,\bbX_{e'})$ depends only on $\bbX_e$ and $\bbX_{e'}$; but the reason that the formula for $\omega(\bbS_v,\bbS_{v'})$ depends only on $\bbS_v$ and $\bbS_{v'}$ is more subtle, which is the task of our final lemma. 
\begin{lem}
	The expression
$$\frac{N_c[\bbS_v]\ell(\bbS_e)\ell(\bbS_{e'})}{\ell_{[e]}\rho_c\rho_{[u]}}$$ 
depends only on $\bbS_v$ and $\bbS_{v'}$.
\end{lem}
\begin{proof}

\begin{align*}
\ell_{[e]}\rho_c\rho_{[u]}|\mathcal{X}| & = \ell_{[e]}\sum_{\lambda(\bbS_{e_*})=c,\,e_*\in E_1\Gamma}\ell(\bbS_{e_*})&\text{by Lemma \ref{lem:densityEquation},} \\
&=\ell_{[e]}\sum_{\substack{[\bbS_{v_*}]=[\bbS_v],\,v_*\in V_1\Gamma\\ \phi_{e_*}\phi^{-1}_{\bar{e}_*}(\bbS_{e_*})=\bbS_{v_*},\,\lambda(\bbS_{e_*})=c,\,e_*\in\lk(v)}}\ell(\bbS_{e_*})\\
&\ell_{[e]}\sum_{[\bbS_{v_*}]=[\bbS_v],\,v_*\in V_1\Gamma}N_c[\bbS_v]\ell(\bbS_{e_*})&\text{by Lemma \ref{lem:linkmapsexist},}\\
&=\sum_{[\bbS_{v_*}]=[\bbS_v],\,v_*\in V_1\Gamma}\frac{N_c[\bbS_v]\ell_{[e]}^2}{d_{v_*}}&\text{by (\ref{omnistretch}),}\\
&=\sum_{[\bbS_{v_*}]=[\bbS_v],\,v_*\in V_1\Gamma}\frac{N_c[\bbS_v]d_v d_{v'}\ell(\bbS_e)\ell(\bbS_{e'})}{d_{v_*}}&\text{again by (\ref{omnistretch}).}
\end{align*}
And so our required expression
\begin{equation*}
\frac{N_c[\bbS_v]\ell(\bbS_e)\ell(\bbS_{e'})}{\ell_{[e]}\rho_c\rho_{[u]}}=|\calX|\left(\sum_{[\bbS_{v_*}]=[\bbS_v],\,v_*\in V_1\Gamma}\frac{d_v d_{v'}}{d_{v_*}}\right)^{-1},
\end{equation*}
only depends on $\bbS_v$ and $\bbS_{v'}$.
\end{proof}

We conclude that (\ref{gluesolu}) gives a well-defined solution to the Gluing Equations, and so by Lemma \ref{lem:GluingEquations} we can form a common finite cover $\hat{\calX}$ of $\calX$ and $\calX'$. Thus $G$ and $G'$ are commensurable, completing the proof of Theorem \ref{thm:main}.

\bigskip
\section{Counter example for higher rank cylinders} \label{sec:counterExample}

We now consider the wider class of groups $ \mathscr{C}^\bullet$ of all subgroup separable, one-ended, finitely presented groups with JSJ decomposition consisting of virtually free vertex groups, and no QH vertex groups.
By Theorem \ref{thm:separableBalancedRelHyp} such groups are hyperbolic relative to virtually free-by-cyclic vertex groups.

We present the following pair of groups which we assert are quasi-isometric, but not commensurable.

Let $w \in \mathbb{F}_2 = \langle x, y \rangle$ be a word that induces a rigid line pattern in $\mathbb{F}_2$.
We consider the following groups:
\[
G = \mathbb{F}_2 *_{\mathbb{Z}}(\mathbb{F}_2 \times \mathbb{Z}) \; = \; \langle x,y, a,b ,z\mid w=z,\, [a,z] = [b,z] =1 \rangle,
\]
and 
\[
G' = \mathbb{F}_2 *_{\mathbb{Z}}(\mathbb{F}_3 \times \mathbb{Z}) \; = \; \langle x,y, a,b,c,z \mid w=z,\, [a,z] = [b,z] = [c,z] = 1 \rangle.
\]
We note that $\mathcal{X}(G) = \mathcal{X}(G') = -1$ since the free-by-cyclic factors contribute nothing.
These groups are torsion-free, and in the language of Guiradel and Levitt~\cite{JSJ}, the given splitting corresponds to the canonical tree of cylinders with respect to a JSJ decomposition.
We also note that these groups are virtually special.

\begin{lem} \label{lem:qiExamples}
 $G$ and $G'$ are quasi-isometric.
\end{lem}

\begin{proof}
 Let $f:\mathbb{F}_2=\langle a,b\rangle\to\mathbb{F}_3=\langle a,b,c\rangle$ be a bi-Lipschitz bijection with bi-Lipschitz constant $C\geq1$ -- this exists by \cite{Papasoglu95}.\par 
 Write an element of $G$ as $g=\alpha_1\beta_1\alpha_2\beta_2\cdots\alpha_k\beta_k$, where $\alpha_i\in\langle x,y\rangle$, $\beta_i\in\langle a,b\rangle$, and $\alpha_i,\beta_i\neq 1$, $\alpha_i\notin\langle w\rangle$ (except possibly $\alpha_1$ and $\beta_k$). Define a map $\psi:G\to G'$ by $\psi(g):=\alpha_1f(\beta_1)\alpha_2f(\beta_2)\cdots\alpha_kf(\beta_k)$ (viewing $\langle x,y\rangle$ as a common subgroup of $G$ and $G'$). Note that the $\beta_i$ are uniquely determined by $g$, and the only ambiguity in the $\alpha_i$ comes from making replacements $(\alpha_i,\alpha_{i+1})\mapsto(\alpha_iw^j,w^{-j}\alpha_{i+1})$, which does not change $\psi(g)$, thus $\psi$ is well-defined.\par 
 We claim that $\psi$ is a quasi-isometry. Take elements $g,\bar{g}\in G$ written in the above normal form, and make replacements as above so that they agree on an initial subword of maximum possible length. If the first term where they differ is an $\alpha_i$ term then we can write $g=\alpha_1\beta_1\alpha_2\beta_2\cdots\alpha_k\beta_k$ and $\bar{g}=\alpha_1\beta_1\cdots\alpha_{l-1}\beta_{l-1}\bar{\alpha}_l\bar{\beta}_l\cdots\bar{\alpha}_m\bar{\beta}_m$ with $\bar{\alpha}_l\notin\alpha_l\langle w\rangle$. Working with respect to the given generators for $G$ and $G'$, we use $d$ to denote the metrics on $G$ and $G'$ and $|\cdot|$ to denote the distance to the identity. Then for appropriate choices of the $\alpha_i$ and $\bar{\alpha}_i$ we have
 \begin{align}
 d(g,g')&=|\beta_k^{-1}\alpha_k^{-1}\cdots\beta_l^{-1}\alpha_l^{-1}\bar{\alpha}_l\bar{\beta}_l\cdots\bar{\alpha}_m\bar{\beta}_m|\nonumber\\
 &=|\beta_k^{-1}|+|\alpha_k^{-1}|+\dots +|\beta_l^{-1}|+|\alpha_l^{-1}\bar{\alpha}_l|+|\bar{\beta}_l|+\dots+|\bar{\alpha}_m|+|\bar{\beta}_m|\nonumber\\
 &=|\beta_k|+|\alpha_k|+\dots +|\beta_l|+|\alpha_l^{-1}\bar{\alpha}_l|+|\bar{\beta}_l|+\dots+|\bar{\alpha}_m|+|\bar{\beta}_m|.
 \end{align}
 On the other hand
 \begin{align}
 d(\psi(g),\psi(g'))&=|f(\beta_k)^{-1}\alpha_k^{-1}\cdots f(\beta_l)^{-1}\alpha_l^{-1}\bar{\alpha}_l f(\bar{\beta}_l)\cdots\bar{\alpha}_m f(\bar{\beta}_m)|\nonumber\\
 &\leq|f(\beta_k)^{-1}|+|\alpha_k^{-1}|+\dots +|f(\beta_l)^{-1}|+|\alpha_l^{-1}\bar{\alpha}_l|+|f(\bar{\beta}_l)|+\dots+|\bar{\alpha}_m|+|f(\bar{\beta}_m)|\nonumber\\
 &=|f(\beta_k)|+|\alpha_k|+\dots +|f(\beta_l)|+|\alpha_l^{-1}\bar{\alpha}_l|+|f(\bar{\beta}_l)|+\dots+|\bar{\alpha}_m|+|f(\bar{\beta}_m)|\nonumber\\
 &\leq C(|\beta_k|+\dots|\beta_l|+|\bar{\beta}_l|+\dots+|\bar{\beta}_m|)+|\alpha_k|+\dots+|\alpha_l^{-1}\bar{\alpha}_l|+\dots|\bar{\alpha}_m|\nonumber\\
 &\leq Cd(g,g').
 \end{align}
 A similar argument works if the first term where $g$ and $\bar{g}$ differ is a $\beta_i$ term rather than an $\alpha_i$ term. Using $f^{-1}$ we can define an inverse to $\psi$ (so in particular $\psi$ is a bijection), and by the same argument as above we get $d(g,g')\leq Cd(\psi(g),\psi(g'))$ for any $g,\bar{g}\in G$.
\end{proof}

\begin{lem}
 $G$ and $G'$ are not commensurable.
\end{lem}

\begin{proof}
 Suppose that there exist finite index subgroups $\hat{G} \leqslant G$ and $\hat{G}' \leqslant G'$, such that $\hat{G} \cong \hat{G}'$.
 
 There is an induced graph of groups decomposition of $\hat{G}$ from the decomposition of $G$.
 Let $(\hat{G}, \hat{\Gamma})$ denote that decomposition.
 There is also an induced decomposition of $\hat{G}'$, that we can denote by $(\hat{G}', \hat{\Gamma}')$, but at this point we argue from the uniqueness of these tree of cylinders decompositions (\cite{JSJ}[Corollary 7.4]) that they are the same decomposition.

 We now consider the vertex groups in $(\hat{G}, \hat{\Gamma})$ that cover the free-by-cyclic vertex group $\mathbb{F}_2\times\mathbb{Z}=\langle a,b\rangle\times\langle z\rangle$ in $G$.
 If $\hat{G}_v$ is such a vertex group, then we have an embedding $\hat{G}_v \hookrightarrow \langle a,b\rangle\times\langle z\rangle$ as a finite index subgroup. We know that $\langle z\rangle$ is the edge group incident at $\langle a,b\rangle\times\langle z\rangle$ in $G$, so the edges incident at $v$ correspond to double cosets $\hat{G}_vg\langle z\rangle$ for $g\in\langle a,b\rangle\times\langle z\rangle$.\par 
 Let $\pi:\langle a,b\rangle\times\langle z\rangle\to\langle a,b\rangle$ be the projection map, and consider the short exact sequence
 \begin{equation}
 1\to\hat{G}_v\cap\langle z\rangle\hookrightarrow \hat{G}_v\overset{\pi}{\to}\pi(\hat{G}_v)\to 1.
 \end{equation}
 $\pi(\hat{G}_v)$ is free, so there is a section $\sigma:\pi(\hat{G}_v)\to\hat{G}_v$, with image $F$ say. As $\langle z\rangle$ is central in $\langle a,b\rangle\times\langle z\rangle$, we see that $\hat{G}_v$ splits as a product $\hat{G}_v=F\times(\hat{G}_v\cap\langle z\rangle)$. Note that the rank $n(v)$ of $F\cong\pi(\hat{G}_v)$ is an invariant of $\hat{G}_v$ (it is one less than the rank of the abelianisation of $\hat{G}_v$ for example). A double coset $\hat{G}_vg\langle z\rangle$ must equal $\pi(\hat{G}_v)\pi(g)\times\langle z\rangle\leqslant\langle a,b\rangle\times\langle z\rangle$, so the number of such double cosets is equal to the index $|\langle a,b\rangle:\pi(\hat{G}_v)|$. But we know $\langle a,b\rangle$ and $\pi(\hat{G}_v)$ are free groups of rank 2 and $n(v)$ respectively, so this index must equal $n(v)-1$, and as discussed above this is the degree of the vertex $v$ in $\hat{\Gamma}$.\par 
 We can run exactly the same arguments for $\hat{G}'_v\cong\hat{G}_v$ embedded in $\mathbb{F}_3\times\mathbb{Z}=\langle a,b,c\rangle\times\langle z\rangle\leqslant G'$, and we get the same rank $n(v)$; the only difference is that we compute the degree of $v$ in $\Gamma$ as the index $|\langle a,b,c\rangle:\pi(\hat{G}_v)|$, which is the index between free groups of rank 3 and $n(v)$, and hence equals $(n(v)-1)/2$, a contradiction.
 \end{proof}
\bigskip

\bibliographystyle{alpha}
\bibliography{Ref}

\end{document}